\documentclass[11pt]{article}
\usepackage{setspace}
\usepackage{cprotect}
\usepackage{amsmath,amssymb,amsthm}
\usepackage[noend]{algorithmic}
\usepackage[ruled,vlined]{algorithm2e}
\usepackage{hyperref}
\usepackage{fullpage}
\usepackage{makeidx}
\usepackage{enumerate}
\usepackage{graphicx,float,psfrag,epsfig,caption,subcaption}
\usepackage{epstopdf}
\usepackage{color}
\usepackage{enumitem}

\usepackage[mathscr]{euscript}

\DeclareSymbolFont{rsfs}{U}{rsfs}{m}{n}
\DeclareSymbolFontAlphabet{\mathscrsfs}{rsfs}

\numberwithin{equation}{section}

\newtheoremstyle{myexample} 
    {\topsep}                    
    {\topsep}                    
    {\rm }                   
    {}                           
    {\bf }                   
    {.}                          
    {.5em}                       
    {}  

\newtheoremstyle{myremark} 
    {\topsep}                    
    {\topsep}                    
    {\rm}                        
    {}                           
    {\bf}                        
    {.}                          
    {.5em}                       
    {}  

\newtheorem{claim}{Claim}[section]
\newtheorem{lemma}[claim]{Lemma}

\newtheorem{theorem}{Theorem}
\newtheorem{proposition}[claim]{Proposition}
\newtheorem{corollary}[claim]{Corollary}

\theoremstyle{myremark}
\newtheorem{remark}{Remark}[section]

\theoremstyle{myremark}

\theoremstyle{myexample}

\def\obx{\overline{\boldsymbol x}}
\def\ox{\overline{x}}
\def\hbx{\hat{\boldsymbol x}}
\def\hlambda{\hat{\lambda}}
\def\hgamma{\hat{\gamma}}
\def\hsigma{\hat{\sigma}}
\def\hmu{\hat{\mu}}
\def\hf{{\hat{f}}}
\def\hJ{\hat{J}}

\def\bProd{{\boldsymbol \Omega}}
\def\tbProd{\tilde{\boldsymbol \Omega}}
\def\projp{{\boldsymbol P}^{\perp}}

\def\omu{\overline{\mu}}

\def\osigma{\overline{\sigma}}

\def\FDR{{\rm FDR}}
\def\hFDP{\widehat{\rm FDP}}

\def\Info{{\rm I}}
\def\npi{\mbox{\tiny\rm PI}}
\def\LAMP{\mbox{\tiny\rm L}}

\def\op{\mbox{\tiny\rm op}}
\def\sTV{\mbox{\tiny\rm TV}}
\def\sIT{\mbox{\tiny\rm IT}}
\def\sBayes{\mbox{\tiny\rm Bayes}}
\def\sALG{\mbox{\tiny\rm ALG}}
\def\dcQ{\partial{\mathcal Q}}
\def\cQ{{\mathcal Q}}
\def\cF{{\mathcal F}}
\def\cR{{\mathcal R}}
\def\cX{{\mathcal X}}
\def\cY{{\mathcal Y}}
\def\tcR{\tilde{\mathcal R}}
\def\spn{{\rm span}}
\def\sign{{\rm sign}}

\def\hB{{\hat{B}}}
\def\hC{{\hat{C}}}
\def\hpsi{{\hat{\psi}}}
\def\hxi{{\hat{\xi}}}

\def\ed{\stackrel{{\rm d}}{=}}

\def\supp{{\rm supp}}

\def\sT{{\sf T}}
\def\<{\langle}
\def\>{\rangle}
\def\P{{\mathbb P}}
\def\prob{{\mathbb P}}

\def\naturals{{\mathbb N}}
\def\E{{\mathbb E}} 
\def\Var{{\sf{Var}}}

\def\Ev{{\sf E}}

\newcommand\norm[1]{\left\lVert{#1}\right\rVert}

\def\de{{\rm d}}

\def\tlambda{\tilde{\lambda}}

\def\bsigma{{\boldsymbol \sigma}}
\def\rank{{\rm rank}}
\def\Overlap{{\rm Overlap}}
\def\cV{{\mathcal V}}
\def\hbM{\widehat{\boldsymbol M}}

\newcommand\myeqref[1]{{Eq.\,\eqref{#1}}}

\def\sTV{\mbox{\tiny\rm TV}}

\def\reals{\mathbb{R}}

\def\normal{{\sf N}}
\def\Unif{{\sf Unif}}
\def\GOE{{\sf GOE}}
\def\sB{{\sf B}}
\def\sC{{\sf C}}
\def\sb{{\sf b}}
\def\hsb{\hat{\sf b}}

\def\cE{{\mathcal{E}}}

\def\cG{{\mathcal{G}}}

\def\cS{{\mathcal{S}}}
\def\cW{{\mathcal{W}}}

\def\bg{{\boldsymbol g}}
\def\bB{{\boldsymbol B}}
\def\bM{{\boldsymbol M}}
\def\bO{{\boldsymbol O}}
\def\tbM{\widetilde{\boldsymbol M}}

\def\by{{\boldsymbol y}}
\def\bdelta{{\boldsymbol \delta}}

\def\bS{{\boldsymbol S}}
\def\bT{{\boldsymbol T}}
\def\bQ{{\boldsymbol Q}}
\def\tbQ{\tilde{\boldsymbol Q}}

\def\bG{{\boldsymbol G}}

\def\bI{{\boldsymbol I}}

\def\bPi{{\boldsymbol \Pi}}
\def\be{{\boldsymbol e}}
\def\br{{\boldsymbol r}}

\def\tbx{\tilde{\boldsymbol x}}
\def\bx{{\boldsymbol x}}
\def\bz{{\boldsymbol z}}
\def\ba{{\boldsymbol a}}
\def\bphi{{\boldsymbol \varphi}}

\def\bL{{\boldsymbol L}}
\def\bT{{\boldsymbol T}}
\def\bX{{\boldsymbol X}}
\def\bY{{\boldsymbol Y}}
\def\bZ{{\boldsymbol Z}}
\def\bP{{\boldsymbol P}}
\def\bPp{{\boldsymbol P}^{\perp}}
\def\bv{{\boldsymbol v}}
\def\obv{\tilde{\boldsymbol v}}
\def\tbv{\tilde{\boldsymbol v}}
\def\tv{\tilde{v}}
\def\tbu{\tilde{\boldsymbol u}}

\def\bu{{\boldsymbol u}}

\def\bW{{\boldsymbol W}}

\def\btau{{\boldsymbol \tau}}
\def\hbx{\hat{\boldsymbol x}}

\def\ugamma{\underline{\gamma}}
\def\ind{{\mathbb I}}

\def\bOmega{{\boldsymbol \Omega}}

\def\bSigma{{\boldsymbol \Sigma}}
\def\bU{{\boldsymbol U}}
\def\bV{{\boldsymbol V}}

\def\tbV{\tilde{\boldsymbol V}}
\def\tbX{\tilde{\boldsymbol X}}
\def\tbx{\tilde{\boldsymbol x}}
\def\tx{\tilde{x}}
\def\id{{\boldsymbol I}}
\def\bzero{{\boldsymbol 0}}
\def\bfone{{\boldsymbol 1}}

\def\towass{\stackrel{W_2}{\Rightarrow}}

\def\hbA{\widehat{\boldsymbol A}}
\def\hbW{\widehat{\boldsymbol W}}
\def\hbv{\hat{\boldsymbol v}}

\def\bR{{\boldsymbol R}}
\def\tbW{\tilde{\boldsymbol W}}
\def\tbA{\tilde{\boldsymbol A}}
\def\bA{{\boldsymbol A}}
\def\bB{{\boldsymbol B}}

\def\balpha{{\boldsymbol \alpha}}
\def\bg{{\boldsymbol g}}
\def\bbeta{{\boldsymbol \beta}}

\def\bxi{{\boldsymbol \xi}}

\def\blambda{{\boldsymbol \lambda}}
\def\SigmaAlg{{\mathfrak S}}

\def\U{{\rm U}}
\def\U1{{\rm U}(1)}

\def\Tr{{\rm Tr}}

\def\eps{{\varepsilon}}
\def\oeps{\bar{\varepsilon}}

\def\hS{\hat{S}}
\def\bPhi{{\boldsymbol \Phi}}

\def\tbphi{\tilde{\boldsymbol \varphi}}

\def\toweak{\stackrel{w}{\Rightarrow}}
\def\tow2{\stackrel{W_2}{\Rightarrow}}
\def\PL{{\rm PL}}
\def\cC{{\cal C}}

\def\mmse{{\sf mmse}}

\def\cuP{\mathscrsfs{P}}
\def\cuF{\mathscrsfs{F}}
\def\cuG{\mathscrsfs{G}}
\def\sS{{\sf S}}

\newcommand{\towa}[1]{\stackrel{W_{#1}}{\Rightarrow}}

\newcommand\abs[1]{\left\lvert{#1}\right\rvert}

\newcommand{\bfs}{{\boldsymbol s}}
\newcommand{\mbf}{\mathbf}
\newcommand{\tilg}{\tilde{g}}
\def \bDelta{{\boldsymbol \Delta}}
\def \bdel{{\boldsymbol \delta}}

\newcommand{\beq}{\begin{equation}}
\newcommand{\eeq}{\end{equation}}
\newcommand{\diag}{{\rm diag}}
\def \bLambda{{\boldsymbol \Lambda}}

\title{Estimation of Low-Rank Matrices via\\ Approximate Message Passing}

\author{Andrea  Montanari\thanks{Department of Electrical Engineering and Department
  of Statistics, Stanford University} \and
Ramji Venkataramanan\thanks{Department of Engineering, University of Cambridge. Email: \texttt{ramji.v@eng.cam.ac.uk}}}

\begin{document}

\maketitle

\begin{abstract}
Consider the problem of estimating a low-rank matrix when its entries are perturbed by Gaussian noise,
a setting that is also known as `spiked model'  or `deformed random matrix'. 
If the empirical distribution of the entries of the spikes is known, optimal estimators that exploit
this knowledge can substantially outperform simple spectral approaches. Recent work characterizes the asymptotic
accuracy of Bayes-optimal estimators in the high-dimensional limit. 
In this paper we present a practical algorithm that can achieve Bayes-optimal accuracy above the spectral threshold.
A bold conjecture from statistical physics posits that no polynomial-time algorithm achieves optimal error below
the same threshold (unless the best estimator is trivial).

Our approach uses Approximate Message Passing  (AMP) in conjunction with a spectral initialization. AMP algorithms have proved successful in a variety of statistical estimation tasks, and are amenable to exact  asymptotic analysis via state evolution. Unfortunately, state evolution is uninformative when the algorithm is initialized  near an unstable fixed point, as often happens in low-rank matrix estimation problems. We develop a new analysis of AMP that allows for spectral initializations, and builds on a 
decoupling between the outlier eigenvectors and the bulk in the spiked random matrix model. 

Our main theorem is general and applies beyond matrix estimation.
However, we use it to derive detailed predictions for the problem of estimating a rank-one matrix in noise.
Special cases of this problem are closely related---via universality arguments---to the network community  detection problem for two asymmetric
communities. For general rank-one models, we show that AMP can be used to construct confidence intervals and control false discovery rate. 

We provide illustrations of the general methodology by considering the cases of sparse low-rank matrices and of
block-constant low-rank matrices with symmetric blocks (we refer to
the latter  as to the `Gaussian Block Model').
\end{abstract}

\section{Introduction}

The `spiked model' is the simplest probabilistic model of a data matrix with a latent low-dimensional structure.
Consider, to begin with, the case of a symmetric matrix. The data are written as the sum 
of a low-rank matrix (the signal) and Gaussian component (the noise):
\begin{align}
\bA = \sum_{i=1}^k\lambda_i \bv_i\bv_i^{\sT} + \bW\, .  \label{eq:SpikedDef}
\end{align}
Here $\lambda_1\ge \lambda_2\ge \dots\ge \lambda_k$ are non-random numbers,  $\bv_i\in\reals^n$ are non-random
vectors,  and $\bW\sim\GOE(n)$ is a matrix from the Gaussian Orthogonal 
Ensemble\footnote{Recall that this means that $\bW=\bW^{\sT}$, and the entries $(W_{ij})_{ i\le j\le n}$ are independent with $(W_{ii})_{i\le n}\sim_{iid}\normal(0,2/n)$ and
$(W_{ij})_{i<j\le n}\sim_{iid}\normal(0,1/n)$.}.
The asymmetric (rectangular) version of the same model is also of interest. In this case we observe $\bA\in \reals^{n\times d}$ given by
\begin{align}
\bA = \sum_{i=1}^k\sqrt{\lambda_i} \, \bu_i\bv_i^{\sT} + \bW\, .  \label{eq:SpikedDef-Rectangular}
\end{align}
where $\bW$ is a noise matrix with entries $(W_{ij})_{i\le n,j\le d}\sim_{iid}\normal(0,1/n)$.  An important special case
assumes $\bu_i\sim\normal(0,\id_n/n)$. In this case\footnote{For the formal analysis of this model, it will be convenient to 
consider the case of deterministic vectors $\bu_i$, $\bv_i$ satisfying suitable asymptotic conditions. However, these conditions hold almost surely, e.g. 
$\bu_i\sim\normal(0,\id_n/n)$.}   the rows of $\bA$ are i.i.d. samples from a high-dimensional Gaussian $\ba_i\sim\normal(\bzero,\bSigma)$
where $\bSigma = (\sum_{i=1}^k\lambda_i\bv_i\bv_i^{\sT}+\id_d)/n$. Theoretical analysis of this spiked covariance model
has led to a number of important statistical insights \cite{JohnstoneICM,johnstone2009consistency}.

Within probability theory, the spiked model \eqref{eq:SpikedDef} is also known as `deformed GOE' or `deformed Wigner random matrix', and 
the behavior of its eigenvalues and eigenvectors has been studied in exquisite detail \cite{baik2005phase,baik2006eigenvalues,feral2007largest,capitaine2009largest,benaych2011eigenvalues,benaych2012singular,knowles2013isotropic}. The most basic phenomenon 
unveiled by this line of work is the so-called BBAP phase transition, first discovered in the physics literature
\cite{hoyle2004principal}, and named after the authors of \cite{baik2005phase}. Let
$k_*$ be the number of rank-one terms with $|\lambda_i|>1$. Then the spectrum of $\bA$ is formed by a bulk of eigenvalues in the interval
$[-2,2]$ (whose distribution follows Wigner's semicircle), plus $k_*$ outliers that are in one-to-one correspondence with the 
large rank-one terms in (\ref{eq:SpikedDef}). The eigenvectors associated to the outliers exhibit a significant correlation with 
the corresponding vectors $\bv_i$. To simplify the discussion, in the rest of this introduction we will assume that $\lambda_i\ge 0$ for all $i$.

The spiked model \eqref{eq:SpikedDef}, \eqref{eq:SpikedDef-Rectangular} and their generalizations have also been studied from a statistical perspective
\cite{johnstone2001distribution,paul2007asymptotics}. A fundamental question in this context is to estimate
the vectors $\bv_i$ from a single realization of the matrix $\bA$.  It is fair to say that this question is relatively well understood when 
the vectors $\bv_i$ are unstructured, e.g. they are a uniformly random orthonormal set (distributed according to the Haar measure).
In this case, and in the high-dimensional limit $n,d \to\infty$,  the best estimator of vector $\bv_i$ is the  $i$-th eigenvector of $\bA$.
Random matrix theory provides detailed information about its asymptotic properties.

This 
paper is concerned with the case in which the vectors $\bv_i$ are structured, e.g. they are sparse, or have bounded entries. This structure is not captured by spectral methods,
and other approaches lead to significantly better estimators. This scenario is relevant for a broad range of applications,
including sparse principal component analysis \cite{johnstone2009consistency,zou2006sparse,deshpande2014information}, non-negative principal 
component analysis \cite{lee1999learning,montanari2016non},
community detection under the stochastic block model \cite{deshpande2016asymptotic,abbe2017community,moore2017computer}, 
and so on. Understanding what are optimal ways of exploiting the
structure of signals is ---to a large extent---an open problem. 

Significant progress has been achieved recently under the assumption that the vectors $(v_{1,j},\dots, v_{k,j})$ $\in \reals^k$ (i.e., the $k$-dimensional
vectors obtained by taking the $j$-th component of the vectors $\bv_1,\dots,\bv_k$) are approximately i.i.d. (across $j \in \{1,\dots, n\}$) with some common distribution $\mu_{\bU}$ on $\reals^k$. 
This is, for instance,  the case if each $\bv_\ell$ has i.i.d. components, and distinct vectors are independent
(but mutual  independence between $\bv_1,\dots,\bv_k$ is not required).
Following heuristic derivations using statistical physics methods
(see, e.g. \cite{lesieur2017constrained}), closed form expressions
have been rigorously established for the Bayes-optimal estimation error in the limit $n\to\infty$ (with $\lambda_i$'s  fixed). We refer to 
\cite{deshpande2014information,deshpande2016asymptotic}
for special cases and to  \cite{krzakala2016mutual,barbier2016mutual,lelarge2016fundamental,miolane2017fundamental} for an increasingly general theory. 

Unfortunately, there is no general algorithm that computes the Bayes-optimal estimator and is guaranteed to run in polynomial time.
Markov Chain  Monte Carlo can have exponentially large mixing time and is difficult to analyze \cite{gamerman2006markov}. Variational methods are non-convex
and do not come with consistency guarantees \cite{blei2017variational}. Classical convex relaxations do not generally achieve the Bayes optimal error, since they incorporate limited
prior information \cite{javanmard2016phase}.

In the positive direction, approximate message passing (AMP) algorithms 
have been successfully applied to  a number of low-rank matrix estimation problems \cite{rangan2018iterative,parker2014bilinear,montanari2016non,vila2015hyperspectral,kabashima2016phase}.
In particular, AMP was proved to
achieve the Bayes optimal estimation error in  special cases of the model (\ref{eq:SpikedDef}), in the high-dimensional limit $n\to\infty$ 
\cite{deshpande2013finding,deshpande2014information}.
In fact, a bold conjecture from statistical physics suggests that the estimation error achieved by AMP is
the same that can be achieved by the optimal polynomial-time algorithm.

An important feature of AMP is that it admits an exact characterization in the limit $n\to\infty$ 
that goes under the name of \emph{state evolution} \cite{DMM09,BM-MPCS-2011,bolthausen2014iterative}. 
There is however one notable case in which the state evolution analysis of AMP falls short of its goal:
when AMP is initialized near an unstable fixed point.
This is typically the case for the problem of estimating the vectors $\bv_i$'s
in the spiked model (\ref{eq:SpikedDef}). (We refer to the next section for a discussion of this point.)

In order to overcome this problem, we propose a two-step algorithm:
\begin{enumerate}
\item We compute the principal eigenvectors $\bphi_1,\dots,\bphi_{k_*}$ of $\bA$, which correspond to
the outlier eigenvalues.
\item We run AMP with an initialization that is correlated with these eigenvectors.
\end{enumerate}
Our main result (Theorem \ref{thm:Main}) is a general asymptotically exact analysis of this type of procedure.
The analysis applies to a broad class of AMP algorithms, with initializations that are obtained by applying
separable functions to the eigenvectors $\bphi_1,\dots,\bphi_{k_*}$ (under some technical conditions).
Let us emphasize that \emph{our core technical result (state-evolution analysis) is completely general and applies beyond  low-rank matrix
estimation.}

The rest of the paper is organized as follows.
\begin{description}
\item[Section \ref{sec:Example}] applies our main results to the problem of estimating a rank-one matrix in Gaussian noise (the case
$k=1$ of the model  (\ref{eq:SpikedDef})). We compute the asymptotic empirical distribution of our estimator. In particular, this characterizes 
the asymptotics of all sufficiently regular  separable losses.

We then illustrate how this state evolution analysis can be used to design specific AMP algorithms, depending on what
prior knowledge we have about the entries of $\bv_1$. In a first case study, we only know that $\bv_1$ is sparse,
and analyze an algorithm based on iterative soft thresholding.  In the second, we assume that the empirical distribution of the entries of $\bv_1$  is known,
and develop a Bayes-AMP algorithm.
The asymptotic estimation error achieved by Bayes-AMP coincides (in certain regimes) with the Bayes-optimal error (see Corollary \ref{coro:AMP-OPT}). 
When this is not the case,  no polynomial-time  algorithm is known that outperforms our method.
\item[Section \ref{sec:Inference}] shows how AMP estimates can be used to construct confidence intervals 
and $p$-values. In particular, we prove that the resulting $p$-values are asymptotically valid on the nulls, 
which in turn can be used to establish asymptotic false discovery rate control using a Benjamini-Hochberg
procedure.
\item[Section \ref{sec:ExampleRectangular}] generalizes the  analysis of Section \ref{sec:Example} to the case of rectangular matrices. 
This allows, in particular, to derive optimal AMP algorithms for the spiked covariance model.
The theory for rectangular matrices is completely analogous to the one for symmetric ones, and indeed can be established
via a reduction to symmetric matrices.
\item[Section \ref{sec:Degenerate}] discusses a new phenomenon arising in case of degeneracies between 
the values $\lambda_1,\dots,\lambda_k$. For the sake of concreteness, we consider the case $\bA = \lambda \bA_0+\bW$,
where $\bA_0$ is a rank-$k$ matrix obtained as follows. We partition $\{1,\dots,n\}$ in $q=k+1$ groups and set 
$A_{0,ij} = k/n$ if $i,j$ belong to the same group and $A_{0,ij}=-1/n$ otherwise. 
Due to its close connections with the stochastic block model of random graphs, we refer to this as to the `Gaussian block model'.

It turns out that in such degenerate cases,  the evolution of AMP estimates does not concentrate around a deterministic 
trajectory. Nevertheless, state evolution captures the asymptotic  behavior of the algorithm in terms of a random initialization 
(whose distribution is entirely characterized) plus a deterministic evolution.
\item[Section \ref{sec:Symmetric}] presents our general result in the case of a symmetric matrix $\bA$ distributed according to 
the model  (\ref{eq:SpikedDef}).
Our theorems provide an asymptotic characterization of a general AMP algorithm 
in terms of a suitable state evolution recursion.
A completely analogous result holds for rectangular matrices. The corresponding statement is presented in the supplementary material.
\item[Section \ref{sec:ProofOutline}] provides an outline of the proofs of our main results.
Earlier state evolution results do not allow to rigorously analyze AMP unless its initialization is independent from
the data matrix $\bA$. In particular, they do not allow to analyze the spectral initialization used in our algorithm.
In order to overcome this challenge, we prove a technical lemma (Lemma \ref{lemma:Distr}) that specifies an approximate representation 
for the conditional distribution of $\bA$ given its leading outlier eigenvectors and the corresponding eigenvalues.
Namely, $\bA$ can be approximated by a sum of rank-one matrices, corresponding to the outlier eigenvectors, plus a projection of 
a new random matrix $\bA^{\mbox{\tiny\rm new}}$ independent of $\bA$. 
We leverage this explicit independence to establish state evolution for our algorithm. 
\end{description}
Complete proofs of the main results are deferred to the Appendices \ref{sec:ProofMain} and \ref{sec:ProofMainGeneral}. For the reader's convenience,  we present separate proofs for the case of rank $k=1$, and then for the general case, which is technically more involved. 
The proofs concerning the examples in Section \ref{sec:Example} and \ref{sec:ExampleRectangular} are also presented in the appendices.

As mentioned above, while several of our examples concern low-rank matrix estimation, the main 
result  in Section \ref{sec:Symmetric} is significantly more general, and is potentially relevant to
a broad range of applications in which AMP is run in conjunction with a spectral initialization.

\section{Estimation of symmetric rank-one matrices}
\label{sec:Example}

In order to illustrate our main result (to be presented in Section \ref{sec:Symmetric}), we apply it to the problem of estimating a rank-one symmetric matrix in Gaussian noise.
We will begin with a brief heuristic discussion of AMP and its application to rank-one matrix  estimation. The reader is welcome to
consult the substantial literature on AMP for further background \cite{BM-MPCS-2011,javanmard2013state,bayati2015universality,berthier2017state}. 

\subsection{Main ideas and heuristic justification}

Let $\bx_{0}=\bx_0(n)\in \reals^n$ be a sequence of signals indexed by the dimension $n$, satisfying the following conditions:
\begin{itemize}
\item[$(i)$] Their rescaled $\ell_2$-norms converge $\lim_{n\to\infty}\|\bx_0(n)\|_2/\sqrt{n}= 1$; 
\item[$(ii)$] The empirical distributions of the entries of $\bx_0(n)$ converges weakly to a probability distribution $\nu_{X_0}$  on $\reals$,
with unit second moment.
\end{itemize}

We then consider the following spiked model, for $\bW\sim \GOE(n$):
\begin{align}
\bA = \frac{\lambda}{n} \, \bx_0\bx_0^{\sT} + \bW\, .  \label{eq:SpikedDefSpecial}
\end{align}
Given one realization of the matrix $\bA$, we would like to estimate the signal $\bx_0$.
Note that this matrix is of the form (\ref{eq:SpikedDef}) with $k=1$, $\lambda_1 = \lambda\|\bx_0(n)\|_2^2/n\to \lambda$
and $\bv_1 = \bx_0(n)/\|\bx_0(n)\|_2$. 

In order to discuss informally the main ideas in AMP, assume for a moment to be given an additional noisy observation of $\bx_0$, call it $\by\in\reals^n$, which is independent of 
$\bA$ (i.e., independent of $\bW$, since $\bx_0$ is deterministic). More specifically, assume $\by \sim \normal(\mu_0\bx_0,\sigma_0^2\id_n)$.
How can we denoise this observation, and incorporate the quadratic observation $\bA$ in  \eqref{eq:SpikedDefSpecial}? 

A first idea would be to denoise $\by$, using an entry-wise scalar denoiser $f_0:\reals\to\reals$. We denote the vector obtained by applying $f_0$ component-wise
by $f_0(\by)$. Of course, the choice of $f_0$ depends on our knowledge of $\bx_0$. For instance if we know that $\bx_0$ is sparse, then we could apply
component-wise soft thresholding:
\begin{align}
f_0(y_i) = \eta\big(y_i;  \tau\big)\, ,
\end{align}
where $\eta(x;\tau) = \sign(x)(|x|-\tau)_+$, and $\tau$ is a suitable threshold level. 
Classical theory guarantees the accuracy of such a denoiser \cite{DJ94a,DJ98}.

However, $f_0(\by)$ does not exploit  the observation $\bA$ in any way. We could try to improve this estimate by multiplying $f_0(\by)$ by $\bA$:
\begin{align}
\bx^1 = \bA f_0(\by) = \frac{\lambda}{n}\<\bx_0,f_0(\by)\>\, \bx_0 + \bW f_0(\by)\, . \label{eq:FirstIteration}
\end{align}
It is not hard to see that the second term is a centered Gaussian vector whose entries have variance close to
$\|f_0(\by)\|^2/n \to \sigma_{1}^2\equiv \E \{ f_0(\mu_0 X_0 + \sigma_0 G )^2 \}$, while the first term is essentially deterministic by the law of large numbers.
We thus obtain that $\bx^1$ is approximately $\normal(\mu_1\bx_0, \sigma^2_1\id_n)$, where
\begin{align}
\mu_{1}  = \lambda \E\{  X_0 f_0(\mu_0 X_0+\sigma_0 G) \}\, , \;\;\;\;\;
 \sigma_{1}^2& = \E \{ f_0(\mu_0 X_0 + \sigma_0 G )^2 \}\, .\label{eq:First-StateEvolution}
\end{align}
Here expectation is taken with respect to $X_0\sim\nu_{X_0}$ independent of $G\sim\normal(0,1)$.
This analysis also suggests how to design the function $f_0$: ideally, it should maximize the signal-to-noise ratio (SNR) $\mu^2_1/\sigma^2_1$.
Of course, the precise choice of $f_0$ depends on our prior knowledge of $\bx_0$. For instance, if we know the law $\nu_{X_0}$, we can maximize this ratio by taking
$f_0(y) = \E\{X_0|\mu_0 X_0 + \sigma_0 G =y\}$.

At this point it would be tempting to iterate the above procedure, and consider the non-linear power iteration
\begin{align}
\bx_{\npi}^{t+1} = \bA\,  f_t(\bx^t_{\npi})\, ,\label{eq:PowerIteration}
\end{align}
for a certain sequence of functions $f_t:\reals\to\reals$. (As above, $f_t(\bx^t_{\npi})$ is the vector obtained by applying $f_t$ component-wise to $\bx_{\npi}^t$,
and we will use superscripts to indicate the iteration number.) 
While this approach has been studied in the literature \cite{journee2010generalized,yuan2013truncated,chen2018projected}, sharp results could 
only be established in a high SNR regime 
where $\lambda=\lambda(n)\to\infty$ at a sufficiently fast rate. Indeed, analyzing the recursion \eqref{eq:PowerIteration} is difficult because 
$\bA$ is correlated with $\bx_{\npi}^t$ (unlike in Eq.~\eqref{eq:FirstIteration}), and hence the simple calculation that yields Eq.~\eqref{eq:First-StateEvolution} is no longer permitted. 
This problem is compounded by the fact that we do not have an additional observation $\by$ independent of $\bA$, and instead we
plan to use a spectral initialization $\bx^0\propto \bphi_1$ that depends on the top eigenvector $\bphi_1$ of $\bA$. As a consequence, 
even the first step of the analysis (given in Eq.~\eqref{eq:First-StateEvolution}) is no longer obvious.

Let us emphasize that these difficulties are not a limitation of the proof technique. For $t>1$, the iterates \eqref{eq:PowerIteration}
are no longer Gaussian or centered around $\mu_t\bx_0$, for some scaling factor $\mu_t$. This can be easily verified by considering, for instance, 
the function $f_t(x) = x^2$ (we refer to  \cite{bayati2015universality} which carries out the calculation for such an example).

AMP solves the correlation problem in nonlinear power iteration by modifying Eq.~\eqref{eq:PowerIteration}: namely, we subtract from  
from $\bA  f_t(\bx^t_{\LAMP})$ the part that is correlated to the past iterates. Let $\SigmaAlg_t\equiv\sigma(\{\bx_{\LAMP}^0,\bx_{\LAMP}^2,\dots,\bx_{\LAMP}^t\})$ be the $\sigma$-algebra generated by iterates up to time $t$.  The correction that compensates for correlations is most conveniently explained by using the following Long AMP recursion,
introduced in \cite{berthier2017state}:
\begin{align}
\bx_{\LAMP}^{t+1} &= \bA\,  f_t(\bx^t_{\LAMP})-\E\{\bW  f_t(\bx^t_{\LAMP})|\SigmaAlg_t\} + \overline{\alpha}_{t}\bx_0+\sum_{s=0}^t\alpha_{t,s}  \bx^s_{\LAMP}\label{eq:LongAMP}\\
&= \sum_{s=0}^t\alpha_{t,s}  \bx^s_{\LAMP} +  \left( \overline{\alpha}_{t}  + \tfrac{\lambda}{n}\<\bx_0,f_t(\bx^t_{\LAMP})\> \right)\bx_0 + \bW\,  f_t(\bx^t_{\LAMP})-\E\{\bW  f_t(\bx^t_{\LAMP})|\SigmaAlg_t\}\, .
\end{align}
where $(\overline{\alpha}_{t})_{0\le t}$, $(\alpha_{t,s})_{0\le s\le t}$ are suitable sequences of deterministic numbers. 
In words, the new vector $\bx_{\LAMP}^{t+1}$ is a linear combination of iterates up to time $t$, plus 
a term $\bx_0 ( \overline{\alpha}_{t} +  \lambda\<\bx_0,f_t(\bx^t_{\LAMP})\>/n)$ that is essentially deterministic, plus a random term  $(\bW\,  f_t(\bx^t_{\LAMP})-\E\{\bW  f_t(\bx^t_{\LAMP})|\SigmaAlg_t\})$ that is uncorrelated with the past.
If the past iterates $(\bx_{\LAMP}^s)_{0\le s\le t}$ are jointly Gaussian, then the first two components (linear and deterministic)  are also jointly Gaussian with 
$(\bx_{\LAMP}^s)_{0\le s\le t}$.  Since the third (random) term is uncorrelated with the past iterates,  it can be shown by induction that the sequence $(\bx_{\LAMP}^t)_{0\le t\le T}$ is approximately Gaussian as $n\to\infty$, for any fixed $t$ (in the sense of finite dimensional marginals), and centered around $\bx_0$, see \cite{berthier2017state}.

At first sight, this might appear as a mathematical trick, with no practical implications. Indeed 
Eq.~\eqref{eq:LongAMP} does not provide an algorithm. We are  explicitly using the true signal $\bx_0$ which we are supposed to estimate,
and the expectation $\E\{\bW  f_t(\bx^t_{\LAMP})|\SigmaAlg_t\}$ is, at best, hard to compute. However it turns out that (for a certain choice of the numbers
$(\overline{\alpha}_{t})_{0\le t}$, $(\alpha_{t,s})_{0\le s\le t}$), the term subtracted from $\bA\,  f_t(\bx^t_{\LAMP})$ in Eq.~\eqref{eq:LongAMP}
can be approximated by $\sb_t f_{t-1}(\bx_{\LAMP}^{t-1})$ with a coefficient $\sb_t$ that can be computed easily. We will not try to justify this approximation here
(see, for instance, \cite{berthier2017state}). We will instead use the resulting algorithm (given below in Eq. \eqref{eq:AMPspecial0}) as the starting point of our analysis.

\subsection{General analysis}
\label{sec:GeneralRank1}

Motivated by the discussion in the previous section, we consider the following general algorithm for rank-one matrix estimation in the model \eqref{eq:SpikedDefSpecial}.
In order to estimate $\bx_0$, we compute the principal eigenvector of $\bA$, to be denoted by $\bphi_1$,
and apply the following iteration, with initialization $\bx^0= \sqrt{n}\bphi_1$:
\begin{align}
\bx^{t+1} &= \bA\, f_t(\bx^t) -\sb_t f_{t-1}(\bx^{t-1})\, ,\;\;\;\;\;\sb_t = \frac{1}{n}\sum_{i=1}^n f_{t}'(x^t_i) \, .\label{eq:AMPspecial0}
\end{align}
Here $f_t(\bx) = (f_t(x_1),\dots,f_t(x_n))^{\sT}$ is a separable function for each $t$.
As mentioned above, we can think of this iteration as an approximation of Eq.~\eqref{eq:LongAMP} where all the terms except the first one
have been estimated by $-\sb_t f_{t-1}(\bx^{t-1})$. The fact that this is an accurate estimate for large $n$ 
is far from obvious, but can be established by induction over $t$ \cite{berthier2017state}.

Note that  $\bx_0$ can be estimated from the data $\bA$  only up to an overall sign
(since $\bx_0$ and $-\bx_0$ give rise to the same matrix $\bA$ as per Eq.~(\ref{eq:SpikedDefSpecial})).
In order to resolve this ambiguity, we will assume, without loss of generality, that $\<\bx_0,\bphi_1\>\ge 0$. 
\begin{theorem}\label{thm:Rank1}
Consider the $k=1$ spiked  matrix model  of \myeqref{eq:SpikedDefSpecial},  with $\bx_0(n)\in\reals^n$ a sequence of vectors satisfying
assumptions $(i)$, $(ii)$ above, and $\lambda>1$ . 
Consider the AMP iteration in Eq.~\eqref{eq:AMPspecial0}  with initialization  $\bx^0 = \sqrt{n} \, \bphi_1$ (where, without loss of generality $\<\bx_0,\bphi_1\> \ge 0$). Assume $f_t:\reals\to\reals$
to be Lipschitz continuous for each $t\in \naturals$. 

Let  $(\mu_t, \sigma_t)_{t \geq 0}$ be defined via the recursion
\begin{align}
\mu_{t+1} & = \lambda \E[ X_0 f_t(\mu_t X_0+\sigma_t G)]\, ,  \label{eq:MutRecX0}\\
 \sigma_{t+1}^2& = \E [f_t(\mu_t X_0 + \sigma_t G )^2]\, , \label{eq:SigtRecX0}
\end{align}
where $X_0\sim \nu_{X_0}$ and $G \sim \normal (0,1)$ are independent, and 
the initial condition is $\mu_0= \sqrt{1-\lambda^{-2}}$, $\sigma_0 = 1/\lambda$. 

Then, for any function $\psi:\reals\times\reals\to\reals$  with $|\psi(\bx)-\psi(\by)|\le C(1+\|\bx\|_2+\|\by\|_2)\|\bx - \by\|_2$ for a universal constant $C>0$, the following holds almost surely for $t \geq 0$:
\begin{align}
\lim_{n \to \infty} \frac{1}{n} \sum_{i=1}^n \psi (x_{0,i},x^t_i) = \E \left\{ \psi( X_0, \mu_t X_0  +\sigma_t G) \right\} \ . \label{eq:rank1_SE}
\end{align}
\end{theorem}
The proof of this theorem is presented in Appendix \ref{sec:ProofMain}.

One peculiarity of our approach is that we do not commit to a specific choice of the nonlinearities $f_t$, and instead
develop a sharp asymptotic characterization for any---sufficiently regular---nonlinearity. A poor choice of the functions $f_t$ might result
in large estimation error, and yet Theorem \ref{thm:Rank1} will continue to hold. 

On the other hand, the state evolution characterization can be used to design optimal nonlinearities in a principled way. 
Given Eqs.~\eqref{eq:MutRecX0} and \eqref{eq:SigtRecX0}, the general principle is quite transparent. 
The optimal nonlinearity is defined in terms of a scalar denoising problem. For $X_0\sim \nu_{X_0}$ and $G\sim\normal(0,1)$ independent, consider the problem of
estimating $X_0$ from the noisy observation $Y = \mu_t\, X_0+\sigma_tG$. 
At step $t$, $f_t$ should be
constructed as to maximize the ratio  $\E[ X_0 f_t(\mu_t X_0+\sigma_t G)]/\E [f_t(\mu_t X_0 + \sigma_t G )^2]^{1/2}$. Two specific instantiations of this principle 
are given in Sections \ref{sec:SparseSpike} and \ref{sec:BayesAMP}.

\begin{remark}
The state evolution recursion of Eqs.~\eqref{eq:MutRecX0}, \eqref{eq:SigtRecX0} in Theorem \ref{thm:Rank1} was already derived by Fletcher and Rangan 
in \cite{rangan2018iterative}. However, as explained in \cite[Section 5.3]{rangan2018iterative},  their results only apply to cases in which AMP can be initialized in a way
that: $(i)$~has positive correlation with the spike $\bx_0$ (and this correlation does not vanish as $n\to\infty$); $(ii)$~is independent of $\bA$.

 Theorem \ref{thm:Rank1} analyzes an algorithm which does not require such an initialization, and hence applies more broadly.
\end{remark}

\subsection{The case of a sparse spike}
\label{sec:SparseSpike}

In some applications we might know that the spike $\bx_0$ is sparse. We consider  a simple model in which $\bx_0$  is known to  have 
at most $n\eps$ nonzero entries for some $\eps\in (0,1)$.

Because of its importance, the use of nonlinear power iteration methods for this problem has been studied by several authors in the past
 \cite{journee2010generalized,yuan2013truncated,ma2013sparse}. However, none of these works obtains precise asymptotics in the moderate SNR regime 
(i.e., for $\lambda$, $\eps$ of order one). In contrast, sharp results can be obtained by applying Theorem \ref{thm:Rank1}.
Here we will limit ourselves to taking the first steps, deferring a more complete analysis to future work.
We focus  on the case of symmetric matrices for simplicity, cf. Eq.~\eqref{eq:SpikedDefSpecial}, but a generalization to rectangular matrices is straightforward
along the lines of Section \ref{sec:ExampleRectangular}.

The sparsity assumption implies that the random variable $X_0$ entering the state evolution recursion in \myeqref{eq:MutRecX0} should satisfy
$\nu_{X_0}(\{0\})\ge 1-\eps$.  Classical theory for the sparse sequence model \cite{DJ94a,DJ98} suggests taking $f_t$ to be the soft thresholding denoiser
$f_t(x) = \eta(x;\tau_t)$, for $(\tau_t)_{t\ge 0}$ a well-chosen sequence of thresholds. The resulting algorithm reads
\begin{align}
\bx^{t+1} &= \bA\, \hbx^t-\sb_t \hbx^{t-1}\, ,\;\;\;\; \hbx^t  = \eta(\bx^t;\tau_t)\, ,\label{eq:SparseAlg}\\
&\;\;\;\;\;\sb_t = \frac{1}{n}\|\hbx^t\|_0\, , \nonumber
\end{align}
where $\|\bv\|_0$ is the number of non-zero entries of vector $\bv$. The initialization is, as before $\bx^0 = \sqrt{n}\bphi_1$.
The algorithm alternates soft thresholding, to produce sparse estimates, and power iteration, with the crucial correction term $-\sb_t \hbx^{t-1}$.

Theorem  \ref{thm:Rank1} can be directly applied to characterize the performance of this algorithm for any fixed distribution
$\nu_{X_0}$ of the entries of $\bx_0$. For instance, we obtain the following exact prediction for the asymptotic correlation between estimates $\hbx^t$ and the
signal $\bx_0$:
\begin{align}
\lim_{n\to\infty}\frac{|\<\hbx^t(\bA),\bx_0\>|}{\|\hbx^t(\bA)\|_2\|\bx_0\|_2} = \frac{\mu_{t+1}}{\lambda\sigma_{t+1}}\, .
\end{align}
For a given distribution $\nu_{X_0}$, it is easy to compute $\mu_t,\sigma_t$ using Eq.~\eqref{eq:MutRecX0} with $f_t(\bx^t) = \eta(\bx^t;  \tau_t)$.

  We can also use Theorem \ref{thm:Rank1}  to characterize the minimax behavior over $n\eps$-sparse vectors.
We sketch the argument next: similar arguments were developed  in \cite{DMM09,donoho2013accurate} in the context of compressed sensing.
The basic idea is to lower bound the singnal-to-noise ratio (SNR) $\mu^2_{t+1}/\sigma^2_{t+1}$ iteratively as a function of the SNR at the previous iteration,
over the set of probability distributions $\cF_{\eps} = \{\nu_{X_0}: \; \nu_{X_0}(\{0\})\ge 1-\eps, \,  \int x^2\nu_{X_0}(\de x) = 1\}$. 
As shown in Appendix \ref{sec:Reduction},  it is sufficient to consider the extremal points of the set $\cF_{\eps}$,
which are given  by the three-points priors 
\beq
\pi_{p,a_1,a_2} \equiv (1-\eps)\delta_0+\eps p\delta_{a_1}+\eps(1-p)\delta_{a_2},
\quad pa_1^2+(1-p)a_2^2 =1,  \quad p\in [0,1]. 
\eeq
We then define the following SNR maps
\begin{align}
S_* (\gamma,\theta;\nu_{X_0})& \equiv \frac{[\E\{X_0\eta(\sqrt{\gamma}X_0+G;\theta)\}]^2}{\E\{\eta(\sqrt{\gamma}X_0+G;\theta)^2\}}\, , \label{eq:Sstar_def}\\
S(\gamma;\theta)&\equiv \inf\Big\{S_* (\gamma,\theta;\pi_{p,a_1,a_2}) :\;\; pa_1^2+(1-p)a_2^2 =1, \,  p\in[0,1]\Big\}\, . \label{eq:Sdef}
\end{align}
The interpretation of these quantities is as follows: $\gamma\mapsto S_* (\gamma,\theta;\nu_{X_0})$ describes the evolution of the
signal-to-noise ratio after one step of AMP, when the signal distribution is $\nu_{X_0}$; the map
$\gamma\mapsto S(\gamma;\theta)$ is  the same evolution, for the least favorable prior, which can be taken of the form $\pi_{p,a_1,a_2}$.

Notice that the function $S_* (\gamma,\theta;\pi_{p,a_1,a_2})$ can be evaluated by performing a small number (six, to be precise) of Gaussian integrals.
The function $S$ is defined by a two-dimensional optimization problem, which can be computed numerically quite efficiently.

We  define the sequences $(\ugamma_t)_{t\ge 0}$, $(\theta_t)_{t\ge 0}$ by setting $\ugamma_0 = \lambda^2-1$, and then recursively
\begin{align}
\ugamma_{t+1} = \lambda^2S(\ugamma_t;\theta_t)\, ,\;\;\; \theta_t = \arg\max_{\theta\in [0,\infty]} S(\ugamma_t;\theta)\, .
\label{eq:ugam_rec}
\end{align}

The next proposition provides the desired lower bound for the signal-to-noise ratio over the class of sparse vectors.
\begin{proposition}
Assume the setting of Theorem \ref{thm:Rank1}, and furthermore $\|\bx_0(n)\|_0\le n\eps$. 
Let $(\hbx^t = \hbx^t(\bA) )_{t\ge 0}$ be the sequence of estimates produced by the AMP
iteration \myeqref{eq:SparseAlg} with initialization $\bx^0 = \sqrt{n} \bphi_1$, and  thresholds $\tau_t=\theta_t\hat{\sigma}_t$ where
 $\hsigma_t$ is a estimator of $\sigma_t$ from data $\bx^0,\dots,\bx^t$ such that $ \hsigma_t \stackrel{\text{a.s.}}{\to} \sigma_t$. (For instance, take $\hsigma_t^2  \equiv \big\|f_{t-1}(\bx^{t-1})\big\|_2^2/n$  for $t \geq 1$. For $t=0$, take $\hsigma_0^2 \equiv 1/\hat{\lambda}$, where $\hat{\lambda}$ is given in \myeqref{eq:lambda_hat}.)

Then for any fixed $t\ge 0$ we have, almost surely,
\begin{align}
\lim_{n\to\infty} \frac{|\<\hbx^t(\bA),\bx_0\>|}{\|\hbx^t(\bA)\|_2\|\bx_0\|_2} = \frac{\mu_{t+1}}{\lambda\sigma_{t+1}}\ge \frac{\sqrt{\ugamma_{t+1}}}{\lambda}\, .
\label{eq:sparse_spike_lb}
\end{align}
Here, $(\mu_{t+1}, \sigma_{t+1})$ are recursively defined as follows, starting from $\mu_0=\sqrt{1-\lambda^{-2}}$ and $\sigma_0^2= \lambda^{-2}$:
\begin{align}
\mu_{t+1} & = \lambda \E\{ X_0 \eta(\mu_t X_0+\sigma_t G; \, \theta_t \sigma_t)\} \, ,  \qquad 
 \sigma_{t+1}^2 = \E \{ \eta(\mu_t X_0 + \sigma_t G; \,  \theta_t \sigma_t )^2\} \,. \label{eq:SE_sparse_spike}
\end{align}
\label{prop:sparse_spike}
\end{proposition}
%
%
The proof of Proposition \ref{prop:sparse_spike} is given in Appendix \ref{app:sparse_spike}. 
The proposition reduces the analysis of algorithm \eqref{eq:SparseAlg} to the study of a one-dimensional recursion $\ugamma_{t+1} = \lambda^2 S(\ugamma_t;\theta_t)$, which is much simpler. We defer this analysis to future work. We emphasize that the AMP algorithm in \myeqref{eq:SparseAlg} with thresholds $\tau_t=\theta_t\hat{\sigma}_t$ does not require knowledge of either the sparsity  level $\eps$ or the SNR parameter $\lambda$---these quantities are only required to compute the sequence of lower bounds $(\ugamma_t)_{t\ge 0}$.

\subsection{Bayes-optimal estimation}
\label{sec:BayesAMP}

As a second application of  Theorem \ref{thm:Rank1}, we consider the case in which 
the asymptotic empirical distribution $\nu_{X_0}$ of the entries of $\bx_0$ is known. 
This case is of special interest because  it provides a lower bound on the error achieved by any AMP algorithm.

To simplify some of the formulas below, we assume here a slightly different normalization for the initialization, but otherwise we use the same algorithm as in the general case,
namely
\begin{align}
\bx^0 & = \ \sqrt{n\lambda^2 (\lambda^2-1)}\, \bphi_1\, ,\label{eq:AMP_Bayes_In}\\
\bx^{t+1} &= \bA\, f_t(\bx^t) -\sb_t f_{t-1}(\bx^{t-1})\, ,\;\;\;\;\;\sb_t = \frac{1}{n}\sum_{i=1}^n f_{t}'(x^t_i) \, .\label{eq:AMPspecial000}
\end{align}

In order to define the optimal nonlinearity, consider again the scalar denoising problem of estimating $X_0$ from the noisy observation $Y = \sqrt{\gamma}\, X_0+G$
(note that $X_0, G\in\reals$ are scalar random variables).
The minimum mean square error is
\begin{align}
\mmse(\gamma) = \E\big\{\big[X_0-\E(X_0|\sqrt{\gamma}\, X_0+G)\big]^2\big\}\, .
\end{align}
With these notations, we can introduce the state evolution recursion 
\begin{align}
\gamma_0 & = \lambda^2-1\, ,\label{eq:SEspecial_1}\\
\gamma_{t+1}&= \lambda^2\big\{1-\mmse(\gamma_t)\big\}\, . \label{eq:SEspecial_2}
\end{align}
These describe the evolution of the effective signal-to-noise ratio along the algorithm execution.

The optimal non-linearity $f_t(\,\cdot\,)$ after $t$ iterations is the minimum mean square
error denoiser for signal-to-noise ratio $\gamma_t$:
\begin{align}
f_t(y) &\equiv  \lambda\, F( y;\gamma_t) \,,\label{eq:OptFspecial}\\
F(y;\gamma) & \equiv \E\{ X_0 \mid \gamma\, X_0+\sqrt{\gamma} \,G=y \} .\label{eq:Fdef}
\end{align}
After $t$ iterations, we produce an estimate of $\bx_0$ by computing $\hbx^t(\bA)\equiv f_t(\bx^t)/\lambda = F(\bx^t; \gamma_t)$. 
We will refer to this choice as to Bayes AMP. 
\begin{remark}
Implementing the Bayes-AMP algorithm requires to approximate the function $F(y;\gamma)$
of Eq.~\eqref{eq:Fdef}. This amounts to a one-dimensional integral and can be done very accurately by standard quadrature methods: a simple approach
that works well in practice is to replace the measure $\nu_{X_0}$ by a combination of finitely many point masses. Analogously, the function $\mmse(\gamma)$
(which is needed to compute the sequence $\gamma_t$), can be computed by the same method\footnote{AMP noes not  require high
accuracy in the approximations of the nonlinear functions $f_t$. As shown several times in the appendices (see,
e.g., Appendix \eqref{sec:ProofMain}) the algorithm is stable with respect to perturbations of $f_t$.}.
\end{remark}

We are now in position to state the outcome of our analysis for Bayes AMP, whose proof is deferred to Appendix \ref{app:ProofRank1}.
\begin{theorem}
\label{thm:Special}
Consider the spiked matrix model (\ref{eq:SpikedDefSpecial}), with $\bx_0(n)\in\reals^n$ a sequence of vectors satisfying
assumptions $(i)$, $(ii)$ above, and $\lambda>1$. Let $(\bx^t)_{t\ge 0}$ be the sequence of iterates generated by the Bayes AMP algorithm
defined in Eqs.~(\ref{eq:AMPspecial0}), with initialization (\ref{eq:AMP_Bayes_In}), and optimal choice of the nonlinearity defined by
Eq.~(\ref{eq:OptFspecial}). Assume $F(\,\cdot\, ;\gamma):\reals\to\reals$ to be Lipschitz continuous for any $\gamma\in (0,\lambda^2]$.
Finally, define state evolution by Eqs. (\ref{eq:SEspecial_1}), (\ref{eq:SEspecial_2}).

Then, for any function $\psi:\reals\times\reals\to\reals$  with $|\psi(\bx)-\psi(\by)|\le C(1+\|\bx\|_2+\|\by\|_2)\|\bx - \by\|_2$ for a universal constant $C>0$, the following holds almost surely for $t \geq 0$:
\begin{align}
\lim_{n\to\infty}\frac{1}{n}\sum_{i=1}^n\psi(x_{0,i},x_i^t) = \E\big\{\psi\big(X_0,\gamma_t X_0+\gamma_t^{1/2} Z\big)\big\}\, ,  
\label{eq:special_AMP_conv}
\end{align}
where expectation is taken with respect to $X_0\sim \nu_{X_0}$ and $Z\sim\normal(0,1)$ mutually independent, and we assumed without loss
of generality that $\<\bphi_1,\bx_0\>\ge 0$.

In particular, let $\gamma_{\sALG}(\lambda)$ denote the smallest strictly positive solution of the  fixed point equation
$\gamma = \lambda^2[1-\mmse(\gamma)]$. Then the AMP estimate $\hbx^t(\bA)= f_t(\bx^t)/\lambda$ achieves 
\begin{align} 
\lim_{t\to\infty}\lim_{n\to\infty}\frac{|\<\hbx^t(\bA),\bx_0\>|}{\|\hbx^t(\bA)\|_2\|\bx_0\|_2} &= \frac{\sqrt{\gamma_{\sALG}(\lambda)}}{\lambda}\, ,\\
\lim_{t\to\infty}\lim_{n\to\infty}\frac{1}{n}\min_{s\in \{+1,-1\}} \| s \hbx^t(\bA)-\bx_0\|_2^2&= 1- \frac{\gamma_{\sALG}(\lambda)}{\lambda^2}\, .
\end{align}
Finally, the algorithm has total complexity $O(n^2 \log n)$.
\end{theorem}
\begin{remark}\label{rem:LipExpectation}
The assumption on $F(\,\cdot\,;\gamma):\reals\to\reals$ being Lipschitz continuous is required in order to apply our general theory. 
Note that this is implied by either of the following: $(i)$ $\supp(\nu_{X_0})\in [-M,M]$ for some constant $M$; $(ii)$ $\nu_{X_0}$ has log-concave density.
\end{remark}

It is interesting to compare the above result with the Bayes optimal estimation accuracy. The following statement is a  consequence of
the results of \cite{lelarge2016fundamental} (see Appendix \ref{app:BayesOpt}).
\begin{proposition} \label{thm:BayesOpt}
Consider the spiked matrix model (\ref{eq:SpikedDefSpecial}), with $\bx_0(n)\in\reals^n$ a vector with i.i.d. entries with distribution 
$\nu_{X_0}$ with bounded support and $\int x^2 \nu_{X_0}(\de x) = 1$. 
Then there exists a countable set $D\subseteq \reals_{\ge 0}$ such that, for $\lambda\in\reals\setminus D$, the Bayes-optimal accuracy in the rank-one estimation problem is given by
\begin{align}
\lim_{n\to\infty} \, \sup_{\hbx(\,\cdot\,)} \E\left\{\frac{\<\hbx(\bA),\bx_0\>^2}{\|\hbx(\bA)\|_2^2\|\bx_0\|^2_2} \right\} =
\frac{\gamma_{\sBayes}(\lambda)}{\lambda^2}\, ,\label{eq:BayesOpt}
\end{align}
where the supremum is over (possibly randomized) estimators, i.e. measurable functions $\hbx:\reals^{n\times n}\times [0,1]\to \reals^n$, where $[0,1]$ is endowed with the uniform measure.
Here $\gamma_{\sBayes}(\lambda)$ is the fixed point of the recursion (\ref{eq:SEspecial_2}) that maximizes the following free energy functional
\begin{align}
\Psi(\gamma,\lambda) = \frac{\lambda^2}{4}+\frac{\gamma^2}{4\lambda}-\frac{\gamma}{2}+\Info(\gamma)\, ,\label{eq:FreeEnergy}
\end{align}
where $\Info(\gamma) = \E\log\frac{\de p_{Y|X_0}}{\de p_{Y}}(Y,X_0)$ is the mutual information for the scalar channel
$Y = \sqrt{\gamma}\, X_0+G$, with $X_0\sim \nu_{X_0}$ and $G\sim\normal(0,1)$ mutually independent. 
\end{proposition}
Together with this proposition, Theorem \ref{thm:Special} precisely characterizes the gap between Bayes-optimal estimation and message passing algorithms for
rank-one matrix estimation. Simple calculus (together with the relation $\Info'(\gamma)=\mmse(\gamma)/2$ \cite{Guo05mutualinformation}) 
implies that the fixed point of the recursion (\ref{eq:SEspecial_2}) coincide with the stationary points of $\gamma\mapsto\Psi(\gamma,\lambda)$.
We therefore have the following characterization of the Bayes optimality of Bayes-AMP.
\begin{corollary}\label{coro:AMP-OPT}
Under the setting of Theorem \ref{thm:Special} (in particular, $\lambda>1$), let the function $\Psi(\gamma,\lambda)$ be defined as in Eq.~\eqref{eq:FreeEnergy}.
Then Bayes-AMP  asymptotically achieves the Bayes-optimal error (and $\gamma_{\sALG}(\lambda)=\gamma_{\sBayes}(\lambda)$)
 if and only if the global maximum of  $\gamma\mapsto \Psi(\gamma,\lambda)$ 
over $(0,\infty)$ is also the first stationary point of the same function (as $\gamma$ grows).
\end{corollary}

As illustrated in  Section \ref{sec:TwoPoints}, this condition holds for some cases of interest,
and hence message passing is asymptotically optimal for these cases.

\begin{remark}
In some applications, it is possible to construct an initialization $\bx^0$ that is positively correlated with the signal $\bx_0$
and independent of $\bA$. If this is possible, then the spectral initialization is not required and Theorem \ref{thm:Special}
follows immediately from \cite{BM-MPCS-2011}. For instance, if  $\nu_{X_0}$ has  positive mean, then it is sufficient to
initialize $\bx^0 = \bfone$. This principle was exploited in \cite{deshpande2013finding,deshpande2014information,montanari2016non}.

However such a positively correlated initialization is not available in general: the spectral initialization analyzed here aims at overcoming this problem.
\end{remark}

\begin{remark}\label{rmk:OptBayes}
No polynomial-time algorithm is known that achieves estimation accuracy superior to the one guaranteed by 
Theorem \ref{thm:Special}. In particular, it follows from  the optimality of posterior mean with respect to square loss and the monotonicity of the function 
$\gamma\mapsto \lambda^2\{1-\mmse(\gamma)\}$ that Bayes AMP is optimal among  AMP algorithms.
That is, for any other sequence of nonlinearities $f_t(\,\cdot\,)$, we have 
\begin{align} 
\lim_{n\to\infty}\frac{\big|\<f_t(\bx^t),\bx_0\>\big|}{\|f_t(\bx^t)\|_2\|\bx_0\|_2}&= \frac{\mu_{t+1}}{\lambda \sigma_{t+1}} \le \frac{\sqrt{\gamma_{t+1}}}{\lambda}\, .
\end{align}
As further examples, \cite{javanmard2016phase} analyzes a semi-definite programming (SDP) algorithm for the
special case of a two-points symmetric mixture $\nu_{X_0} = (1/2)\delta_{+1}+(1/2)\delta_{-1}$. Theorem \ref{thm:Special}
implies that,  in this case, message passing is Bayes optimal (since $\gamma_{\sALG}=\gamma_{\sBayes}$ follows from \cite{deshpande2016asymptotic}).
In contrast, numerical simulations and non-rigorous calculations using the cavity method from statistical physics (see \cite{javanmard2016phase}) 
suggest that SDP is sub-optimal.
\end{remark}

\begin{remark}
A result analogous to Theorem \ref{thm:Special} for the symmetric two-points distribution $\nu_{X_0} = (1/2)\delta_{+1}+(1/2)\delta_{-1}$
is proved in \cite[Theorem 3]{mossel2016density} in the context of the stochastic block model of random graphs. Note, however,  that the
approach of \cite{mossel2016density} requires the graph to have average degree $d\to\infty$, $d=O(\log n)$. 
\end{remark}

\subsection{An example: Two-points distributions}
\label{sec:TwoPoints}

Theorem \ref{thm:Special} is already interesting in very simple cases. 
Consider the two-points mixture
\begin{align}
\nu_{X_0} &= \eps\, \delta_{a_+} +(1-\eps)\delta_{-a_-}\, ,\label{eq:EmpiricalSpecial}\\
& a_+ = \sqrt{\frac{1-\eps}{\eps}}\, ,\;\;\;\;\;\; a_- = \sqrt{\frac{\eps}{1-\eps}}\, .
\end{align}
Here the coefficients $a_+,a_-$ are chosen to ensure that $\int x\nu_{X_0}(\de x) = 0$, $\int x^2\nu_{X_0}(\de x) = 1$. 
The conditional expectation $F(y;\gamma)$ of Eq.~(\ref{eq:Fdef}) can be computed explicitly, yielding
\begin{align}
F(y;\gamma)&= \frac{\eps a_{+}e^{a_+ y-\gamma a_+^2/2}-(1-\eps)a_-e^{- a_- y-\gamma a_-^2/2}}{\eps e^{a_+ y-\gamma a_+^2/2}+(1-\eps)e^{-a_- y-\gamma a_-^2/2}}\, .\label{eq:F_2pts}
\end{align}
\begin{figure}[t!]
\includegraphics[width=0.48\textwidth]{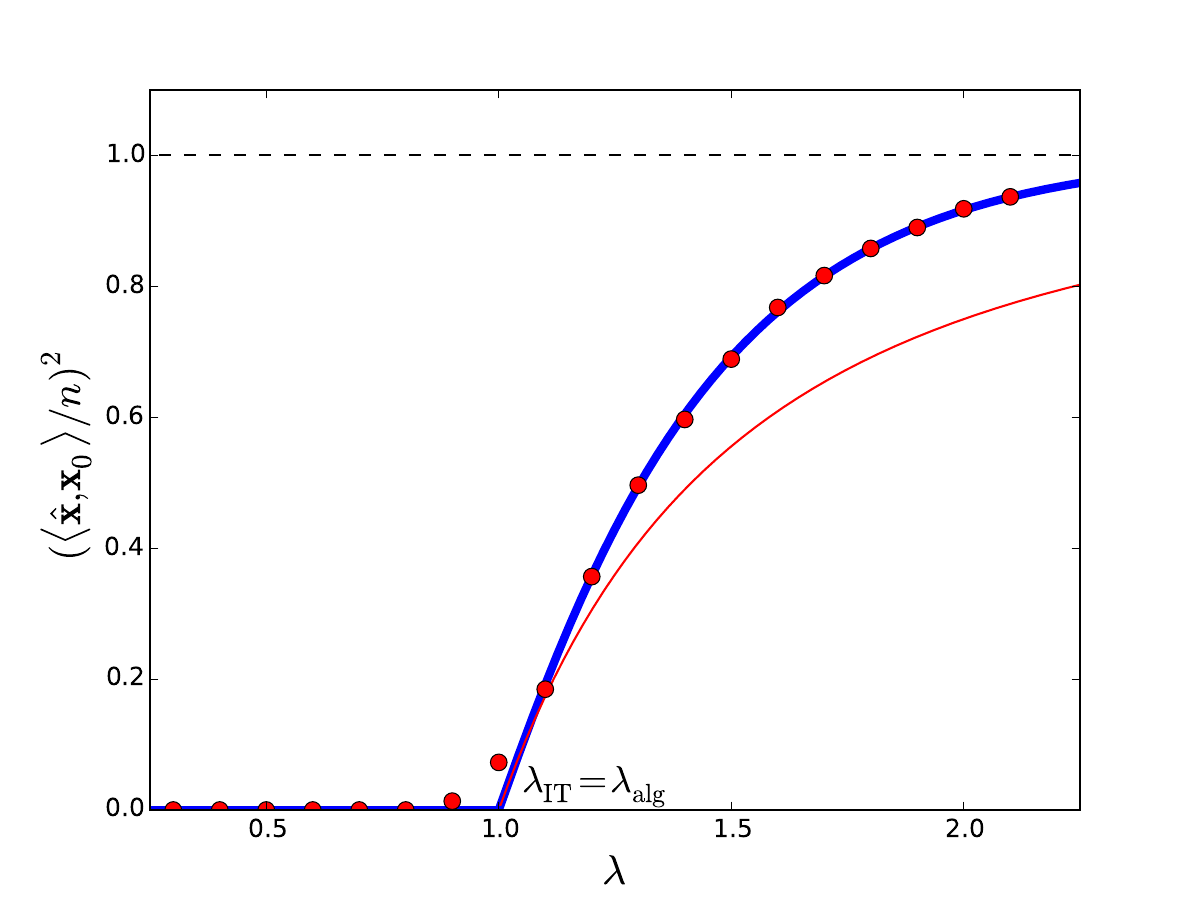}\hspace{-0.1cm}
\includegraphics[width=0.48\textwidth]{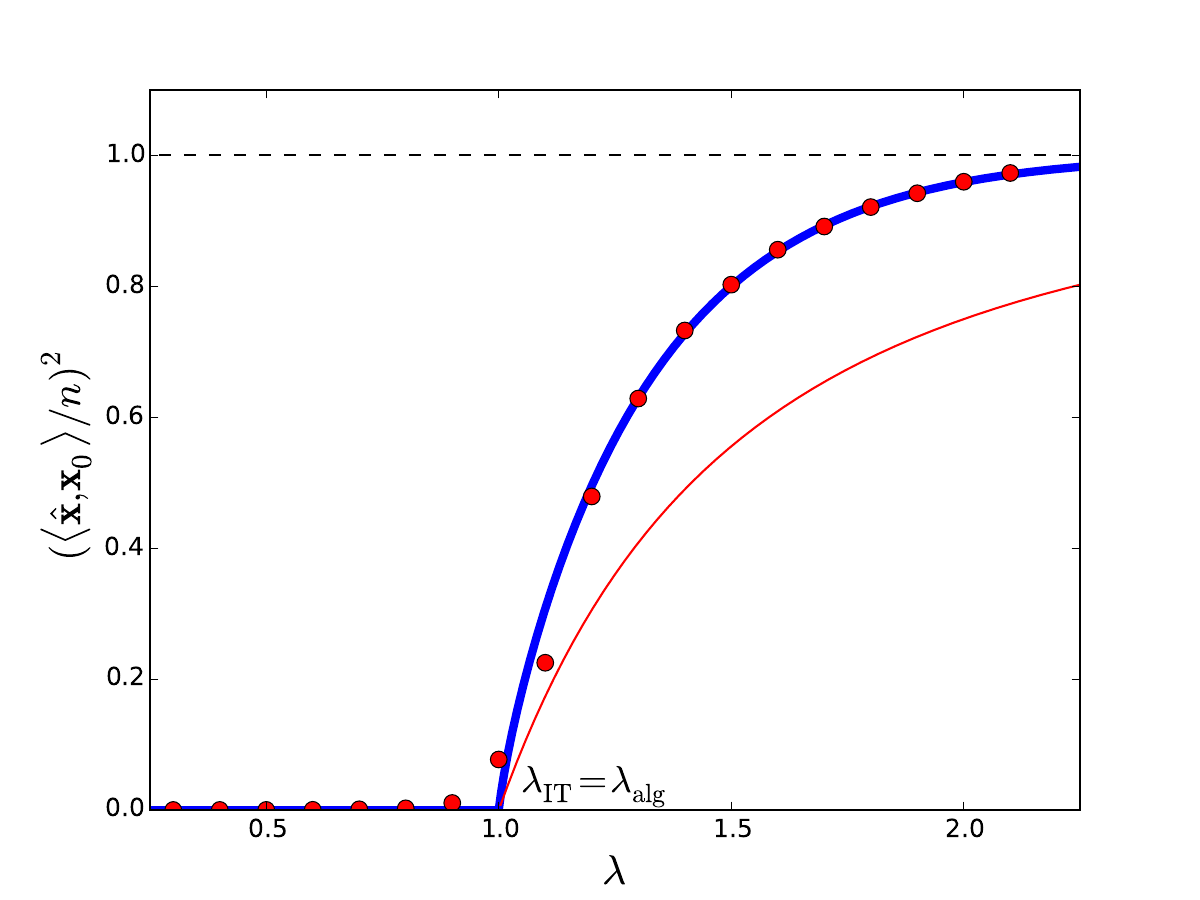}\\
\vspace{-1.75cm}

\includegraphics[width=0.48\textwidth]{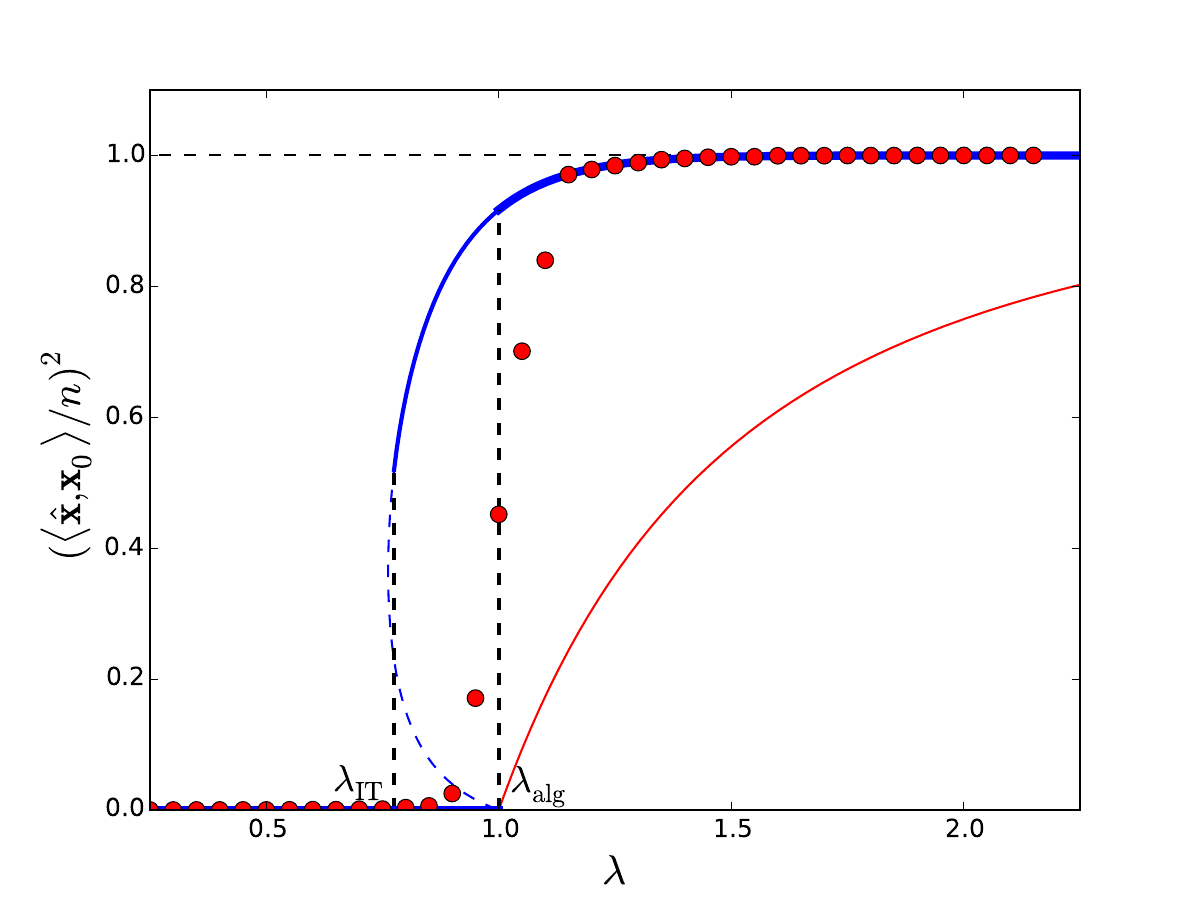}\hspace{-0.1cm}
\includegraphics[width=0.48\textwidth]{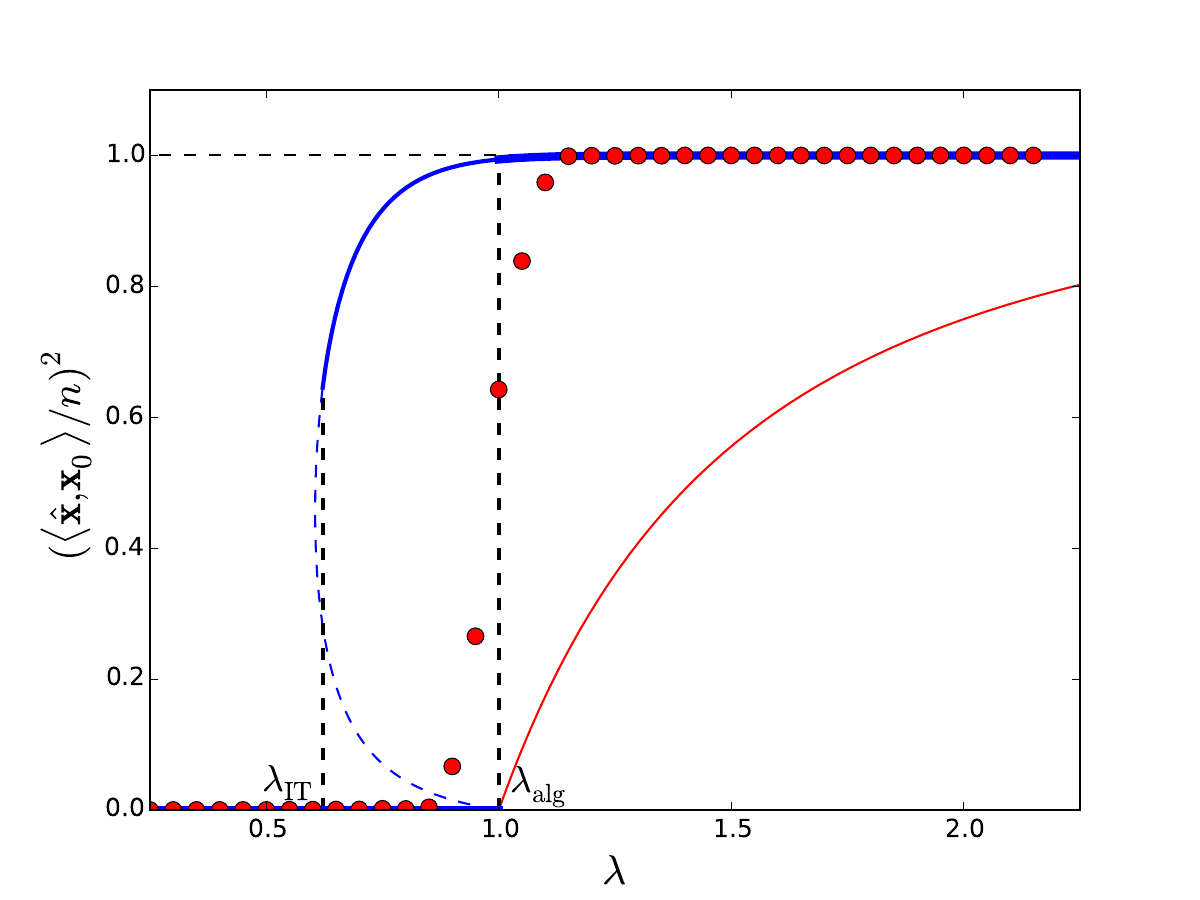}
\put(-310,205){$\eps=0.5$}
\put(-70,205){$\eps=0.25$}
\put(-310,25){$\eps=0.05$}
\put(-70,25){$\eps=0.025$}
\caption{Estimation in the single spiked model (\ref{eq:SpikedDefSpecial}) with entries of $\bx_0$ following the two-points distribution of 
Eq.~(\ref{eq:EmpiricalSpecial}), and four different values of the sparsity $\eps\in\{0.025,0.05,0.25,0.5\}$. Continuous thick blue line: asymptotic accuracy achieved by AMP (with spectral initialization). Red circles: numerical simulations with the AMP algorithm (form matrices of dimension $n=2000$ and $t=200$ iterations).
Continuous thin blue line: Bayes optimal estimation accuracy. Dashed blue line: other fixed points of state evolution. Red line: Accuracy achieved by principal component analysis. Vertical dashed black lines: the thresholds $\lambda_{\sIT}$ and $\lambda_{\sALG}$.}\label{fig:curvesAMP}
\end{figure}

Figure \ref{fig:curvesAMP} reports the results of numerical simulations with the AMP algorithm decribed in
the previous section. We also plot $\gamma_*(\lambda)/\lambda^2$ as a function of $\lambda$, where $\gamma_*(\lambda)$ is the fixed point of the state-evolution equation (\ref{eq:SEspecial_2}). The figure shows plots for four values of $\eps\in (0,1/2]$. The qualitative behavior depends on the value of $\eps$. For $\eps$ close enough to $1/2$,
Eq.~(\ref{eq:SEspecial_2}) only has one stable fixed point\footnote{This is proved formally in \cite{deshpande2016asymptotic} for $\eps=1/2$ and
holds by a continuity argument for $\eps$ close enough to $1/2$. However, here we will limit ourselves to a heuristic discussion based on the numerical
solution of Eq.~(\ref{eq:SEspecial_2}).} 
that is also the  minimizer of the free energy functional (\ref{eq:FreeEnergy}). Hence
$\gamma_{\sALG}(\gamma) = \gamma_{\sBayes}(\lambda)$ for all values of $\lambda$: message passing is always Bayes optimal.

For $\eps$ small enough, there exists $\lambda_0(\eps)<1$ such that Eq.~(\ref{eq:SEspecial_2})  has three fixed points for $\lambda \in (\lambda_{0}(\eps),1)$:
$\gamma_0(\lambda)<\gamma_1(\lambda)<\gamma_2(\lambda)$ whereby $\gamma_0=0$ and $\gamma_2$ are stable and $\gamma_1$ is unstable.
AMP is controlled by the smallest stable fixed point, and hence $\gamma_{\sALG}(\lambda) = 0$ for all $\lambda<1$. On the other hand, by 
minimizing the free energy (\ref{eq:FreeEnergy}) over these fixed points, we obtain that there exists $\lambda_{\sIT}(\eps) \in  (\lambda_{0}(\eps),1)$
such that $\gamma_{\sBayes}(\lambda) = 0$ for $\lambda<\lambda_{\sIT}(\eps)$ while $\gamma_{\sBayes}(\lambda) = \gamma_2(\lambda)$ for 
$\lambda>\lambda_{\sIT}(\eps)$.
We conclude that AMP is asymptotically sub-optimal for $\lambda\in (\lambda_{\sIT}(\eps),1)$, while it
is asymptotically optimal for $\lambda\in [0,\lambda_{\sIT}(\eps))$ and $\lambda\in (1,\infty)$.

\section{Confidence intervals, $p$-values, asymptotic FDR control}
\label{sec:Inference}

As an application of Theorem \ref{thm:Special}, we can construct confidence intervals that achieve a 
pre-assigned coverage level $(1-\alpha)$, where $\alpha\in (0,1)$. Indeed, Theorem \ref{thm:Special} informally states that the AMP iterates $\bx^t$
are approximately Gaussian with mean (proportional to) the  signal $\bx_0$. This relation can be inverted to construct confidence intervals.

We begin by noting that we do not need to know the signal strength $\lambda$. Indeed, for $\lambda>1$, the latter can be estimated from
the maximum eigenvalue of $\bA$, $\lambda_{\max}(\bA)$, via
\begin{align}
\hlambda(\bA) \equiv \frac{1}{2}\Big\{\lambda_{\max}(\bA) + \sqrt{\lambda_{\max}(\bA)^2-4}\Big\}\, .
\label{eq:lambda_hat}
\end{align}
This is a consistent estimator for $\lambda>1$, and can replace $\lambda$ in the iteration of Eq.~(\ref{eq:AMPspecial0}) and initialization (\ref{eq:AMP_Bayes_In})
as well as in the state evolution iteration of Eqs.~\eqref{eq:SEspecial_1} and \eqref{eq:SEspecial_2}. 
We discuss two constructions of confidence intervals: the first one uses the Bayes AMP algorithm of Section \ref{sec:BayesAMP},
and the second instead uses the general algorithm of Section \ref{sec:GeneralRank1}. The optimality of Bayes AMP translates into shorter confidence intervals but also requires knowledge 
of the empirical distribution $\nu_{X_0}$.

\vspace{0.5cm} 

\noindent{\bf Bayes-optimal construction.} In order to emphasize the fact that we use the estimated $\lambda$ both in the AMP iteration and in the state evolution recursion,
we write $\obx^t$ for the Bayes AMP iterates and $\hgamma_t$ for the state evolution parameter, instead of $\bx^t$ and $\gamma_t$. 
We then form the intervals:
\begin{align}
\hJ_i(\alpha;t) = \left[\frac{1}{\hgamma_t} \ox^t_i- \frac{1}{\sqrt{\hgamma_t}} \Phi^{-1}\left(1- \frac{\alpha}{2} \right), \
  \frac{1}{\hgamma_t} \ox^t_i+\frac{1}{\sqrt{\hgamma_t}} \Phi^{-1} \left(1- \frac{\alpha}{2} \right) \right]\, . \label{eq:ConfInterval}
\end{align}
We can also define corresponding $p$-values by
\begin{align}
p_i = 2\left(1-\Phi\Big(\frac{1}{\sqrt{\hgamma_t}} |\ox^t_i|\Big)\right)\, .
\label{eq:pi_Bayes}
\end{align}

\vspace{0.5cm} 

\noindent{\bf General construction (no prior knowledge).} Given a sequence of Lipschitz functions $f_t:\reals\to \reals$, we  let $\bx^t$ be the general AMP iterates as per Section
 \ref{sec:GeneralRank1}, cf. Eq.~(\ref{eq:AMPspecial0}). In order to form confidence intervals, we need to estimate the parameters $\mu_t$, $\sigma_t$. 
In view of Theorem \ref{thm:Rank1}, a possible choice is given by
\begin{align}
\hsigma_t^2 & \equiv \frac{1}{n}\big\|f_{t-1}(\bx^{t-1})\big\|_2^2\,,\\
\hmu_t^2 & \equiv \frac{1}{n}\big\|\bx^t\big\|_2^2- \frac{1}{n}\big\|f_{t-1}(\bx^{t-1})\big\|_2^2\, .
\end{align}
We then construct confidence intervals and $p$-values
\begin{align}
\hJ_i(\alpha;t) &= \left[\frac{1}{\hmu_t} x^t_i- \frac{\hsigma_t}{\hmu_t} \Phi^{-1}\left(1-\frac{\alpha}{2}\right), \ 
\frac{1}{\hmu_t} x^t_i+\frac{\hsigma_t}{\hmu_t}  \Phi^{-1}\left(1-\frac{\alpha}{2} \right)\right]\, ,\label{eq:ConfInterval_2}\\
p_i(t) &= 2\left(1-\Phi\Big(\frac{1}{\hsigma_t} |x^t_i|\Big)\right)\, . \label{eq:pi_def}
\end{align}
\begin{corollary}\label{coro:ConfInt}
Consider the spiked matrix model (\ref{eq:SpikedDefSpecial}), under the assumptions of Theorem \ref{thm:Rank1} (in case of no prior knowledge)
or Theorem \ref{thm:Special} (for the Bayes optimal construction).
Defining the confidence intervals $\hJ_i(\alpha;t)$ as per Eqs. ~(\ref{eq:ConfInterval})  ~(\ref{eq:ConfInterval_2}), we have almost surely
\begin{align}
\lim_{n\to\infty}\frac{1}{n}\sum_{i=1}^n \ind\big(x_{0,i}\in \hJ_i(\alpha;t)\big) = 1-\alpha\, . \label{eq:AS-ConfInt}
\end{align}
Further assume that the fraction of non-zero entries in the spike is $\|\bx_0(n)\|_0/n\to \eps\in [0,1)$, and $\nu_{X_0}(\{0\}) = 1-\eps$. 
Then the  $p$-values constructed above are asymptoticaly valid for the nulls. Namely,
let $i_0 = i_0(n)$ any index such that $x_{0,i_0}(n) = 0$. Then, for any $\alpha\in [0,1]$, and any fixed $t\ge 0$
\begin{align}
\lim_{n\to\infty} \prob\big(p_{i_0(n)}(t)\le \alpha \big) = \alpha\, .  \label{eq:ValidNull}
\end{align}
\end{corollary}
The proof of this result is presented in Appendix \ref{app:CoroConfInt}. Notice that, by dominated convergence, 
this corollary also implies validity of the confidence intervals on average, namely $\lim_{n\to\infty}\frac{1}{n}\sum_{i=1}^n \prob\big(x_{0,i}\in \hJ_i(\alpha;t)\big) = 1-\alpha$.
As mentioned above, cf. Remark \ref{rmk:OptBayes}, the Bayes-optimal construction maximizes the ratio $\mu_t/\sigma_t$ and therefore
minimizes the length of confidence intervals. This requires however additional knowledge of the empirical distribution $\nu_{X_0}$.

Corollary \ref{coro:ConfInt} allows to control the probability of false positives when using the $p$-values
$p_i$,  see Eq.~\eqref{eq:ValidNull}. We might want to use these $p$-values to select a subset of variables $\hS\subseteq [p]$ to be considered for further exploration. 
For such applications, it is common to aim for false discovery rate (FDR) control. The $p$-values $p_i$ guarantee asymptotic FDR control through a simple 
Benjamini-Hochberg procedure \cite{benjamini1995controlling}.
For a threshold  $s\in [0,1]$, we define the following estimator of false discovery proportion \cite{efron2012large}:
\begin{align}
\hFDP(s;t) \equiv \frac{ n s}{1\vee \Big( \sum_{i=1}^n\ind_{\{p_i(t)\le s\}}\Big)}\, .
\label{eq:FDP_def}
\end{align}
Using this notion, we define a threshold and a rejection set as follows. Fix $\alpha \in (0,1)$, let
\begin{align}
s_*(\alpha;t) \equiv \inf\big\{\, s \in [0,1]: \; \hFDP(s;t)\ge \alpha\, \big\}\, ,\;\;\;\; \hS(\alpha;t) \equiv \big\{i\in [n]: \; p_i(t)<s_*(\alpha;t)\big\}
\label{eq:thresh_rej_set}
\end{align}
The false discovery rate for this procedure is defined as usual
\begin{align}
\FDR(\alpha,t;n) \equiv \E\left\{\frac{ |\hS(\alpha;t)\cap \{i:\, x_{0,i}=0\}| }{1\vee |\hS(\alpha;t)| }\right\}
\end{align}
Our next corollary shows that the above procedure is guaranteed to control FDR in an asymptotic sense. Its proof can be found in
Appendix \ref{app:FDR}.
\begin{corollary}
\label{corr:FDR}
Consider the spiked matrix model (\ref{eq:SpikedDefSpecial}), under the assumptions of Theorem \ref{thm:Rank1} (in case of no prior knowledge)
or Theorem \ref{thm:Special} (for the Bayes optimal construction). Further assume that the fraction of non-zero entries in the spike is $\|\bx_0(n)\|_0/n \to  \eps\in [0,1)$, and $\nu_{X_0}(\{0\}) = 1-\eps$.  Then, for any fixed $t\ge 0$,
\begin{align}
\lim_{n\to\infty}\FDR(\alpha,t;n) = (1-\eps)\alpha\, .
\label{eq:FDRbound}
\end{align}
\end{corollary}

\begin{remark}
\label{rem:FDR}
The procedure defined by threshold and rejection set in \myeqref{eq:thresh_rej_set} does not assume knowledge of the sparsity level $\eps$. If one knew $\eps$, then  an asymptotic false discovery rate of exactly $\alpha$ can be obtained by defining \cite{storey2002direct}
\[ \hFDP(s;t) \equiv \frac{ n (1-\eps)s}{1\vee \Big( \sum_{i=1}^n\ind_{\{p_i(t)\le s\}}\Big)}\, . \]
With the threshold and rejection set defined as in \myeqref{eq:thresh_rej_set}, such a procedure would have an asymptotic FDR equal to  $\alpha$, and higher power than the procedure using the estimator in \myeqref{eq:FDP_def}. 
\end{remark}


\def\ssc{{\sf c}}
\def\omu{{\overline{\mu}}}
\def\osigma{{\overline{\sigma}}}
 \def\ogamma{{\overline{\gamma}}}
 
\section{Estimation of rectangular rank-one matrices}
\label{sec:ExampleRectangular}

The algorithms and analysis developed in previous sections can be generalized to rectangular matrices. 
We illustrate this by generalizing the rank-one result of Theorem \ref{thm:Rank1}. 
We consider a data matrix $\bA\in\reals^{n\times d}$ given by
\begin{align}
\bA = \frac{\lambda}{n}\bu_0\bx_0^{\sT}+\bW\, ,\label{eq:RectangularSpikedRankOne}
\end{align}
where $(\bW_{ij})_{i\le n,j\le d}\sim_{iid}\normal(0,1/n)$. 
To be definite, we will think of sequences of instances indexed by $n$ and assume $n,d\to\infty$ with aspect ratio $d(n)/n\to \alpha\in (0,\infty)$.

We will make the following assumptions on the sequences of vectors
$\bu_0 = \bu_0(n)$, $\bx_0=\bx_0(n)$:
\begin{itemize}
\item[$(i)$] Their rescaled $\ell_2$-norms converge: $\lim_{n\to\infty}\|\bu_0(n)\|_2/\sqrt{n}= 1$,  $\lim_{n\to\infty}\|\bx_0(n)\|_2/\sqrt{d(n)}= 1$; 
\item[$(ii)$] The empirical distributions of the entries of $\bx_0(n)$ and $\bu_0(n)$ converges weakly to probability distributions $\nu_{X_0}$, $\nu_{U_0}$,  on $\reals$,
with unit second moment.
\end{itemize}

In analogy with the symmetric case, we initialize the AMP iteration by using the principal right singular vector of $\bA$, denoted by $\bphi_1$
(which we assume to have unit norm). In the present case, the phase transition for the principal singular vector takes place at 
$\lambda^2\sqrt{\alpha}=1$ \cite{paul2007asymptotics,BaiSilverstein}. Namely, if $\lambda^2\sqrt{\alpha}>1$ then the correlation between
$|\<\bx_0,\bphi_1\>|/\|\bx_0\|$ stays bounded away from zero as $n,d\to\infty$.

Setting $\bx^0= \sqrt{d}\bphi_1$ and  $g_{t-1}(\bu^{t-1}) =\bzero$, we consider the following AMP iteration:
\begin{align}
\bu^t & = \bA f_t(\bx^t) - \sb_t g_{t-1}(\bu^{t-1})\, ,\;\;\;\;\;\;\;
        \sb_t =\frac{1}{n}\sum_{i=1}^df_t'(x^t_i)\, ,\\
\bx^{t+1} & =\bA^{\sT} g_{t}(\bu^t)-\ssc_{t} f_t(\bx^t)\, ,\;\;\;\;\;\;\;\;\;\;
        \ssc_t =\frac{1}{n}\sum_{i=1}^ng_t'(u^t_i)\, .
\end{align}

The asymptotic characterization of this iteration is provided by the next theorem, which generalizes Theorem \ref{thm:Rank1}
to the rectangular case. 
\begin{theorem}\label{thm:Rank1-Rectangular}
Consider the $k=1$ spiked  matrix model  of Eq.~\eqref{eq:RectangularSpikedRankOne},  with $n,d\to\infty$, $d/n\to \alpha$.
Assume  $\bx_0(n)\in\reals^d$, $\bu_0(n)\in\reals^d$ to be two sequences of vectors satisfying assumptions $(i)$, $(ii)$ above, and $\lambda^2\sqrt{\alpha}>1$ . 
Consider the AMP iteration in Eq.~\eqref{eq:AMPspecial0}  with initialization  $\bx^0 = \sqrt{n} \, \bphi_1$ (where, without loss of generality $\<\bx_0,\bphi_1\> \ge 0$). 
Assume $f_t,g_t:\reals\to\reals$
to be Lipschitz continuous for each $t\in \naturals$. 

Let  $(\mu_t, \sigma_t)_{t \geq 0}$ be defined via the recursion
\begin{align}
\mu_{t+1} & = \lambda \E[ U_0\, g_t(\omu_t U_0+\osigma_t G)]\, ,  \;\;\;\;\;\;\;\;\;\;
 \sigma_{t+1}^2 = \E [g_t(\omu_t U_0 + \osigma_t G )^2]\, , \label{eq:Rect-X0}\\
\omu_{t} & = \lambda\alpha\, \E[ X_0 \, f_t(\mu_t X_0+\sigma_t G)]\, ,   \;\;\;\;\;\;\;\;\;\;
 \osigma_{t}^2 = \alpha\, \E [f_t(\mu_t X_0 + \sigma_t G )^2]\, , \label{eq:Rect-U0}
\end{align}
where $X_0\sim \nu_{X_0}$, $U_0\sim \nu_{U_0}$ and $G \sim \normal (0,1)$ are independent, and 
the initial condition is 
\begin{align}
\mu_0= \sqrt{\frac{1-\alpha^{-1}\lambda^{-4}}{1+\lambda^{-2}}}\, ,\;\;\;\;\; \sigma_0 = \sqrt{\frac{\lambda^{-2}+\alpha^{-1}\lambda^{-4}}{1+\lambda^{-2}}}\, .
\end{align}
(This is to be substituted in Eq.~\eqref{eq:Rect-U0} to yield $\omu_0,\osigma_0$.)

Then, for any function $\psi:\reals\times\reals\to\reals$  with $|\psi(\bx)-\psi(\by)|\le C(1+\|\bx\|_2+\|\by\|_2)\|\bx - \by\|_2$ for a universal constant $C>0$, the following holds almost surely for $t \geq 0$:
\begin{align}
\lim_{n \to \infty} \frac{1}{d(n)} \sum_{i=1}^{d(n)} \psi (x_{0,i},x^t_i) = \E \left\{ \psi( X_0, \mu_t X_0  +\sigma_t G) \right\} \ , \label{eq:rank1_SE_Rect_A}\\
\lim_{n \to \infty} \frac{1}{n} \sum_{i=1}^n \psi (u_{0,i},u^t_i) = \E \left\{ \psi( U_0, \omu_t U_0  +\osigma_t G) \right\} \ . \label{eq:rank1_SE_Rect_B}
\end{align}
\end{theorem}

As a special class of examples covered by this setting, we can consider the case in which we are given i.i.d. 
Gaussian samples $(\by_i)_{i\le n}\sim\normal(\bzero,\bSigma)$, with covariance matrix $\bSigma = \rho^2\tbx_0\tbx_0^{\sT}+\id_d$ where $\tbx_0 = \bx_0/\sqrt{d}$.
Letting $\bA$ be the matrix with $i$-th row equal to $\by_i/\sqrt{n}$,  this takes the form of Eq.~\eqref{eq:RectangularSpikedRankOne},
with $\bu_{0}\sim\normal(0,\id_n)$, and $\lambda = \rho/\sqrt{\alpha}$. Notice that the sequence of random Gaussian vectors $\bu_0(n)$, $n\ge 1$ 
satisfies conditions $(i)$, $(ii)$ above almost surely, with limit distribution $\nu_{U_0}$ equal to the standard Gaussian measure.

In this case, the optimal choice of the function $g_t$ in Eq.~\eqref{eq:Rect-X0} is of course linear: $g_t(u) = a_tu$ for some $a_t>0$. The value
of the constant $a_t$ is immaterial, because it only amounts to a common rescaling of the  $\mu_t,\sigma_t$, which can be compensated
by a redefinition of $f_t$ in Eq.~\eqref{eq:Rect-U0}. We set $a_t = \lambda\omu_t/(\omu_t^2+\osigma_t^2)$. Substituting in Eq.~\eqref{eq:Rect-X0},
we obtain $\mu_{t+1} = \sigma_{t+1}^2=\gamma_{t+1}$, where
\begin{align}
\gamma_{t+1} = \frac{\lambda^2\ogamma_t^2}{1+\ogamma_t^2}\, ,\label{eq:GaussianCov-1}
\end{align}
where $\ogamma_t= \omu_t^2/\osigma_t^2$. Taking the ratio of the two equations in \eqref{eq:Rect-U0}, we obtain
\begin{align}
\ogamma_{t} = \lambda^2\alpha \frac{\E\{X_0 f_t(\gamma_t X+\sqrt{\gamma_t}G)\}^2}{\E\{f_t(\gamma_t X+\sqrt{\gamma_t}G)^2\}}\, . \label{eq:GaussianCov-2}
\end{align}
We thus reduced the problem of covariance estimation in the spiked model $\bSigma = \rho^2\tbx_0\tbx_0^{\sT}+\id_d$, to
the analysis of a one-dimensional recursion defined by Eqs.~\eqref{eq:GaussianCov-1}, \eqref{eq:GaussianCov-2}. 
%
\section{Degenerate cases and non-concentration}
\label{sec:Degenerate}

The spectral  initialization at unstable fixed points leads to a new phenomenon
that is not captured by previous theory \cite{BM-MPCS-2011}: the evolution of empirical averages (e.g. estimation accuracy) does not always concentrate
around a deterministic value. 
Our main result,  Theorem \ref{thm:Main} below, provides a description of this phenomenon by establishing a state evolution limit that is 
\emph{dependent on the random initial condition}. The initial condition converges in distribution to a well defined limit, which--- together with state evolution---yields a complete 
characterization of the asymptotic behavior of the message passing algorithm.

The non-concentration phenomenon arises when the deterministic low-rank component in Eq.~(\ref{eq:SpikedDef}) 
has   degenerate eigenvalues. This is unavoidable in cases in which
the underlying low-rank model to be estimated has symmetries. 

Here we illustrate this phenomenon on a simple model that we will refer to as the Gaussian Block Model (GBM).
For $q\ge 3$ a fixed integer, let $\bsigma=  (\sigma_1,\dots,\sigma_n)$ be a vector of vertex labels with $\sigma_i\in\{1,\dots,q\}$ and consider deterministic 
matrix $\bA_0\in\reals^{n\times n}$ (with $\rank(\bA_0) = q-1$) defined by:
\begin{align}
A_{0,ij}= \begin{cases}
(q-1)/n & \mbox{if $\sigma_i = \sigma_j$}\\
-1/n & \mbox{otherwise.}\\
\end{cases}
\end{align}
We assume  the vertex labeling to be perfectly balanced. i.e. $\sum_{i=1}^n\bfone_{\sigma_i=\sigma}=n/q$ for $\sigma\in\{1,\dots,q\}$:
While most of our discussion holds under an approximate balance condition, this assumption avoids some minor technical complications.
Notice that $\bA_0$ is an orthogonal  projector on a subspace
$\cV_n\in\reals^{n}$ of dimension $q-1$. 
We observe the noisy matrix (with noise $\bW\sim\GOE(n)$)
\begin{align}
\bA = \lambda \bA_0+\bW\, ,\label{eq:BlockMatrix}
\end{align}
and would like to estimate $\bA_0$ from these noisy observations. The matrix $\bA$ takes the form of Eq.~(\ref{eq:SpikedDef}) 
with $k=q-1$, $\lambda_1=\dots=\lambda_k = \lambda$ and $\bv_1$, \dots, $\bv_k$ an orthonormal basis of the space $\cV_n$.
We will assume $\lambda >1$ so that $k_* = k$. In particular, for $q\ge 3$, the low-rank signal has degenerate eigenvalues.

We use the following AMP algorithm to estimate $\bA_0$. We compute the top $k$ eigenvectors of $\bA$, denoted by $\bphi_1,\dots,\bphi_k\in\reals^n$
and generate $\bx^t\in\reals^{n\times q}$ for $t \geq 0$, according to
\begin{align}
\bx^0&= [\sqrt{n}\bphi_1|\cdots|\sqrt{n}\bphi_k|\bzero]\, ,\label{eq:BlockSpectral}\\
\bx^{t+1} & = \bA f(\bx^t) - f(\bx^{t-1})\, \sB_t^{\sT}\, ,\label{eq:AMP-Block}
\end{align}
where the  `Onsager coefficient' $\sB_t\in \reals^{q\times q}$ is a  matrix given by
\begin{align}
\sB_t = \frac{1}{n}\sum_{i=1}^n\frac{\partial f}{\partial \bx}(\bx^t_i,y_i)\,.
\end{align}
Here $\frac{\partial f}{\partial \bx}\in \reals^{q\times q}$ denotes the Jacobian matrix of the function $f:\reals^q\to\reals^q$. Furthermore, the function $f:\reals^q\to\reals^q$ is defined by
letting, for $\sigma\in\{1,\dots,q\}$:
\begin{align}
f(\bz)_\sigma = \lambda\left[\frac{q e^{z_{\sigma}}}{\sum_{\tau=1}^q e^{z_{\tau}}}-1\right]\, ,\label{eq:FunctionBlock}
\end{align}
and $f(\bx)$ is defined for $\bx\in\reals^{n\times q}$ by applying the same function row by row. 
This choice of the function $f$ corresponds to Bayes-optimal estimation as can be deduced from the state evolution analysis below: we will not discuss this point in  detail here.

The output $\bx^t$ after $t$ iterations of (\ref{eq:AMP-Block})  can be interpreted as an estimate of the labels $\bsigma$
in the following sense. Let $\bx_0\in\reals^{n\times q}$ be the matrix whose $i$-th row is  $\bx_{0,i}= \projp\be_{\bsigma_i}$, with $\projp\in\reals^{q\times q}$
the projector orthogonal to the all ones vector, and $\be_1,\dots,\be_q$ the canonical basis in $\reals^q$.
Note that $\bA_0= (q/n)\bx_0\bx_0^{\sT}$. Then $\bx^t$ is an estimator of $\bx_0$ (up to a permutation of the labels' alphabet
$\{1,\dots,q\}$).

Let ${\sf S}_q$ be the group of $q\times q$ permutation matrices. 
We evaluate the estimator $\bx^t$ via the overlap
\begin{align}
\Overlap_n(\lambda;t) \equiv \max_{\bPi\in {\sf S}_q}\frac{\<\bx^t,\bx_0\bPi\>}{\|\bx^t\|_F\|\bx_0\|_F}\, ,
\end{align}
where $\< \cdot, \cdot\>$ denotes the Frobenius inner product. In Figure \ref{fig:Degenerate}, we plot the evolution of the overlap in two sets of numerical simulations, 
for $q=3$ and $q=4$. Each curve is obtained by running AMP (with spectral initialization)
on a different realization of the random matrix $\bA$. The non-concentration phenomenon is quite clear: 
\begin{itemize}
\item For fixed number of iterations $t$ and large $n$,
the quantity $\Overlap_n(\lambda;t)$ has large fluctuations, that do not seem to vanish as $n\to\infty$.
\item Despite this, the algorithm is effective in reconstructing the signal: after $t= 10$ iterations, the accuracy achieved
is nearly independent of the initialization.
\end{itemize}

\begin{figure}[t!]
\includegraphics[width=0.48\textwidth]{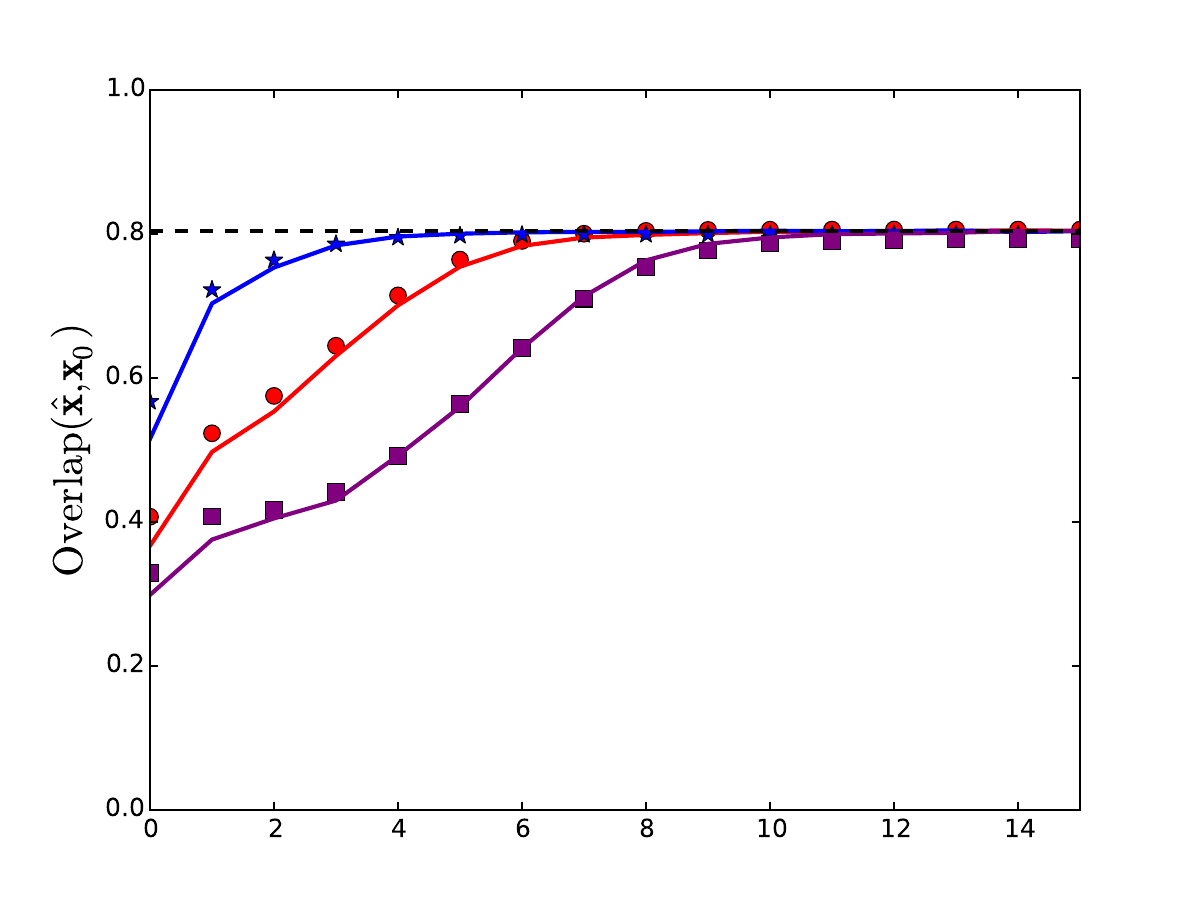}\hspace{-0.15cm}
\includegraphics[width=0.48\textwidth]{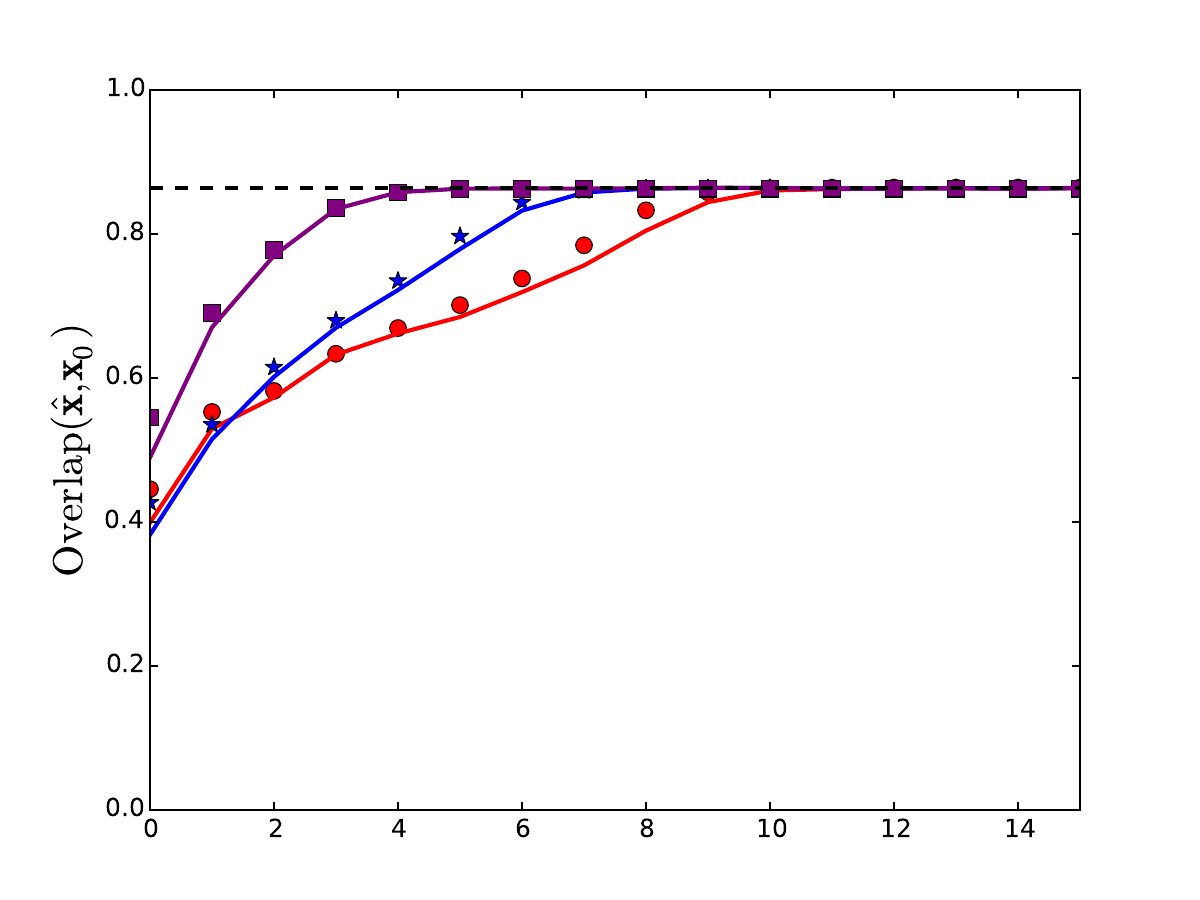}
\put(-90,-10){$t$}
\put(-340,-10){$t$}
\caption{Estimation in the Gaussian Block Model of Eq.~(\ref{eq:BlockMatrix}) using the AMP algorithm with spectral initialization of
Eqs.~(\ref{eq:BlockSpectral}), (\ref{eq:AMP-Block}). We plot the reconstruction accuracy (overlap) as a function of the number of iterations
for $q=3$, $\lambda=1.5$, $n=6000$ (left frame),  and $q=4$, $\lambda=1.75$, $n=8000$ (right frame). Each set of symbols corresponds to a different realization 
of the random matrix $\bA$, and curves report the corresponding prediction of Theorem \ref{thm:Block}. Dashed black lines report the Bayes optimal accuracy as per
\cite{barbier2016mutual,lelarge2016fundamental}.} 
\label{fig:Degenerate}
\end{figure}

The empirical data in Figure \ref{fig:Degenerate} are well described by the state evolution prediction 
that is shown as continuous curves in the same figure. 
In this case, state evolution operates on the pair of matrices $\bM_t, \bQ_t\in\reals^{q\times q}$, which
are updated according to 
\begin{align}
\bM_{t+1}& = \lambda\E\Big\{f\big(q\bM_t\be_{\sigma} +\bQ_t^{1/2}\bG\big)\, \be_{\sigma}^{\sT}\projp\Big\}\,, \label{eq:SE_Block_1}\\
\bQ_{t+1}& = \E\Big\{f\big(q\bM_t\be_{\sigma} +\bQ_t^{1/2}\bG\big)\, f\big(q\bM_t\be_{\sigma} +\bQ_t^{1/2}\bG\big)^{\sT}\Big\}\, ,\label{eq:SE_Block_2}
\end{align}
where $f:\reals^q\to\reals^q$ is defined as per Eq.~(\ref{eq:FunctionBlock}), and expectation is with respect to 
$\sigma$ uniform in $\{1,\dots,q\}$ independent of $\bG\sim\normal(0,\id_q)$. Note that $\bQ_t$ is symmetric and 
both $\bQ_t\bfone= \bM_t\bfone = \bfone^{\sT} \bM_t = 0$ for all $t\ge 1$. 

The state evolution prediction for the present model is provided by the next theorem, which is proved in Appendix \ref{app:Block}.
\begin{theorem}\label{thm:Block}
Let $\bA\in\reals^{n\times n}$ be the random matrix of Eq.~(\ref{eq:BlockMatrix}) with $\lambda>1$, and let $\bphi_1,\dots,\bphi_k$ be its top $k$ eigenvectors.
Denote by $\bx^t$ the sequence of estimates produced by the AMP algorithm of Eq.~(\ref{eq:BlockSpectral}) with the spectral initialization in Eq.~(\ref{eq:AMP-Block}).

Let $\{\bM_t, \bQ_t\}_{t\ge 0}$ be the state evolution iterates with initialization $\bM_0= (\bx^0)^{\sT}\bx_0/n$ and $\bQ_0=\lambda^{-1}\diag(1,1,\dots,1,0)$ 
Then, for any function $\psi:\reals^{2q}\to \reals$ with $|\psi(\bx)-\psi(\by)|\le C(1+\|\bx\|_2+\|\by\|_2)\|\bx-\by\|_2$,  we have, almost surely
\begin{align}
\lim_{n\to\infty}\left|\frac{1}{n}\sum_{i=1}^n\psi(\bx_i^t,\bx_{0,i}) - \E\big\{\psi(q\bM_t\be_{\sigma}+\bQ_t^{1/2}\bG,\projp\be_{\sigma})\big\}\right|= 0\, . \label{eq:SE_Block}
\end{align}
where expectation is with respect to  $\sigma$ uniform in $\{1,\dots,q\}$ independent of $\bG\sim\normal(0,\id_q)$. 

Further as $n\to\infty$, $\bM_0$ converges in distribution as 
\begin{align}
\bM_0 \stackrel{{\rm d}}{\Rightarrow} \sqrt{q^{-1}\left(1-\lambda^{-2}\right)}\, 
\left[\begin{matrix}
\bO^{\sT}_{(q-1)\times q}\\
\bzero_{1\times q}
\end{matrix}\right]\, ,\label{eq:InitBlock}
\end{align}
where $\bO\in \reals^{q\times (q-1)}$ is Haar distributed orthogonal matrix with column space orthogonal to $\bfone$.
\end{theorem}

The continuous curves in Figure \ref{fig:Degenerate} are obtained as
described in the last theorem. For each experiment we generate a random matrix $\bA$ according to Eq.~(\ref{eq:BlockMatrix}),  compute 
the spectral initialization of Eq.~(\ref{eq:BlockSpectral}) and set $\bM_0= (\bx^0)^{\sT}\bx_0/n$. We then compute the state evolution sequence $\{(\bM_t,\bQ_t)\}_{t\ge 0}$
via Eqs.~(\ref{eq:SE_Block_1}), (\ref{eq:SE_Block_2}), and use Eq.~(\ref{eq:SE_Block}) to predict the evolution of the overlap. 
The variability in the initial condition $\bM_0$ leads to a variability in the predicted  trajectory $\{(\bM_t,\bQ_t)\}_{t\ge 0}$ that matches well with the empirical data.

Finally, as mentioned above, AMP converges to an accuracy that is roughly independent of the matrix realization for large $t$, and matches the Bayes optimal prediction
of \cite{barbier2016mutual,lelarge2016fundamental}.  While a full explanation of this phenomenon goes beyond the scope of the present paper, this behavior can be also explained by
Theorem \ref{thm:Block}: the  initialization $\bM_0$ breaks the symmetry between the $q$ blocks uniformly, as per Eq.~(\ref{eq:InitBlock}). Once the symmetry is broken, the
state evolution  iteration of Eqs.~(\ref{eq:SE_Block_1}), (\ref{eq:SE_Block_2}) converges to a fixed point that is unique up to permutations.


\section{Main result}
\label{sec:Symmetric}

\subsection{Notations and definitions}
\label{sec:Notation}

We say that a function $\psi:\reals^d\to \reals$ is pseudo-Lipschitz of order $k$ (and write $\psi\in \PL(k)$)
if there exists a constant $L$ such that $|\psi(\bx)-\psi(\by)|\le L(1+(\|\bx\|/\sqrt{d})^{k-1}+(\|\by\|/\sqrt{d})^{k-1})\|\bx-\by\|_2/\sqrt{d}$.

Recall that a sequence of probability distributions $\nu_n$ on $\reals^m$ \emph{converges weakly} to 
$\nu$ ($\nu_n\toweak \nu$) if, for any bounded Lipschitz function $\psi:\reals^m\to\reals$,
$\lim_{n\to\infty}\E \psi(\bX_n) = \E \psi(\bX)$ where expectation is with respect to $\bX_n\sim \nu_n$, $\bX\sim \nu$.
Given a (deterministic) sequence of matrices $\bZ_n\in\reals^{n \times d}$ indexed by $n$ (with $d\ge 1$ fixed), we 
say that the empirical distribution of $\bZ_n$ converges weakly to a probability distribution $\nu$ on $\reals^d$ if,
letting $\bz_{i}=\bZ_n^{\sT}\be_i$ denote the $i$-th row of $\bZ_n$, for  each $i$ we have 
\begin{align}
\frac{1}{n}\sum_{i=1}^n \delta_{\bz_{n,i}} \toweak \nu\, .
\end{align}
Equivalently, $\lim_{n\to\infty}n^{-1}\sum_{i=1}^n\psi(\bz_{i}) = \E\psi(\bz)$ for $\bz\sim \nu$ and any bounded Lipschitz function $\psi$.
We apply the same terminology if we are given $d$ vectors $(\bz^{(n)}_1,\dots, \bz^{(n)}_d)$, where $\bz^{(n)}_\ell\in\reals^n$:
in this case $\bZ_n$ is the matrix with columns $\bz^{(n)}_1,\dots, \bz^{(n)}_d$.

Given two probability measures $\mu$ (on the space $\cX$) and $\nu$ (on the space $\cY$), a coupling $\rho$ of
$\mu$ and $\nu$ is a probability distribution on $\cX\times\cY$ whose first marginal coincides with $\mu$ and second
coincides with $\nu$. We denote the set of couplings of $\mu,\nu$ by $\cC(\mu,\nu)$.
For $k\ge 1$, the Wasserstein-$k$ ($W_k$) distance between two probability measures $\mu$, $\nu$ on $\reals^d$
is defined by
\begin{align}
W_k(\mu,\nu) \equiv \inf_{\rho\in\cC(\mu,\nu)}\E_{(\bX,\bY)\sim \rho}\big\{\|\bX-\bY\|_2^k\}^{1/k}\, ,\label{eq:WassersteinDef}
\end{align}
where the infimum is over all the couplings of $\mu$ and $\nu$.
A sequence of probability distributions $\nu_n$ on $\reals^m$ \emph{converges in $W_k$} to 
$\nu$ ($\nu_n\towa{k} \nu$) if $\lim_{n\to\infty} W_k(\nu_n,\nu) = 0$. An equivalent definition is that, for any 
$\psi\in\PL(k)$, $\lim_{n\to\infty}\E \psi(\bX_n) = \E \psi(\bX)$ where expectation is with respect to $\bX_n\sim \nu_n$, $\bX\sim \nu$
\cite[Theorem 6.9]{villani2008optimal}. 

Generalizing from the definitions introduced for weak convergence, given sequence of matrices $\bZ_n\in\reals^{n\times d}$ indexed by $n$ (with $d\ge 1$ fixed), we 
say that the empirical distribution of $\bZ_n$ converges in $W_k$ to  $\nu$ (a probability distribution on $\reals^d$), if
letting $\bz_{i}=\bZ_n^{\sT}\be_i$ denote the $i$-th row of $\bZ_n$,
\begin{align}
\frac{1}{n}\sum_{i=1}^n \delta_{\bz_{n,i}} \towa{k} \nu\, .
\end{align}
Equivalently, $\lim_{n\to\infty}n^{-1}\sum_{i=1}^n\psi(\bz_i) = \E\psi(\bz)$ for any $\psi\in\PL(k)$ (where $\bz\sim \nu$). Again the
same terminology is used for $d$-tuples of vectors $(\bz^{(n)}_1,\dots, \bz^{(n)}_d)$.

We will typically use upper case bold symbols for matrices (e.g. $\bA$, $\bB$,\dots), lower case bold for vectors (e.g. $\bu$, $\bv,\dots$) and
lower case plain font for scalars (e.g. $x,y,\dots$). However, we will often denote random variables and random vectors using upper case.

We often consider vectors (or matrices) whose elements are indexed by arbitrary finite sets. For instance, given finite sets $S_1, S_2$,
$\bQ\in \reals^{S_1\times S_2}$ is a matrix $\bQ=(Q_{i,j})_{i\in S_1,j\in S_2}$. When there is an obvious ordering of the elements of $S_1$, $S_2$, such a matrix is
understood to be identified with a matrix in $\reals^{n_1\times n_2}$, where $n_i = |S_i|$. For instance $\reals^{[m]\times [n]}$ is identified with $\reals^{m\times n}$.
Given a vector $\bv\in\reals^m$ an a set $S\subseteq [m]$ we denote by $\bv_S\in\reals^S$ the subvector indexed by elements of $S$.
Analogously, for a matrix $\bM\in\reals^{m\times n}$, we let $\bM_{R,S}\in \reals^{R\times S}$ be the submatrix with row indices in $R$ and column indices in $S$.
If the submatrix includes all the rows, we adopt the shorthand $\bM_{[m],S}$. 

Finally, we adopt the convention that all vectors (including the rows of a matrix) are viewed as column vectors, unless explicitly transposed.

\subsection{Statement of the result: Symmetric case}

Recall the spiked model of Eq.~(\ref{eq:SpikedDef}), which we copy here for the reader's convenience:
\begin{align}
\bA = \sum_{i=1}^k\lambda_i \bv_i\bv_i^{\sT} + \bW \equiv \bV\bLambda\bV^{\sT}  +\bW\, .\label{eq:LowRankPlusNoise}
\end{align}
Here $\bv_i\in\reals^n$ are non-random orthonormal vectors and $\bW\sim\GOE(n)$.
We denote by $\bphi_1,\dots,\bphi_n$ the eigenvectors of $\bA$, with corresponding eigenvalues $z_1\ge  z_2\ge \dots\ge z_n$.

For a sequence of functions $f_t(\,\cdot \, ):\reals^{q}\times\reals\to \reals^q$, we consider the AMP algorithm that produces a sequence of iterates $\bx^t$ according to the recursion
\begin{align}
\bx^{t+1} = \bA f_t(\bx^t,\by) - f_{t-1}(\bx^{t-1},\by)\, \sB_t^{\sT}\, .  \label{eq:AMP}
\end{align}
Here $\by\in \reals^n$ is a fixed vector, and it is understood that $f(\cdot;t)$ is applied row-by-row. Namely, denoting by $\bx^t_i\in \reals^q$ the
$i$-th row of $\bx^t$, the $i$-th row of $f_t(\bx^{t};\by)$ is given by $f_t(\bx^t_i,y_i)$. 
The `Onsager coefficient' $\sB_t\in \reals^{q\times q}$ is a matrix given by
\begin{align}
\sB_t = \frac{1}{n}\sum_{i=1}^n\frac{\partial f_t}{\partial \bx}(\bx^t_i,y_i)\, ,\label{eq:OnsagerDef}
\end{align}
where $\frac{\partial f_t}{\partial \bx}\in \reals^{q\times q}$ denotes the Jacobian matrix of the function $f_t(\,\cdot,y):\reals^q\to\reals^q$.  The algorithm is initialized with  $\bx^0 \in \reals^{n \times q}$ and $f_{-1}(\bx^{-1},\by) \in  \reals^{n \times q}$ is taken to be the all-zeros matrix.

\begin{remark}
Notice that the present setting generalizes the one of Section \ref{sec:Example} in two directions (apart from the more general model for
the matrix $\bA$, cf. Eq.~\eqref{eq:LowRankPlusNoise}). First, the state of the algorithm is a matrix $\bx^t\in\reals^{n\times q}$ with $q$ an arbitrary fixed integer.
While it is natural to take $q$ equal to the number of outliers in the spectrum of $\bA$ (i.e. $q=k_*$ according to the notations introduced below), 
we believe that a more general choice of $q$ can be useful for certain applications. Further, the nonlinearity $f_t$ is a function of $\bx^t$
but also on the independent vector $\by$ that can be regarded as side information: again, we believe this additional freedom will be useful for future
applications of our main result.
\end{remark}

We will make the following assumptions:
\begin{enumerate}[font={\bfseries},label={(A\arabic*)}]
\item\label{ass:Spikes} The values $\lambda_i(n)$ have finite limits as $n\to \infty$, that we denote by $\lambda_i$. 
Further, assume there exist $k_+$, $k_-$ such that
$\lambda_1\ge \dots \lambda_{k_+}>1>\lambda_{k_++1}$ and $\lambda_{k-k_-}>-1> \lambda_{k-k_-+1}\ge \dots\ge \lambda_k$.
We let $S \equiv (1,\dots, k_+,k-k_-+1,\dots,k)$, $k_* = k_++k_-$ and $\hS \equiv  (1,\dots, k_+,n-k_-+1,\dots,n)$.
Further, we let $\bLambda_S$ denote the diagonal matrix with entries $(\bLambda_S)_{ii} = \lambda_i$, $i\in S$.
\item\label{ass:SpikesInit} Setting $q\ge k_*$, we initialize the iteration (\ref{eq:AMP}) by setting $\bx^0\in\reals^{n\times q}$ equal to the matrix with first $k_*$ ordered columns
 given by $(\sqrt{n}\bphi_i)_{i\in \hS}$, and $\bzero$ for the remaining $q-k_*$ columns. 
\item\label{ass:SpikesDistr} The joint empirical distribution of the vectors $(\sqrt{n}\bv_\ell(n))_{\ell\in S}$, and $\by$ has a limit
in Wasserstein-$2$ metric. 
Namely, if we let $\tbv_i = (\sqrt{n}v_{\ell,i})_{\ell\in S}\in\reals^{k_*}$, then there exists a random vector $\bU$ taking values
in $\reals^{k_*}$ and a random variable $Y$, with joint law $\mu_{\bU,Y}$, such that
\begin{align}
\frac{1}{n}\sum_{i=1}^n\delta_{\tbv_i,y_i} \towass \mu_{\bU,Y}\, .
\end{align}
\item\label{ass:Lip}  The functions $f_t(\,\cdot\,,\,\cdot\,):\reals^q\times\reals\to\reals^q$ are Lipschitz continuous.
\end{enumerate}

State evolution operates on the pair of matrices $\bM_t\in\reals^{q\times k_*}$, $\bQ_t\in\reals^{q\times q}$, with $\bQ_t\succeq \bzero$,
evolving according to
\begin{align}
\bM_{t+1} &= \E\Big\{f_t(\bM_t\bU+\bQ_t^{1/2}\bG,Y)\bU^{\sT}\big\} \bLambda_S\, , \label{eq:Mt_def} \\
\bQ_{t+1} & = \E\Big\{f_t(\bM_t\bU+\bQ_t^{1/2}\bG,Y) f_t(\bM_t\bU+\bQ_t^{1/2}\bG,Y)^{\sT}\Big\}\, , \label{eq:Qt_def}
\end{align}
where expectation is taken with respect to $(\bU,Y)\sim\mu_{\bU,Y}$ independent of $\bG\sim\normal(0,\id_q)$.
These recursions are initialized with $\bQ_0, \bM_0$ which will be specified in the statement of  Theorem \ref{thm:Main}  below.

We  denote by  $\cR(\bLambda) \subseteq \reals^{S\times [k]}$ the set of orthogonal matrices $\bR$ (with $\bR\bR^{\sT}= \id_{S}$)
such that $R_{ij}=0$ if $\lambda_i\neq\lambda_j$ or if $j\not \in S$. Notice that the $k_*\times k_*$ submatrix $\bR_{S,S}$ of $\bR\in\cR(\bLambda)$ is a
block-diagonal orthogonal matrix, with blocks in correspondence with the degenerate  $\lambda_i$'s. As such, these matrices form a compact
group, which we will denote by $\cR_*(\bLambda) \subseteq \reals^{k_*\times k_*}$. This group can be endowed with the Haar measure, which is just the product of Haar measures over the orthogonal group corresponding to
each block. We define the Haar measure on $\cR(\bLambda)$ by adding $k-k_*$ columns equal to $0$ for column indices $j\in [k]\setminus S$.
\begin{theorem}\label{thm:Main}
Let $(\bx^t)_{t \geq 0}$ be the AMP iterates generated by algorithm (\ref{eq:AMP}), under assumptions \ref{ass:Spikes} to \ref{ass:Lip},
for the spiked matrix model (\ref{eq:SpikedDef}). For $\eta_n\ge n^{-1/2+\eps}$ such that $\eta_n\to 0$ as  $n\to\infty$, define the set of matrices
\begin{align}
\cG_n(\bLambda)\equiv \Big\{\bQ \in \reals^{S\times [k]}:\, \min_{\bR\in\cR(\bLambda)}\|\bQ- (\id-\bLambda_{S}^{-2})^{1/2}\bR\|_{F}\le \eta_n\Big\}\, ,
\end{align}
Let $\bProd\equiv \bPhi_{\hS}^{\sT}\bV\in\reals^{k_*\times k}$ where $\bPhi_{\hS}\in\reals^{n\times k_*}$ is the matrix with columns $(\bphi_i)_{i\in\hS}$ and $\bV\in \reals^{n\times k}$ is the matrix 
with columns $(\bv_i)_{i\in [k]}$.  Denote by $\bProd_0\in \reals^{S\times S}$ the submatrix corresponding to the  $k_*$ columns of $\bProd$ with index in $S$, and let $\tbProd_0 = (\id-\bProd_0\bProd_0^{\sT})^{1/2}$.
 
Then, for any pseudo-Lipschitz function $\psi:\reals^{q+k_*+1}\to \reals$,  $\psi\in\PL(2)$, the following holds almost surely for $t \geq 0$:
\begin{align}
\lim_{n\to\infty}\left|\frac{1}{n}\sum_{i=1}^n\psi(\bx_i^t,\tbv_i,y_i) - \E\big\{\psi(\bM_t\bU+\bQ_t^{1/2}\bG,\bU,Y)\big\}\right|= 0\, . \label{eq:main_SE_result}
\end{align}
Here $\tbv_i = (\sqrt{n}v_{\ell,i})_{\ell\in S}\in\reals^{k_*}$ and expectation is with respect to $(\bU,Y)\sim \mu_{\bU,Y}$ independent of
$\bG\sim\normal(0,\id_q)$. Finally, $(\bM_t,\bQ_t)$ is the state evolution sequence specified by Eqs. \eqref{eq:Mt_def} and \eqref{eq:Qt_def} with initialization $(\bM_0)_{[k_*],[k_*]}=\bProd_0$, $(\bM_0)_{[q]\setminus [k_*],[k_*]}=\bzero$,
$(\bQ_0)_{[k_*],[k_*]} =\tbProd_0^2$, and $(\bQ_{0})_{i,j} = 0$ if $(i,j)\not \in [k_*]\times [k_*]$.

Further, $\prob(\bOmega\in \cG_n(\bLambda)) \ge 1-n^{-A}$ for any $A>0$ provided $n>n_0(A)$,  and $\bOmega$ converges in distribution to  $(\id-\bLambda_S^{-2})^{1/2}\bR$, with $\bR$ 
Haar distributed on $\cR(\bLambda)$.
\end{theorem}
The  theorem is proved for the  case of a rank one spike in Appendix \ref{sec:ProofMain}. The proof for the general case is given in Appendix \ref{sec:ProofMainGeneral}. In the following section, we provide a brief overview of the key steps in the proof.

\begin{remark}
Theorem \ref{thm:Main}  focuses on the case of symmetric square matrices $\bA$.
However, a standard reduction (see, for instance,  \cite[Section 6]{berthier2017state}) allows to obtain a completely analogous statement for 
rectangular matrices, namely  $\bA\in\reals^{n\times d}$ with
\begin{align}
\bA = \sum_{i=1}^k\lambda_i\bu_i\bv_i^{\sT} + \bW\, ,
\end{align}
where $\bW$ is a noise matrix with independent entries $W_{ij}\sim\normal(0,1/n)$. 
We already  considered the case $k=1$ of this model in Section \ref{sec:ExampleRectangular}. Given Theorem \ref{thm:Main},
the generalization to $k>1$ rectangular matrices is straightforward: we provide a precise statement in Appendix
\ref{app:RectangularGeneral}.

Another generalization of interest would be to non-Gaussian matrices. It might be possible to address this by using the
methods of \cite{bayati2015universality}.
\end{remark}

\section{Proof outline}
\label{sec:ProofOutline}

We first consider the rank one spiked model in Eq. \eqref{eq:SpikedDefSpecial}, and give an outline of the proof of Theorem \ref{thm:Rank1}.  Letting $\bv \equiv \frac{\bx_0}{\sqrt{n}}$,  Eq. \eqref{eq:SpikedDefSpecial} can be written as 
\beq 
\bA = \lambda\bv\bv^{\sT}+\bW. \label{eq:rank1_v} 
\eeq

Recalling that $(\bphi_1, z_1)$ are the principal eigenvector and eigenvalue of $\bA$, we write $\bA$ as the sum of a rank one projection onto the space spanned by $\bphi_1$, plus a matrix that is the restriction of $\bA$ to the subspace orthogonal to $\bphi_1$. That is,
\beq
\bA = z_1 \bphi_1 \bphi_1^{\sT} +  \bPp\left (\lambda \bv \bv^{\sT} + \bW \right)  \bPp\, ,
\eeq
where $\bPp= \bI - \bphi_1 \bphi_1^\sT$ is the projector onto the space orthogonal to $\bphi_1$.  
The proof of Theorem \ref{thm:Rank1} is based on an approximate representation of the conditional distribution of $\bA$ given $(\bphi_1, z_1)$. To this end, we define the matrix
\begin{align}
\tbA = z_1 \bphi_1 \bphi_1^{\sT} +  \bPp\left (\lambda \bv \bv^{\sT} +\tbW \right)  \bPp\, , 
\end{align}
where $\tbW \sim \GOE(n)$ is independent of $\bW$. 

The proof is based on a key technical lemma (Lemma \ref{lemma:Distr}) which shows that for large enough $n$,  the conditional distribution  of $\bA$ given $(\bphi_1, z_1)$ is close in (in total variation distance)  to that of $\tbA$ with high probability.
Given $\tbA$, we consider a sequence of AMP iterates $(\tbx^t)_{t \geq 0}$ obtained by replacing  $\bA$ with $\tbA$ in Eq.~\eqref{eq:AMPspecial0} . That is, we set
\begin{align}
\tbx^0 & = \sqrt{n}\, \sign(\<\bx_0,\bphi_1\>) \, \bphi_1, \,  \qquad 
\tbx^{t+1} = \tbA\, f_t(\tbx^t)-\sb_t f_{t-1}(\tbx^{t-1})\, .\label{eq:AMPspecial_tilde0}
\end{align} 
Theorem \ref{thm:Rank1} is proved in three steps:

\begin{enumerate}
\item Using the conditional distribution lemma (Lemma \ref{lemma:Distr}), we show that for any PL(2) test function $\psi: \reals \times \reals \to \reals$, almost surely
\beq
 \lim_{n \to \infty}  \frac{1}{n} \sum_{i=1}^n \psi (x^t_i, x_{0,i})
=  \lim_{n \to \infty}  \frac{1}{n} \sum_{i=1}^n \psi (\tilde{x}^t_i, x_{0,i}),
\label{eq:AMP_orig_mod}
\eeq
whenever the limit on the right exists. 

\item Step 1 allows us to establish Theorem \ref{thm:Rank1} by analyzing the modified AMP iteration in \myeqref{eq:AMPspecial_tilde0}. For the modified AMP, the initialization $\tbx^0$ is independent of $\tbW$. Consequently, adapting techniques from standard AMP analysis we show that the following holds almost surely for any PL(2) test function $\psi: \reals^3 \to \reals$:
\beq
\lim_{n \to \infty}  \frac{1}{n} \sum_{i=1}^n \psi (\tilde{x}^t_i, x_{0,i}, \sqrt{n} \varphi_{1,i}) = \E \left\{ \psi( \alpha_t  X_0 +  \beta_t L + \tau_t G_0, X_0, L) \right\}. \label{eq:mod_SE0}
\eeq
Here the random variables $(X_0, L, G_0)$ are jointly distributed as follows: $X_0 \sim \nu_{X_0}$ and $G_0 \sim \normal(0,1)$ are independent, and  $L= \sqrt{1-\lambda^{-2}} X_0+ \lambda^{-1} G_1$, where $G_1 \sim \normal(0,1)$ is independent of both $X_0$ and  $G_0$.  It is shown in Corollary \ref{coro:ConvergenceEigenvectors} that (almost surely) the empirical distribution of  
$(\bx_0, \sqrt{n}\bphi_{1})$ converges in $W_2$ to  the distribution of $(X_0, L)$. 
The constants $(\alpha_t, \beta_t, \tau_t)$ in \myeqref{eq:mod_SE0} are iteratively defined using a suitable state evolution 
recursion  given in Eqs. \eqref{eq:AlphaRec}--\eqref{eq:TauRec}.

\item The proof of Theorem \ref{thm:Rank1} is completed by showing that for $t \geq 0$,
\beq
\E \left\{ \psi( \alpha_t  X_0 +  \beta_t L + \tau_t G_0, X_0, L) \right\} = 
\E \left\{ \psi( \mu_t X_0 +  \sigma_t G, X_0 ) \right\},
\label{eq:SE_orig_mod}
\eeq
where $(\mu_t, \sigma_t)_{t \geq 0}$ are the state evolution parameters defined in the statement of Theorem \ref{thm:Rank1}.
\end{enumerate}
Combining Eqs. \eqref{eq:AMP_orig_mod}--\eqref{eq:SE_orig_mod} yields the claim of Theorem \ref{thm:Rank1}. The detailed proof of this theorem is given in Appendix 
\ref{sec:ProofMain}.

\vspace{10pt}
\emph{General case}: For the general spiked model  \myeqref{eq:SpikedDef}, the proof of the  state evolution result (\myeqref{eq:main_SE_result} of Theorem \ref{thm:Main})  is  along similar lines. Here the modified matrix $\tbA$ is defined as
\begin{align}
\tbA &\equiv \sum_{i\in \hS}z_i\bphi_i\bphi_i^{\sT} + \bPp \left(\sum_{i=1}^k\lambda_i \bv_i\bv_i^{\sT} \, + \, \tbW \right)\bPp \,,\label{eq:SpikedModelModified0}
\end{align}
where $\bPp$ is the projector onto the orthogonal complement of the space spanned by $(\bphi_i)_{i\in\hS}$, and $\tbW\sim\GOE(n)$ is independent of $\bW$. (Recall that $\hS$ contains the indices $i$  for which $|\lambda_i| >1$.) Lemma \ref{lemma:Distr} shows that with high probability the conditional distributions of $\bA$ and $\tbA$ are close in total variation distance.  We then consider iterates $(\tbx)_{t \geq 0}$ generated via the AMP iteration using $\tbA$:
\begin{align}
 \tbx^0 & =  \sqrt{n}  \, [\bphi_1|\cdots|\bphi_{k_*}|\bzero|\cdots|\bzero] \,  ,\\
\tbx^{t+1} &= \tbA\,  f_t(\tbx^t,\by) - f_{t-1}(\tbx^{t-1},\by)\, \sB_t^{\sT} \, .\label{eq:AMPspecial_tilde_gen0}
\end{align} 
Using  Lemma \ref{lemma:Distr}, we first show that once the state evolution result   \myeqref{eq:main_SE_result}   holds for  $\tbx^t$, it also holds for  $\bx^t$. The result for  $\tbx^t$ is then shown in two steps, which are analogous to Eqs. \eqref{eq:mod_SE0} and \eqref{eq:SE_orig_mod} for the rank one case. 


\section*{Acknowledgements}

We thank Leo Miolane for pointing out a gap in an earlier proof of Proposition \ref{thm:BayesOpt}.
A.~M. was partially supported by grants NSF CCF-1714305 and NSF
IIS-1741162. R.~V. was partially supported by a Marie Curie Career Integration Grant (Grant Agreement No. 631489).


\appendix

\section{Proof of Theorem \ref{thm:Rank1} and Theorem \ref{thm:Main} in the rank $1$ case}
\label{sec:ProofMain}

In this section we assume $k =q=1$, and hence  write $\bA = \lambda\bv\bv^{\sT}+\bW$ dropping the indices. 
In order for this to be a non-trivial perturbation of the standard GOE model, we will assume $\lambda>1$ (the case $\lambda<-1$ being equivalent). We will prove Theorem  \ref{thm:Rank1} and show that this implies   Theorem \ref{thm:Main} in the rank $1$ case.

For convenient extension from Theorem \ref{thm:Rank1}  to the general statement in Theorem \ref{thm:Main}, in this section we use the notation  $f(\bx;t) \equiv f_t(\bx)$, $\sqrt{n} \bv = \tbv \equiv  \bx_0$ and $U \equiv X_0$.
With this notation, we write the state evolution recursion
 in Eqs. \eqref{eq:MutRecX0}--\eqref{eq:SigtRecX0} as 
\begin{align}
\mu_{t+1} & = \lambda \E[ U f(\mu_t U+\sigma_t G; t)]\, ,  \label{eq:MutRec}\\
 \sigma_{t+1}^2 & = \E [f(\mu_t U + \sigma_t G; t )^2]\, , \label{eq:SigtRec}
\end{align}
 where $U \sim \nu_{X_0}$ and $G \sim \normal(0,1)$ are independent. 

\subsection{Reduction to conditional model}

We will begin by showing that Theorem \ref{thm:Rank1} implies   Theorem \ref{thm:Main} in the rank $1$ case.
\begin{remark}
In this case $\cR(\bLambda) = \cR_*(\bLambda)$ consists of the two $1\times 1$ matrices $\bR = +1$ and $\bR=-1$,
which implies
\begin{align}
\cG_n(\bLambda) = \big\{ q \in\reals:\; \big| |q|-(1-\lambda^{-2})^{1/2}\big|\le \eta_n\big\} \, .
\end{align}
Hence $\Omega=\<\bphi_1,\bv\>\in\cG_n(\bLambda)$ holds with the claimed probability  by  Lemma \ref{lem:ConvergenceEigenvectors}. Further,
conditional on this,  $|\Omega -(1-\lambda^{-2})^{1/2}|\le\eta_n$ and $|\Omega +(1-\lambda^{-2})^{1/2}|\le\eta_n$ each  hold with probability $1/2$ by symmetry.
This implies the weak convergence of $\bProd$ as in the statement.

It remains to prove Eq.~(\ref{eq:main_SE_result}). Let $\cG_n^+(\bLambda) = \cG_n(\bLambda)\cap\{\Omega\ge 0\}$
and  $\cG_n^-(\bLambda) = \cG_n(\bLambda)\cap\{\Omega<0\}$.
For $t\ge 0$, set $\bM_t = \mu_t(n)$, $\bQ_t = \sigma^2_t(n)$ (as these are $1\times 1$ matrices). 
For $\Omega\in\cG_n^+(\bLambda)$, the initialization in the statement of the theorem implies $|\mu_0(n) - \sqrt{1-\lambda^{-2}}|\le C \eta_n$, $|\sigma_0(n)-(1/\lambda)|\le C\eta_n$. 
Since for  any fixed $t$,  $\mu_t(n), \sigma_t(n)$ are continuous in 
the initial condition,  we have $|\mu_t(n) - \mu_t|\le \delta_t(\eta_n)$, $|\sigma_t(n)-\sigma_t|\le \delta_t(\eta_n)$ for some function $\delta_t$
such that $\delta_t(x)\to 0$ as $x\to 0$.
It follows from Theorem \ref{thm:Rank1} that, almost surely
\begin{align}
\lim_{n\to\infty}\left|\frac{1}{n}\sum_{i=1}^n\psi(x_i^t,\tv_i) -
  \E\big\{\psi(\mu_t(n)U+\sigma_t(n)G,U\big)\} \right|\,
 \bfone_{\{\Omega\in \cG^+_n(\bLambda)\} }= 0\, .  \label{eq:1_Gn+}
\end{align}
Considering next $\Omega\in\cG_n^-(\bLambda)$, we can apply Theorem \ref{thm:Rank1} to $\bA = \lambda(-\bv)(-\bv)^{\sT}+\bW$ to get
\begin{align}
\lim_{n \to \infty} \frac{1}{n} \sum_{i=1}^n \tilde{\psi} (x^t_i, -\tv_{i}) =  \E \left\{ \tilde{\psi}(  -\overline{\mu}_t U+\overline{\sigma}_t G, -U )
  \right\} \ ,\label{eq:LimitGminus}
\end{align}
where $(\omu_t, \osigma_t)$ satisfy Eqs.~(\ref{eq:MutRec}), (\ref{eq:SigtRec}) with $U$ replaced by $-U$, and initial condition
$\omu_0 = \sqrt{1-\lambda^{-2}}$,$\osigma_0 = 1/\lambda$.
It is easy to check that $(-\omu_t,\osigma_t)$ satisfies Eqs.~(\ref{eq:MutRec}), (\ref{eq:SigtRec}) with initial condition $\omu_0 = -\sqrt{1-\lambda^{-2}}$.
Since   $\Omega\in\cG_n^-(\bLambda)$, we have $|\mu_0(n) +\sqrt{1-\lambda^{-2}}|\le C \eta_n$, $|\sigma_0(n)-(1/\lambda)|\le C\eta_n$.
Again by continuity of state evolution in the initial condition,
we have $|\mu_t(n) + \omu_t|\le \delta_t(\eta_n)$, $|\sigma_t(n)-\osigma_t|\le \delta_t(\eta_n)$ for some function $\delta_t$
such that $\delta_t(x)\to 0$ as $x\to 0$. Therefore, by using $\tilde{\psi}(x,y) = \psi(x,-y)$, Eq.~(\ref{eq:LimitGminus}) implies
\begin{align}
\lim_{n\to\infty}\left|\frac{1}{n}\sum_{i=1}^n\psi(x_i^t,\tv_i) -
  \E\big\{\psi(\mu_t(n)U+\sigma_t(n)G,U\big)\} \right|\,
 \bfone_{\{\Omega\in \cG^-_n(\bLambda)\} }= 0\, .  \label{eq:1_Gn-}
\end{align}
The claim in Theorem \ref{thm:Main} then follows by from Eqs.~(\ref{eq:1_Gn+}) and (\ref{eq:1_Gn-}), using the fact that $\Omega\in\cG_n(\bLambda)$ eventually almost surely.
\end{remark}

The proof of Theorem \ref{thm:Rank1} is based on an approximate representation for the conditional distribution of 
$\bA$ given $(\bphi_1, z_1)$, that is established in Lemma \ref{lemma:Distr} below. Namely, we introduce the matrix
\begin{align}
\tbA = z_1 \bphi_1 \bphi_1^{\sT} +  \bPp\left (\lambda \bv \bv^{\sT} +\tbW \right)  \bPp\, , \label{eq:tbadef}
\end{align}
 where $(\bphi_1, z_1)$ are the principal eigenvector and eigenvalue of $\bA$ in \myeqref{eq:SpikedDefSpecial}, $\tbW \sim \GOE(n)$ is independent of $\bW$, and the matrix $\bPp= \bI - \bphi_1 \bphi_1^\sT$ is the projector onto the space orthogonal to $\bphi_1$. 
The bulk of our work consists in analyzing this simplified model, as per the next lemma, which is proved in the next section.
\begin{lemma}\label{lemma:SE-rank1-simple}
Consider  the modified spiked model \eqref{eq:tbadef} and let $\tbx^t$ be the AMP sequence obtained by replacing $\bA$ with $\tbA$ in
Eq.~\eqref{eq:AMPspecial0} . Namely, we set
\begin{align}
\tbx^0 & = \sqrt{n}\, \sign(\<\bv,\bphi_1\>)\bphi_1, \,  \qquad 
\tbx^{t+1} = \tbA\, f(\tbx^t; t)-\sb_t f(\tbx^{t-1}; t-1)\, .\label{eq:AMPspecial_tilde}
\end{align} 
Then the state evolution statement, Eq.~(\ref{eq:rank1_SE}), holds with $\bx^t$ replaced by $\tbx^t$.
\end{lemma}

\begin{proof}[Proof of Theorem \ref{thm:Rank1}]
For any $\eps \in (0, \eps_0)$, Lemma \ref{lemma:Distr} bounds the total variation distance between the  conditional joint distributions of $(\tbA,\bphi_1)$ and $(\bA,\bphi_1)$ given $(z_1, \bphi_1) \in \cE_{\eps}$, where from \myeqref{eq:cE_eps} 
\beq
\cE_{\eps} = \left\{ \abs{z_1 - (\lambda + \lambda^{-1})} \leq \eps, \quad (\bphi_1^{\sT}\bv)^2 \geq 1- \lambda^{-2} - \eps  \right\}.
\eeq
 Since $\tbx^t$ and $\bx^t$ are obtained by applying the same deterministic algorithm to  $(\tbA,\bphi_1)$ and $(\bA,\bphi_1)$
it follows that   there exists a coupling of the laws of $\bA$ and $\tbA$ such that, for $(z_1,\bphi_1)\in\cE_{\eps}$
\begin{align}
\prob\left\{\sum_{i=1}^n\psi(x_i^t,\tv_i) \neq \sum_{i=1}^n\psi(\tx_i^t,\tv_i) \, \Big \vert \, z_1, \bphi_1 \right\}\, \le \frac{1}{c(\eps)}\, e^{-nc(\eps)} \,   \label{eq:rank1_coup1}
\end{align}
for some constant $c(\eps) >0$. With this coupling, we therefore have
\begin{align}
\prob\left\{\sum_{i=1}^n\psi(x_i^t,\tv_i)  \neq \sum_{i=1}^n\psi(\tx_i^t,\tv_i) \right\} &  \leq 
\prob( \cE_{\eps}^c) +  \frac{e^{-nc(\eps)}}{c(\eps)} \, \leq   \frac{2 e^{-nc(\eps)}}{c(\eps)},
 \label{eq:rank1_coup2}
\end{align}
where the last inequality is  obtained using \myeqref{eq:SpikeEvent} of Lemma \ref{lemma:Distr}.
Therefore by Borel-Cantelli, $\sum_{i=1}^n\psi(x_i^t,\tv_i) =\sum_{i=1}^n\psi(\tx_i^t,\tv_i)$ eventually  almost surely. 
Theorem \ref{thm:Rank1} hence follows by applying Lemma \ref{lemma:SE-rank1-simple}.
\end{proof}

\subsection{Proof of Lemma \ref{lemma:SE-rank1-simple}}
\label{sec:SE-rank1-simple}

In this section we analyze the simplified recursion \myeqref{eq:AMPspecial_tilde}, that uses the conditional model  \eqref{eq:tbadef}. 
Since there is no possibility of confusion, we will drop the tilde and write $\bx^t$ instead of $\tbx^t$. Recall that $\bx_0 = \sqrt{n}\,\bv$, and for two equal-length vectors $\bx, \by$, we write $\< \bx, \by\>$ for the Euclidean inner product $\bx^{\sT}\by$.

To simplify notation, we will assume that $\< \bv, \bphi_1 \> \geq 0$. The proof for the case  $\< \bv, \bphi_1 \> \leq 0$ is identical except for a sign change in the definition in \myeqref{eq:Ldef}.

From \myeqref{eq:tbadef} and \myeqref{eq:AMPspecial_tilde}, we have
\begin{align} 
 \bx^{t+1}  & = z_1  \< \bphi_1, f(\bx^t; t) \> \bphi_1 + \lambda  \bPp \bv \bv^{\sT} \bPp f(\bx^t; t) + \bPp \tbW \bPp f(\bx^t; t) 
 -  \sb_t  f(\bx^{t-1}; t-1) \\
& = z_1  \< \bphi_1, f(\bx^t; t) \> \bphi_1 +  \lambda  \bPp \bv \left[ \< \bv, f(\bx^t; t)  \> - \< \bv, \bphi_1 \>  \< \bphi_1, f(\bx^t; t) \> \right]   \nonumber \\
&\qquad +\bPp \tbW \left[ f(\bx^t; t) -    \< \bphi_1, f(\bx^t; t) \> \bphi_1 \right]     -  \sb_t  f(\bx^{t-1}; t-1) \nonumber \\
& = \left[ z_1  \< \bphi_1, f(\bx^t; t) \>  - \sb_t \< \bphi_1, f(\bx^{t-1}; t-1) \> \right] \bphi_1  +  \lambda \< \bPp\bv, f(\bx^t; t)  \> \,  \bPp \bv   \nonumber \\
& \qquad  +   \tbW  \left[ f(\bx^t; t)  - \< \bphi_1, f(\bx^t; t) \> \bphi_1 \right]     - \sb_t \left[ f(\bx^{t-1}; t-1) - \< \bphi_1, f(\bx^{t-1}; t-1)  \bphi_1  \right] \nonumber \\
& \qquad \qquad - \, 
\bphi_1 \bphi_1^{\sT} \tbW  \left[ f(\bx^t; t)  - \< \bphi_1, f(\bx^t; t) \> \bphi_1 \right], \label{eq:xt1_full}
\end{align}
where we have used $\bPp = \bI - \bphi_1\bphi_1^{\sT}$ to obtain \eqref{eq:xt1_full}.
Defining 
\begin{align}
& g(\bx^t; t)  = f(\bx^t; t) - \< \bphi_1, f(\bx^t; t) \> \bphi_1,  \label{eq:gtdef} \\
& \bdel^t   = \bphi_1 \bphi_1^{\sT} \tbW  \left[ f(\bx^t; t)  - \< \bphi_1, f(\bx^t; t) \> \bphi_1 \right],   \label{eq:bdeldef}
\end{align}
we can write \myeqref{eq:xt1_full} as 
\begin{align} 
 \bx^{t+1}  & 
= \left[ z_1  \< \bphi_1, f(\bx^t; t) \>  - \sb_t \< \bphi_1, f(\bx^{t-1}; t-1) \> -\lambda \<\bv,\bphi_1\>\< \bPp\bv, f(\bx^t; t)  \>
\right] \bphi_1  +  \lambda  \< \bPp\bv, f(\bx^t; t)  \> \bv  \nonumber \\
& \qquad  +   \tbW  g(\bx^t; t)  - \sb_t g(\bx^{t-1}; t-1)\, - \bdel^t. \label{eq:xt1_exp}
\end{align}

Note that (almost surely) the empirical distribution of  $(\sqrt{n} \bv, \sqrt{n}\bphi_{1})$ converges in $W_2$ to  the distribution of $(U,
L)$, where  $U\sim \nu_{X_0}$ and  
\beq
L = \sqrt{1-\lambda^{-2}} U+ \lambda^{-1} G_1 \label{eq:Ldef} 
\eeq
with  $G_1\sim\normal(0,1)$ independent of $U$, see Corollary \ref{coro:ConvergenceEigenvectors}.

We define
\begin{align} 
\tilde{X}_t  \equiv \alpha_t U+\beta_t L+\tau_t G_0, \label{eq:tilXt_def} 
\end{align}  
where $G_0\sim\normal(0,1)$ is independent of $(U,L)$ and the constants $(\alpha_t, \beta_t, \tau_t)$  are defined via the following recursion. Starting with 
$\beta_0 = 1, \alpha_0= \tau_0=0$, so that $\tilde{X}_0= L$,  we compute for $t\ge 0$
\begin{align}
\alpha_{t+1} & = \lambda\left\{\E\{U\, f(\tilde{X}_t;t)\}- \sqrt{1-\frac{1}{\lambda^2}} \E\{L\, f(\tilde{X}_t;t)\}\right\}\, ,   \label{eq:AlphaRec} \\
\beta_{t+1} & = 2\lambda  \E \{ L\, f(\tilde{X}_t;t)\} - \E \left\{ f'(\tilde{X}_t; t) \right\} \E\{L\, f(\tilde{X}_{t-1};t-1)\} -\sqrt{\lambda^2-1} \, \E\{U\, f(\tilde{X}_t;t)\} \label{eq:BetaRec} \\
\tau_{t+1}^2 & = \E\{ f(\tilde{X}_t;t)^2\} - \left( \E\{L\, f(\tilde{X}_t;t)\} \right)^2\, .\label{eq:TauRec}
\end{align}
Here and below, we assume the convention  $\E\{L\, f(\tilde{X}_{-1}; \, -1)\}  =1/\lambda$.

We will prove \myeqref{eq:rank1_SE} in two steps.  We  show that almost surely 
\begin{align}
\lim_{n \to \infty}  \frac{1}{n} \sum_{i=1}^n \psi (x^t_i, x_{0,i}, \sqrt{n} \varphi_{1,i}) = \E \left\{ \psi( \alpha_t  U +  \beta_t L + \tau_t G_0, U, L) \right\}.
\label{eq:SE_conv1}
\end{align}
We then claim that 
\begin{align}
 \E \left\{ \psi( \alpha_t  U +  \beta_t L + \tau_t G_0, U ) \right\} = \E \left\{ \psi( \mu_t U +  \sigma_t G, U ) \right\}.  \label{eq:SE_conv2}
\end{align}

\subsection*{ Proof of \myeqref{eq:SE_conv2}}
For $(\alpha_t, \beta_t, \tau_t^2)$ defined via the recursion in  Eqs.~\eqref{eq:AlphaRec} -- \eqref{eq:TauRec}, we show below that for $t \geq 0$
\begin{align}
\beta_{t+1} = \lambda \E \{ L f(\tilde{X}_t; t) \}. \label{eq:betat1_prop}
\end{align}
 Using \myeqref{eq:betat1_prop},  we observe that the recursion in  Eqs. \eqref{eq:AlphaRec} -- \eqref{eq:TauRec} is equivalent to the recursion in Eqs. \eqref{eq:MutRec} -- \eqref{eq:SigtRec} if we set
\begin{align}
\mu_t & \equiv \alpha_t + \beta_t \sqrt{1- \lambda^{-2}}, \label{eq:mut_def0} \\
\sigma_t^2 & \equiv \beta_t^2 \lambda^{-2} + \tau_t^2 \label{eq:sigt_def0}.
\end{align}

Recalling that  $L = \sqrt{1-\lambda^{-2}} U + \lambda^{-1} G_1$, we have 
\begin{align}
 \E \left\{ \psi( \alpha_t  U +  \beta_t L + \tau_t G_0, U ) \right\}  
 = \E \left\{ \psi \left( (\alpha_t + \beta_t \sqrt{1- \lambda^{-2}} ) U +  \beta_t \lambda^{-1}G_1 + \tau_t G_0, U \right) \right\}.
\end{align}
Since $U \sim \mu_U$, $G_0 \sim \normal(0,1)$, and $G_1 \sim \normal(0,1)$ are independent,  we use  \myeqref{eq:mut_def0} and \myeqref{eq:sigt_def0} to observe that $(\alpha_t + \beta_t \sqrt{1- \lambda^{-2}} )U = \mu_t U$, and 
$ \beta_t \lambda^{-1}G_1 + \tau_t G_0 \ed\sigma_t Z_0$.  We finally show  \myeqref{eq:betat1_prop}.

\emph{Proof of \myeqref{eq:betat1_prop}}: Using the definition of $\beta_{t+1}$ in \myeqref{eq:BetaRec}, it suffices to show that, for $t\ge 0$, 
\begin{align}
\lambda \E \{ L f(\tilde{X}_t; t) \} = 
\E \{ f'(\tilde{X}_t; t) \} \E\{L\, f(\tilde{X}_{t-1};t-1) \} + \sqrt{\lambda^2-1} \, \E\{U\, f(\tilde{X}_t;t)\}.
 \label{eq:ind_claim}
\end{align}
We prove Eqs.~\eqref{eq:betat1_prop} and \eqref{eq:ind_claim} inductively. 
 
For $t=0$, using the definition of $L$ in  \myeqref{eq:Ldef} we write the LHS of \myeqref{eq:ind_claim} as
\begin{align}
\lambda \E \{ L f(\tilde{X}_0; 0) \} &  = \lambda \left[ \E \{ \lambda^{-1} G_1  f(\tilde{X}_0; 0)  \} 
+  \E \{ \sqrt{1- \lambda^{-2}} U f(\tilde{X}_0; 0) \} \right]  \nonumber \\ 
& =  \E \{ G_1  f(\tilde{X}_0; 0) \} +  \sqrt{\lambda^2-1} \, \E\{U\, f(\tilde{X}_0;0)\}, \nonumber \\
& \stackrel{(a)}{=} \frac{1}{\lambda}\E \{  f'(\tilde{X}_0; 0) \} +  \sqrt{\lambda^2-1} \, \E\{U\, f(\tilde{X}_0;0)\},
\end{align}
where the last equality $(a)$ is obtained  by noting that 
$\tilde{X}_0 =  L = \sqrt{1-\lambda^{-2}}U + \, \lambda^{-1}G_1$, and  then applying Stein's lemma (Gaussian integration by parts).  Thus \myeqref{eq:ind_claim} holds for $t=0$ since $ \E \{ L f(\tilde{X}_{-1}; -1) \} =1/\lambda$. 

Assume towards induction that  Eqs.~\eqref{eq:betat1_prop} and \eqref{eq:ind_claim} holds for $t=0, \ldots, (r-1)$.  For $t=r$, we have
\begin{align}
\lambda \E \{ L f(\tilde{X}_{r }; r) \} &= \E \{ G_1  f(\tilde{X}_r; r) \} +  \sqrt{\lambda^2-1} \, \E\{ U\, f(\tilde{X}_r;r) \}, \nonumber \\ 
& \stackrel{(a)}{=}   \lambda^{-1} \beta_r \E \{ f'(\tilde{X}_r; r) \} +  \sqrt{\lambda^2-1} \, \E\{U\, f(\tilde{X}_r;r)\},  \nonumber \\
& \stackrel{(b)}{=}  \E \{ L f(\tilde{X}_{r-1}; r-1) \} \E \{ f'(\tilde{X}_r; r) \} +  \sqrt{\lambda^2-1} \, \E\{U\, f(\tilde{X}_r;r)\}
\end{align}
where $(a)$ is obtained by noting that $\tilde{X}_r  = \alpha_r U+\sqrt{1- \lambda^{-2}} \beta_r U +  \lambda^{-1} \beta_r
G_1+\tau_r G_0$, and then applying Stein's lemma; step $(b)$ follows from the induction hypothesis \myeqref{eq:betat1_prop} for $t=r-1$. This proves  \myeqref{eq:betat1_prop}, thus completing the proof of \myeqref{eq:SE_conv2}.

\subsection*{ Proof of \myeqref{eq:SE_conv1}}
Let 
\begin{align}
\tilg(\bx^t; t) & = f(\bx^t; t) - \E \{ L f(\alpha_t U + \beta_t L +\tau_t G_0;t) \}  \sqrt{n} \bphi_1. \label{eq:tilg_def}
\end{align}
Note that $\tilde{g}(\bx^t; t)$ is a separable function obtained by replacing the scalar coefficient of $\bphi_1$  in \myeqref{eq:gtdef} by a deterministic value.

Define a related iteration to generate $(\bfs^t)_{t \geq 0}$ as follows.
\begin{align}
\bfs^{t+1}&  =  \tbW \tilg(\bfs^t + \alpha_t \bx_0  + \beta_t  \sqrt{n} \bphi_1; t ) - \tilde{\sb}_t \tilg (\bfs^{t-1} +  \alpha_{t-1} \bx_0  + \beta_{t-1} \sqrt{n} \bphi_1 ; t-1), \label{eq:st1_def} \\ 
\tilde{\sb}_t &= \frac{1}{n} \sum_{i =1}^n f'(s^t_i+ \alpha_t x_{0,i} + \beta_t  \sqrt{n} \varphi_{1,i} ; t).
\end{align}
The iteration is initialized with 
\begin{align} 
\bfs^0 =  \bx^0 - \alpha_0 \bx_0 - \beta_0 \sqrt{n} \bphi_1 =\mathbf{0},  
\label{eq:bfs0_def} 
\end{align}
where the last equality holds because $\bx^0 =\sqrt{n} \bphi_1$, $\alpha_0=0$, and $\beta_0=1$. 

Noting that: $(i)$ the empirical distribution of $(\bx_0, \sqrt{n} \bphi_1)$ converges in $W_2$ to the distribution of $(U,L)$, and $(ii)$ 
the iteration for $\bfs^t$ is of the standard AMP form in \cite{javanmard2013state}, for any pseudo-Lipschitz function $\tilde{\psi}$ we have almost surely:
\begin{align}
\lim_{n \to \infty}  \frac{1}{n} \sum_{i=1}^n \tilde{\psi} (s^t_i, x_{0,i}, \sqrt{n} \varphi_{1,i}) = 
\E \left\{ \tilde{\psi} ( \tau_t G_0, U, L ) \right\}. \label{eq:st_SE}
\end{align}
where  $\tau_t$ is determined by the  recursion:
\begin{align}
\tau_{t+1}^2 =   \E\{ f( \alpha_t  U +  \beta_t L + \tau_t G_0 ;t)^2\}- \left( \E\{L\, f(\alpha_t  U +  \beta_t L + \tau_t G_0;t)\} \right)^2,
\end{align}
initialized with $\tau_0=0$. Note that this expression for $\tau_{t+1}^2$ matches with that in \myeqref{eq:TauRec}.

Now, choosing $\tilde{\psi} (u, v, z) = \psi(u + \alpha_t v+ \beta_t z, v, z)$ for a pseudo-Lipschitz function $\psi$, \myeqref{eq:st_SE} implies that almost surely
\begin{align}
\lim_{n \to \infty}  \frac{1}{n} \sum_{i=1}^n \psi (s^t_i + \alpha_t x_{0,i} +  \beta_t \sqrt{n} \varphi_{1,i}, x_{0,i}, \sqrt{n} \varphi_{1,i}) = 
\E \left\{ \psi ( \tau_t G_0 + \alpha_t U + \beta_t L, U, L ) \right\}.  \label{eq:st_SE0}
\end{align}
Therefore to prove \myeqref{eq:SE_conv1} it suffices to show that almost surely
\begin{align}
\lim_{n \to \infty}  \left[ \frac{1}{n} \sum_{i=1}^n \psi (x^t_i, x_{0,i}, \sqrt{n} \varphi_{1,i})  -  \frac{1}{n} \sum_{i=1}^n 
\psi (s^t_i + \alpha_t x_{0,i} +  \beta_t \sqrt{n} \varphi_{1,i}, x_{0,i}, \sqrt{n} \varphi_{1,i})  \right] =0. 
\label{eq:psi_xs}
\end{align}
We define
\begin{align}
\bDelta^t = \bx^t - \left( \bfs^t + \alpha_t \bx_0 + \beta_t \sqrt{n} \bphi_1 \right),  \label{eq:bDel_def}
\end{align}
and inductively prove \myeqref{eq:psi_xs} together with the following claims:
\begin{align}
 & \lim_{n \to \infty} \frac{1}{n} \norm{\bDelta^t}^2  = 0, \label{eq:norm_Del} \\
&  \limsup_{n \to \infty} \frac{1}{n} \norm{\bx^t}^2 < \infty, \qquad    
 \limsup_{n \to \infty} \frac{1}{n} \norm{\bfs^t +  \alpha_t \bx_0 + \beta_t \sqrt{n} \bphi_1 }^2 < \infty. \label{eq:norm_finite}
\end{align}
The base case of $t=0$ is easy to verify. Indeed, from the definition of $\bfs^0$ in \myeqref{eq:bfs0_def},  we have $\bDelta^0 = \mbf{0}$
and the equality in \myeqref{eq:psi_xs} holds. Furthermore, since $\bx^0= \sqrt{n}\bphi_1$, we have
$\norm{\bx^0}^2/n = 1$.  

With the induction hypothesis that Eqs.~\eqref{eq:psi_xs} -- \eqref{eq:norm_finite} hold for $t=0, 1, \ldots, r$, we now prove the claim for $t=r+1$.   By the pseudo-Lipschitz property  of $\psi$, for $i \in [n]$ and some constant $C$  we have:
\begin{align}
& \abs{ \psi (x^r_i, x_{0,i}, \sqrt{n} \varphi_{1,i})  -\psi (s^r_i + \alpha_t x_{0,i} +  \beta_r \sqrt{n} \varphi_{1,i}, x_{0,i}, \sqrt{n} \varphi_{1,i}) }  \nonumber \\
& \leq C \abs{\Delta^r_i} \left( 1 +  \abs{x_{0,i}} + \abs{x^r_i} + \abs{s_i^r + \alpha_r x_{0,i} + \beta_r \varphi_{1,i}} \right). \label{eq:psiPL00}
\end{align}
(In what follows we use $C >0$ to denote a generic absolute constant whose value may change as we progress though the proof.)

From \myeqref{eq:psiPL00}, we have
\begin{align}
& \frac{1}{n} \abs{ \sum_{i=1}^n  \left[  \psi (x^r_i, x_{0,i}, \sqrt{n} \varphi_{1,i}) - \psi (s^t_i + \alpha_r x_{0,i} +  \beta_r \sqrt{n} \varphi_{1,i}, x_{0,i}, \sqrt{n} \varphi_{1,i}) \right]} \nonumber \\
& \leq  C \left[ \frac{1}{n} \sum_{i=1}^n \abs{\Delta^r_i}   + \frac{1}{n} \sum_{i=1}^n \abs{\Delta^r_i} \abs{x_{0,i}}   + \frac{1}{n} \sum_{i=1}^n \abs{\Delta^r_i} \abs{x^r_i} + \frac{1}{n} \sum_{i=1}^n \abs{\Delta^r_i}\abs{s_i^r + \alpha_r x_{0,i} + \beta_r \phi_{1,i}}  \right] \nonumber \\
& \leq C \left[ \frac{\norm{\bDelta^r}}{\sqrt{n}} + \frac{\norm{\bDelta^r}}{\sqrt{n}} \frac{\norm{\bx^r}}{\sqrt{n}} 
+ \frac{\norm{\bDelta^r}}{\sqrt{n}} \frac{\norm{\bfs^r + \alpha_r \bx_0 + \beta_r \bphi_1}}{\sqrt{n}} \right]
\end{align}

Substituting the expressions for $\bx^{r+1}$ and $\bfs^{r+1}$ from \myeqref{eq:xt1_exp} and  \myeqref{eq:st1_def} into definition of $\bDelta^{r+1}$ from \myeqref{eq:bDel_def},  and recalling that $\bx_0 = \sqrt{n} \bv$, we get 
\begin{align}
\bDelta^{r+1}  &=   \left[ z_1  \frac{\<  \sqrt{n} \bphi_1, f(\bx^r; r) \>}{n}  - \sb_r  \frac{\< \sqrt{n}\bphi_1, f(\bx^{r-1}; r-1) \> }{n}
-\lambda \<\bv,\bphi_1\> \frac{\< \bPp \bx_0, f(\bx^r; r)  \>}{n}  - \beta_{r+1} \right] \sqrt{n} \bphi_1   \nonumber \\
&  +  \left[ \lambda  \frac{\< \bPp \bx_0, f(\bx^r; r) \>}{n} - \alpha_{r+1} \right] \bx_0 +   \tbW  \left[ g(\bx^r; r)  - \tilg(\bfs^r + \alpha_r \bx_0  + \beta_r  \sqrt{n} \bphi_1; r ) \right]  \nonumber \\
& + \tilde{\sb}_r \tilg (\bfs^{r-1} +  \alpha_{r-1} \bx_0  + \beta_{r-1} \sqrt{n}\bphi_1) - \sb_r g(\bx^{r-1}; r-1) \, - \bdel^r.
\end{align}
Note that $\norm{\bx_0}^2/n = \norm{\bphi_1}^2 =1$. We show that $\norm{\bDelta^{r+1}}^2/n  \to 0$ almost surely by proving that the following limits hold almost surely:
\begin{align}
& \lim_{n \to \infty} \left[ z_1  \frac{\<  \sqrt{n} \bphi_1, f(\bx^r; r) \>}{n}  - \sb_r  \frac{\< \sqrt{n}\bphi_1, f(\bx^{r-1}; r-1) \> }{n}
-\lambda \<\bv,\bphi_1\> \frac{\< \bPp \bx_0, f(\bx^r; r)  \>}{n} \right] = \beta_r,   \label{eq:lim_betat}\\
& \lim_{n \to \infty}  \lambda  \frac{\< \bPp \bx_0, f(\bx^r; r) \>}{n}  = \alpha_r,  \label{eq:lim_alpht} \\
&  \lim_{n \to \infty} \frac{1}{n} \norm{\tbW  \left[ g(\bx^r; r)  - \tilg(\bfs^r + \alpha_r \bx_0  + \beta_r  \sqrt{n} \bphi_1; r ) \right]}^2 = 0, \label{eq:Wg_diff_norm} \\
& \lim_{n \to \infty}  \frac{1}{n} \norm{\tilde{\sb}_r \tilg (\bfs^{r-1} +  \alpha_{r-1} \bx_0  + \beta_{r-1}; \, r-1)  - \sb_r g(\bx^{r-1}; r-1)}^2 =0,
\label{eq:btg_diff_norm} \\
& \lim_{n \to \infty} \frac{1}{n} \norm{\bdel^r}^2 =0. \label{eq:bdel_norm}
\end{align}

\emph{Proof of \myeqref{eq:lim_betat}}: 
From standard results on spiked random matrices \cite{baik2005phase,benaych2012singular},  we know that almost surely,
\begin{align}
\lim_ {n \to \infty} z_1 =  \lambda + \frac{1}{\lambda}, \qquad \lim_{n \to \infty} |\<\bv,\bphi_1\>| = \sqrt{1- \frac{1}{\lambda^2}}. \label{eq:spiked_rm_lim}
\end{align}
Consider the function $\psi(u,v,z) = z f(u; t)$. Since $f(\cdot; t)$ is Lipschitz, it is easy to check that $\psi$ is pseudo-Lipschitz. Therefore, by the induction hypothesis,  using \myeqref{eq:psi_xs}  and \myeqref{eq:st_SE0} with $t=r$ and $t=(r-1)$, we have
\begin{align}
 \lim_{n \to \infty} \frac{\<  \sqrt{n} \bphi_1, f(\bx^r; r) \>}{n}  & = \E \left\{ L f( \tau_r G_0 + \alpha_r U + \beta_r L; \, r) \right\} \quad \text{a.s.}  \label{eq:phi1_f_lim}\\
\lim_{n \to \infty} \frac{\< \sqrt{n}\bphi_1, f(\bx^{r-1}; r-1) \> }{n}  & = \E \left\{ L f( \tau_{r-1} G_0 + \alpha_{r-1} U + \beta_{r-1} L; \, r-1) \right\} \quad \text{a.s.}  \label{eq:phi1_f_limB}
 \end{align}
Next, consider the term
\[ 
\frac{\< \bPp \bx_0, f(\bx^r; r)  \>}{n} = \frac{\< \bx_0, f(\bx^r; r)  \>}{n} 
 - \< \bv , \bphi_1 \> \frac{\< \sqrt{n} \bphi_1, f(\bx^r; r)  \>}{n}.   
 \]
 Using the induction hypothesis and considering the pseudo-Lipschitz function $\psi(u,v,z) = v f(u;r)$, we have from \myeqref{eq:psi_xs}  and \myeqref{eq:st_SE}:
 \begin{align}
 \lim_{n \to \infty} \frac{\< \bx_0, f(\bx^r; r)  \>}{n}  = \E \left\{ U f( \tau_{r} G_0 + \alpha_{r} U + \beta_{r} L;r) \right\} \quad \text{a.s.}
 \end{align}
 Using this together with \myeqref{eq:phi1_f_lim} and \myeqref{eq:spiked_rm_lim}, we get
 \begin{align}
& \lim_{n \to \infty} \frac{\< \bPp \bx_0, f(\bx^r; r)  \>}{n}  \nonumber  \\
& =  \E \left\{ U f( \tau_{r} G_0 + \alpha_{r} U + \beta_{r} L) \right\}
 - \sqrt{1- \frac{1}{\lambda^2}} \E \left\{ L f( \tau_r G_0 + \alpha_r U + \beta_r L; \, r) \right\} \ \text{a.s.} \label{eq:Px0perp_f}
 \end{align}
 Next consider
 \[  \sb_r = \frac{1}{n} \sum_{i=1}^n f'(x^r_i; r). \]
The induction hypothesis implies that the empirical distribution of $\bx^r$ converges weakly to the distribution of 
 $\alpha_r U + \beta_r L + \tau_r G_0$. Combining this with the Lipschitz property of $f( \cdot; r)$,  from  \cite[Lemma 5]{BM-MPCS-2011} we have
 \begin{align}
 \lim_{n \to \infty} \sb_r =  \E\{ f'(\alpha_r U + \beta_r L + \tau_r G_0; r) \}  \ \text{a.s.}  \label{eq:fpr_conv}
 \end{align}
 
 Finally, combining the results in \myeqref{eq:spiked_rm_lim} -- \myeqref{eq:fpr_conv}, we obtain
 \begin{align}
  & \lim_{n \to \infty}  \left[ z_1  \frac{\<  \sqrt{n} \bphi_1, f(\bx^r; r) \>}{n}  - \sb_r  \frac{\< \sqrt{n}\bphi_1, f(\bx^{r-1}; r-1) \> }{n}
-\lambda \<\bv,\bphi_1\> \frac{\< \bPp \bx_0, f(\bx^r; r)  \>}{n} \right] \nonumber \\
& = (\lambda + \lambda^{-1})  \E \left\{ L f( \tau_{r} G_0 + \alpha_{r} U + \beta_{r} L;r) \right\}  \nonumber \\
& \quad - \E\{ f'(\alpha_r U + \beta_r L + \tau_r G_0; \, r) \} \E \left\{ L f( \tau_{r-1} G_0 + \alpha_{r-1} U + \beta_{r-1} L; \, r-1) \right\}  \nonumber \\
& \quad   - \sqrt{\lambda^2-1}\left[ \E \left\{ U f( \tau_{r} G_0 + \alpha_{r} U + \beta_{r} L;\, r) \right\} 
 - \sqrt{1- \frac{1}{\lambda^2}} \E \left\{ L f( \tau_r G_0 + \alpha_r U+ \beta_r L; \, r) \right\}  \right] \nonumber \\ 
 & = 2\lambda  \E \left\{ L f( \tau_{r} G_0 + \alpha_{r} U + \beta_{r} L;\, r) \right\}  - \sqrt{\lambda^2-1} \,  \E \left\{ U f( \tau_{r} G_0 + \alpha_{r} U + \beta_{r} L;\, r) \right\} \nonumber \\
& \quad - \E\{ f'(\alpha_r U + \beta_r L + \tau_r G_0;\, r) \} \E \left\{ L f( \tau_{r-1} G_0 + \alpha_{r-1} U + \beta_{r-1} L; \, r-1) \right\}  \nonumber  \\
& = \beta_{r+1},  \label{eq:phi1_coeff}
 \end{align}
where the last inequality follows from the definition in \myeqref{eq:BetaRec}.

\emph{Proof of \myeqref{eq:lim_alpht}}: From \myeqref{eq:Px0perp_f}, we have almost surely
 \begin{align}
& \lim_{n \to \infty} \lambda \frac{\< \bPp \bx_0, f(\bx^r; r)  \>}{n}  \nonumber  \\
& =  \lambda \left[  \E \left\{ U f( \tau_{r} G_0 + \alpha_{r} U + \beta_{r} L;\, r) \right\}
 - \sqrt{1- \frac{1}{\lambda^2}} \E \left\{ L f( \tau_r G_0 + \alpha_r U+ \beta_r L; \, r) \right\} \right] = \alpha_{r+1}, \label{eq:lamb_Px0perp_f}
 \end{align}
where the last inequality follows from the definition in \myeqref{eq:AlphaRec}.

\emph{Proof of \myeqref{eq:Wg_diff_norm}}:  We have
\begin{align}
& \frac{1}{n} \norm{\tbW  \left[ g(\bx^r; r)  - \tilg(\bfs^r + \alpha_r \bx_0  + \beta_r  \sqrt{n} \bphi_1; r ) \right]}^2 \nonumber  \\
& \leq   { \| \tbW \| }_{\op}^2  \,  \frac{1}{n}\norm{g(\bx^r; r)  - \tilg(\bfs^r + \alpha_r \bx_0  + \beta_r  \sqrt{n} \bphi_1; r )}^2. \label{eq:tWg_CS}
\end{align}
Since $\tbW \sim\GOE(n)$, we know that \cite{AGZ} almost surely
\begin{align}
\lim_{n \to \infty} { \| \tbW \| }_{\op}^2 =4.
\label{eq:normtbW_lim}
\end{align}
Using the definitions  of the functions $g$ and $\tilg$ from \myeqref{eq:gtdef} and \myeqref{eq:tilg_def}, we write 
\begin{align}
 &  \frac{1}{n} \norm{g(\bx^r; r) - \tilg(\bx^r; r) + \tilg(\bx^r; r)   - \tilg(\bfs^r + \alpha_r \bx_0  + \beta_r  \sqrt{n} \bphi_1; r )}^2   \nonumber \\
 & \leq \frac{2}{n}\left[  \norm{g(\bx^r; r) - \tilg(\bx^r; r) }^2 +  
 \norm{ \tilg(\bx^r; r)   - \tilg(\bfs^r + \alpha_r \bx_0  + \beta_r  \sqrt{n} \bphi_1; r )}^2  \right]  \nonumber \\
 & =  2 \norm{\bphi_1}^2 \left[ \E \{ L f(\alpha_r U + \beta_r L +\tau_r G_0) \} - 
 \frac{\< \sqrt{n} \bphi_1, f(\bx^r; r) \>}{n} \right]^2
  \nonumber \\ 
  & \quad  +  \frac{2}{n}   \norm{ f(\bx^r; r)   - f(\bfs^r + \alpha_r \bx_0  + \beta_r  \sqrt{n} \bphi_1; r )}^2.
  \label{eq:gtild_diff1}
\end{align}
Noting that $\norm{\bphi_1}=1$, the first term on the RHS of \myeqref{eq:gtild_diff1} tends to zero almost surely, as shown in \myeqref{eq:phi1_f_lim}. For the last term in \myeqref{eq:gtild_diff1}, we use the fact that $f(\,\cdot\, ; r)$ is Lipschitz to write
\begin{align}
\frac{1}{n}   \norm{ f(\bx^r; r)   - f(\bfs^r + \alpha_r \bx_0  + \beta_r  \sqrt{n} \bphi_1; r )}^2 \leq  \frac{C}{n}
 \| \bDelta^r \|^2,
\end{align}
where $C>0$ is an absolute constant. By the induction hypothesis $\| \bDelta^r \|^2/n \to 0$ almost surely.  Therefore, using \myeqref{eq:gtild_diff1} and \myeqref{eq:normtbW_lim} in \myeqref{eq:tWg_CS} yields the result in \myeqref{eq:Wg_diff_norm}.

\emph{Proof of \myeqref{eq:btg_diff_norm}}:    Using the inequality $(x+y+z)^2 \leq 3 (x^2  + y^2 + z^2)$ for $x,y,z \in \reals$ and dropping the time index  $(r-1)$ within $\tilg$ for brevity, we write
\begin{align}
 & \frac{1}{n} \norm{ \tilde{\sb}_r \tilg (\bfs^{r-1} +  \alpha_{r-1} \bx_0  + \beta_{r-1})  - \sb_r g(\bx^{r-1})}^2   \nonumber \\
&  \leq \frac{3}{n} \Bigg[ \norm{ \tilde{\sb}_r \tilg (\bfs^{r-1} +  \alpha_{r-1} \bx_0  + \beta_{r-1}) - \sb_r \tilg (\bfs^{r-1} +  \alpha_{r-1} \bx_0  + \beta_{r-1}) }^2  \nonumber \\
& \quad  \qquad +    \sb_r^2 \norm{ \tilg (\bfs^{r-1} +  \alpha_{r-1} \bx_0  + \beta_{r-1})  - \tilg (\bx^{r-1}) }^2 
 +  \sb_r^2 \norm{ \tilg(\bx^{r-1}) - g(\bx^{r-1})}^2 \Bigg] \label{eq:Bg_expand}
\end{align}
First consider the first term in \myeqref{eq:Bg_expand}, which  using the Lipschitz property of $\tilde{g}$  can be bounded as 
\begin{align}
\frac{1}{n}\norm{\tilg (\bfs^{r-1} +  \alpha_{r-1} \bx_0  + \beta_{r-1})}^2 ( \tilde{\sb}_r  -  \sb_r )^2 
\leq \frac{C}{n} \norm{\bfs^{r-1} +  \alpha_{r-1} \bx_0  + \beta_{r-1}}^2 ( \tilde{\sb}_r  -  \sb_r )^2. 
\label{eq:bT1}
\end{align}
From the induction hypothesis in \myeqref{eq:norm_finite} for $t=(r-1)$, we have
\begin{equation}
\limsup_{n \to \infty} \frac{1}{n} \norm{\bfs^{r-1} +  \alpha_{r-1} \bx_0  + \beta_{r-1}}^2 < \infty.
\label{eq:fin_norm}
\end{equation}
Next, we claim that 
\begin{align}
\lim_{n \to \infty} \tilde{\sb}_r    =  \lim_{n \to \infty} \frac{1}{n} \sum_{i =1}^n f'(s^r_i+ \alpha_r x_{0,i} + \beta_t  \sqrt{n} \,  \varphi_{1,i} ; r)  =  \E\{ f'(\alpha_r U + \beta_r L + \tau_r G_0; r) \}  \ \text{a.s.} 
\label{eq:fprs_conv}
\end{align}
Indeed,  the result in \myeqref{eq:st_SE}  implies that the empirical distribution of 
$(\bfs^r + \alpha_r \bx_{0} + \beta_r  \sqrt{n} \,  \bphi_{1})$ converges weakly to the distribution of 
 $\alpha_r U + \beta_r L + \tau_r G_0$. Combining this with the Lipschitz property of $f( \,\cdot\, ; r)$,   \myeqref{eq:fprs_conv} follows from  \cite[Lemma 5]{BM-MPCS-2011}. The limiting value for $\sb_r$ is the same, as shown in \myeqref{eq:fpr_conv}. Therefore, from \myeqref{eq:bT1} we have
 \begin{align}
\lim_{n \to \infty} \frac{1}{n}\norm{\tilg (\bfs^{r-1} +  \alpha_{r-1} \bx_0  + \beta_{r-1})}^2 ( \tilde{\sb}_r  -  \sb_r )^2  =0 \quad \text{ a.s. } \label{eq:gtil_b_btil}
 \end{align}
 
  Next, using the Lipschitz property of  $\tilg$, the second term in \myeqref{eq:Bg_expand} can be bounded as
  \begin{align}
  \frac{\sb_r^2}{n}  \norm{ \tilg (\bfs^{r-1} +  \alpha_{r-1} \bx_0  + \beta_{r-1})  - \tilg (\bx^{r-1}) }^2  \leq
  \frac{ \sb_r^2 C}{n} \norm{\bDelta^{r-1}}^2.  \label{eq:tilgs_tilgx}
  \end{align}
By the induction hypothesis, we have ${\| \bDelta^{r-1} \|}^2/n \to 0$ almost surely. Since $\sb_r$ has already been shown to  approach a finite limit almost surely, we therefore have
 \begin{align}
 \lim_{n \to \infty}   \frac{\sb_r^2}{n}  \norm{ \tilg (\bfs^{r-1} +  \alpha_{r-1} \bx_0  + \beta_{r-1})  - \tilg (\bx^{r-1}) }^2 = 0 \quad \text{a.s.} \label{eq:br_gtil_gtil}
 \end{align}
 
 Finally,  we have 
 \begin{align}
& \lim_{n \to \infty}  \frac{\sb_r^2}{n} \norm{ \tilg(\bx^{r-1}) - g(\bx^{r-1})}^2  \nonumber \\
&  =
\lim_{n \to \infty}  \sb_r^2 \norm{\bphi_1}^2 \left[ \E \{ L f(\alpha_{r-1} U + \beta_{r-1} L +\tau_{r-1} G_0;\, r-1) \} - 
 \frac{\< \sqrt{n} \bphi_1, f(\bx^{r-1}; r-1) \>}{n} \right]^2 \stackrel{(a)}{=} 0  \quad \text{ a.s}, \label{eq:br_gtil_g}
 \end{align}
where $(a)$ follows from  \myeqref{eq:phi1_f_limB}. Using \myeqref{eq:gtil_b_btil}, \myeqref{eq:br_gtil_gtil} and \myeqref{eq:br_gtil_g} in \myeqref{eq:Bg_expand} yields the result in \myeqref{eq:btg_diff_norm}.

\emph{Proof of \myeqref{eq:bdel_norm}}:  Using the definition of $\bdel^t$ in \myeqref{eq:bdeldef}, we write
\begin{align}
& \bdel^r  = \bphi_1 \< \bphi_1, \tbW g(\bx^r; r) \> \nonumber  \\
& = \bphi_1 \left[  \< \bphi_1, \, \tbW \tilg(\bfs^r + \alpha_r \bx_0 + \beta_r \bphi_1; r) \> + \< \bphi_1,  \,\tbW ( g(\bx^r; r) - \tilg(\bfs^r + \alpha_r \bx_0 + \beta_r \bphi_1; r)) \>    \right] \nonumber \\
& = \bphi_1 \left[  \< \bphi_1,  \bfs^{r+1} + \tilde{\sb}_r \tilg(\bfs^{r-1} + \alpha_{r-1} \bx_0 + \beta_{r-1} \bphi_1; r-1)  \> + \< \bphi_1,  \tbW ( g(\bx^r; r) - \tilg(\bfs^r + \alpha_r \bx_0 + \beta_r \bphi_1; r)) \>    \right],
\end{align}
where the last equality follows from \myeqref{eq:st1_def}. Therefore,
\begin{align}
\frac{1}{n} \norm{\bdel^r}^2   \leq  &  \frac{3}{n} \Bigg[  \< \bphi_1,  \bfs^{r+1} \>^2 
+ \tilde{\sb}_r^2 \< \bphi_1, \tilg(\bfs^{r-1} + \alpha_{r-1} \bx_0 + \beta_{r-1} \bphi_1; r-1) \>^2   \nonumber  \\ 
&  \qquad + \< \bphi_1,  \tbW ( g(\bx^r; r) - \tilg(\bfs^r + \alpha_r \bx_0 + \beta_r \bphi_1; r)) \>^2  \Bigg].  \label{eq:delr_expand}
\end{align}
Consider the first term in \myeqref{eq:delr_expand}. We almost surely have,
\begin{align}
\lim_{n \to \infty} \frac{1}{n} \< \bphi_1,  \bfs^{r+1} \>^2 = \lim_{n \to \infty}  \left[\frac{1}{n}   \< \sqrt{n} \bphi_1,  \bfs^{r+1} \> \right]^2  
\stackrel{(a)}{=} \left[ \E \left\{ (\tau_{r+1} G_0 L ) \right\} \right]^2 \stackrel{(b)}{=} 0,\label{eq:bphi_bfs} 
\end{align}
where $(a)$ is obtained by applying the state evolution result \myeqref{eq:st_SE} for $\bfs^{r+1}$ with the pseudo-Lipschitz function
$\psi(s,x,y) = ys$. The equality $(b)$ holds because $L, G_0$ are independent. 

For the second term in \myeqref{eq:delr_expand}, using the definition of $\tilg$ in \myeqref{eq:tilg_def} we write
\begin{align}
& \frac{1}{n} \< \bphi_1, \tilg(\bfs^{r-1} + \alpha_{r-1} \bx_0 + \beta_{r-1} \bphi_1; r-1) \>^2  \nonumber \\
&  =  \left[ \frac{1}{n} \< \sqrt{n} \bphi_1,  f(\bfs^{r-1} + \alpha_{r-1} \bx_0 + \beta_{r-1} \bphi_1; r-1) \> - \E \{ Lf(\tau_{r-1}G_0 + \alpha_{r-1} U + \beta_{r-1} L; r-1) \} \right]^2. \label{eq:phi1_gtil}
\end{align}
Now, applying the state evolution result  \myeqref{eq:st_SE0} to the pseudo-Lipschitz function $\psi(u,v,z) = z f(u;r-1)$, we obtain
\begin{align}
\lim_{n \to \infty} \frac{\< \sqrt{n}\bphi_1, f(\bfs^{r-1} + \alpha_{r-1} \bx_0 + \beta_{r-1} \bphi_1; r-1) \> }{n}  & = \E \left\{ L f( \tau_{r-1} G_0 + \alpha_{r-1} U + \beta_{r-1} L; \, r-1) \right\} \ \text{a.s.}  \label{eq:phi1_fs_limB}
 \end{align}
Using this in \myeqref{eq:phi1_gtil}, and recalling from \myeqref{eq:fprs_conv} that $\sb_r$ converges to a finite value,  we get
\begin{align}
\lim_{n \to \infty} \frac{\sb_r^2}{n} \< \bphi_1, \tilg(\bfs^{r-1} + \alpha_{r-1} \bx_0 + \beta_{r-1} \bphi_1; r-1) \>^2 = 0 \ \text{a.s.}
\end{align}

Finally, for the third term in \myeqref{eq:delr_expand}, using Cauchy-Schwarz we have
\begin{align}
\frac{1}{n} \< \bphi_1,  \tbW ( g(\bx^r; r) - \tilg(\bfs^r + \alpha_r \bx_0 + \beta_r \bphi_1; r)) \>^2 \leq \| \tbW \|_{\op}^2 \cdot 
\frac{1}{n} \| g(\bx^r; r) - \tilg(\bfs^r + \alpha_r \bx_0 + \beta_r \bphi_1; r) \|^2. \label{eq:phi1_tw_tg}
\end{align}
Recall that $\| \tbW \|_{\op}^2 \to 4$ almost surely. The last term in \myeqref{eq:phi1_tw_tg} can be bounded as 
\begin{align}
& \lim_{n \to \infty} \frac{1}{n} \| g(\bx^r; r) - \tilg(\bfs^r + \alpha_r \bx_0 + \beta_r \bphi_1; r) \|^2  \nonumber  \\
&\leq  \lim_{n \to \infty} \left[ \frac{2}{n} \| g(\bx^r; r) - \tilg(\bx^r; r) \|^2 + \frac{2}{n} \| \tilg(\bx^r; r)-  \tilg(\bfs^r + \alpha_r \bx_0 + \beta_r \bphi_1; r) \|^2 \right] =0 \text{ a.s}\label{eq:gx_gtils}
\end{align}
from the arguments in \myeqref{eq:tilgs_tilgx} -- \myeqref{eq:br_gtil_g}.

To summarize, we have proven that \myeqref{eq:lim_betat} -- \myeqref{eq:bdel_norm} hold, and  consequently \myeqref{eq:psi_xs} and \myeqref{eq:norm_Del} hold for $t=(r+1)$.  Finally, we need to verify that the conditions in \myeqref{eq:norm_finite} also hold for $t=(r+1)$. But these immediately follow from \myeqref{eq:st_SE0} and \myeqref{eq:psi_xs} with  $t=r$  by considering  the pseudo-Lipschitz function $\psi(u,v,w) = u^2$.


\section{Proof of Theorem \ref{thm:Main}: General case}
\label{sec:ProofMainGeneral}

Throughout this appendix, we use the notation $f(\bx,y;t) = f_t(\bx,y;t)$.

The last statement of the theorem, that $\boldsymbol{\Omega} \in \cG_n(\bLambda)$ with the claimed probability and the weak convergence of $\boldsymbol{\Omega}$,  follows from Lemma \ref{lem:ConvergenceEigenvectors}.

It remains to prove the state evolution result \myeqref{eq:main_SE_result}. To reduce book-keeping, we will assume $k_-=0$ so that $k_*=k_+$, i.e.,  all the large rank-one perturbations are positive-definite. The general case is completely analogous. 

We will use Lemma \ref{lemma:Distr}, which states that the law of $\bA$ in \myeqref{eq:SpikedDef} is close in total variation to the law of 
\begin{align}
\tbA  &= \sum_{i=1}^{k_*}z_i\bphi_i\bphi_i^{\sT}  +  \bPp \left (  \sum_{i=1}^{k_*} \lambda_i \bv_i\bv_i^{\sT} +\tbW \right)  \bPp\, , \label{eq:tbadef_gen}
\end{align}
where $(z_1, \ldots, z_{k_*})$ are the first $k_*$ ordered eigenvalues of $\bA$ in \myeqref{eq:SpikedDef}, and  $(\bphi_1, \ldots, \bphi_{k_*})$ are the corresponding eigenvectors. The matrix $\bPp$ is the projector onto the space orthogonal to the column space of $\bPhi_{\hat{S}}$, where 
\beq
\bPhi_{\hat{S}} \equiv [ \bphi_1 \mid \bphi_2 \ldots \mid \bphi_{k^*} ].  \label{eq:bPhS_def}
\eeq
We also define
\beq
\bZ_{\hS}  \equiv \diag(z_1, \ldots, z_{k_*}).  \label{eq:ZhS}
\eeq

We will first prove the convergence result \myeqref{eq:main_SE_result} assuming $(\bx^t)_{t\ge 0}$ were generated using the AMP iteration with $\tbA$, i.e.: 
\begin{align}
 \bx^0 & =  \sqrt{n}  \, [\bphi_1|\cdots|\bphi_{k_*}|\bzero|\cdots|\bzero] \,  ,\\
\bx^{t+1} &= \tbA\,  f(\bx^t,\by;t) - f(\bx^{t-1},\by;t-1)\, \sB_t^{\sT} \, .\label{eq:AMPspecial_tilde_gen}
\end{align} 
For any $\eps \in (0, \eps_0)$, Lemma \ref{lemma:Distr} bounds the total variation distance between the  conditional joint distributions of $(\tbA,\bphi_1)$ and 
$(\bA,\bphi_1)$ given $(z_1, \bphi_1) \in \cE_{\eps}$, where $\cE_{\eps}$ is defined in \myeqref{eq:cE_eps}. Then, using steps similar to \myeqref{eq:rank1_coup1} -- \myeqref{eq:rank1_coup2},  it follows that there exists a coupling of the laws of $\bA$ and $\tbA$ such that
\begin{align}
\prob\left\{ \frac{1}{n}\sum_{i=1}^n\psi(\bx_i^t(\bA),\tbv_i,y_i) \neq \frac{1}{n}\sum_{i=1}^n\psi(\bx_i^t(\tbA),\tbv_i,y_i) \right\} \le \frac{1}{c(\eps)}\, e^{-nc(\eps)}\, ,
\end{align}
for some constant $c(\eps) >0$. Here we have emphasized the dependence on the matrix $\bA$. By Borel-Cantelli, the two averages coincide eventually almost surely.
Therefore, once   \myeqref{eq:main_SE_result}  of Theorem \ref{thm:Main} holds for  $\tbA$, it also holds for  $\bA$.

Let us now turn to the analysis of the iteration  (\ref{eq:AMPspecial_tilde_gen}). Define
\begin{align}
\bLambda_{S} & \equiv \diag(\lambda_1, \ldots, \lambda_{k_*}),  \\
\tbV  & \equiv  \sqrt{n} \, [\bv_1 \mid  \bv_2 \ldots \mid  \bv_{k^*} ]
= \sqrt{n}  \bV_S \,, \label{eq:Un_def} \\ 
g(\bx^t,\by;t)  &\equiv f(\bx^t,\by;t)  - \bPhi_{\hS}\, [\bPhi_{\hS}^{\sT} f(\bx^t,\by;t)],  \label{eq:gdef_gen} \\
\bdel^t & \equiv \bPhi_{\hS} \left[  \bPhi_{\hS}^{\sT} \tbW g(\bx^t,\by;t) \right]. \label{eq:del_def_gen} 
\end{align}
With these definitions, using \myeqref{eq:tbadef_gen} in \myeqref{eq:AMPspecial_tilde_gen}  and noting that $\bPp = \bI - \bPhi_{\hS} \bPhi_{\hS}^{\sT}$, we can write
\begin{align} 
 \bx^{t+1}  & = \bPhi_{\hS} \left[ \bZ_{\hS} \, \bPhi_{\hS}^{\sT} f(\bx^t, \by; t) - \bPhi_{\hS}^{\sT} f(\bx^{t-1}, \by; t-1) \sB_t^{\sT} 
 -  \frac{1}{n} \bPhi_{\hS}^{\sT} \tbV\bLambda_{S} \tbV^{\sT} \bPp f(\bx^t, \by; t) \right]  \nonumber \\
& + \frac{1}{n} \tbV \bLambda_{S}\tbV^\sT \bPp f(\bx^t, \by; t)   \, + \, \tbW\,  g(\bx^t,\by;t) - g(\bx^{t-1},\by;t-1)\, \sB_t^{\sT} \, - \bdel^t.  \label{eq:xt1_exp_gen}
 \end{align}
 
 Let $(\bU, \bL)$ be random pair of vectors $\bU, \bL \in \reals^{k_*}$, where $\bU \sim \mu_{\bU}$ and,
for  $\bG_1\sim \normal(\mathbf{0},\id_q)$ independent of $\bU$, we let
\begin{align}
 \bL =  \bProd_0\bU \, +  \tbProd_0\bG_1  \, .\label{eq:Ldef_gen} 
\end{align}
Here $\bProd_0, \tbProd_0\in\reals^{k_*\times k_*}$ are defined as in the statement of the theorem.
By Lemma \ref{lem:ConvergenceEigenvectors} we have, almost surely, for any $\psi\in\PL(2)$, 
\begin{align}
\lim_{n\to\infty}\left|\frac{1}{n}\sum_{i=1}^n\psi(\sqrt{n}\bPhi_{\hS,i},\tbv_i)
      - \E\big\{\psi(\bL,\bU)\big\}\right|= 0\, ,
\label{eq:ConvergenceEvectors_Specialized}
\end{align}
where $\bPhi_{\hS,i} =\bPhi_{\hS}^{\sT}\be_i$ is the rescaled $i$-th row of $\bPhi_{\hS}$, and $\tbv_i = \tbV^{\sT}\be_i$ 
is the $i$-th row of $\tbV$.

Next define
\begin{align} 
\tbX_t  \equiv \balpha_t \bU \, + \, \bbeta_t \bL \, + \, \btau_t \bG_0, \label{eq:tilXt_def_gen} 
\end{align}  
where $\bG_0\sim \normal( \mathbf{0}, \bI_q)$ is independent of $(\bU, \bL)$. The matrices $\balpha_t \in \reals^{q \times k_*}$, 
$\bbeta_t \in \reals^{q \times k_*}$, and $\btau_t \in \reals^{q \times q}$  are measurable on the sigma-algebra $\sigma(\bProd,\bZ_{\hS})$ and defined via the following recursion. Starting with 
\beq 
\balpha_0 = \mathbf{0}_{q \times k_*}, \quad   \bbeta_0 = \begin{bmatrix} \id_{k_*} \\ \mathbf{0}_{(q- k_*) \times k_*}  \end{bmatrix}, \quad     \btau_0 = \mathbf{0}_{q \times q},
\label{eq:alph_beta_ini}
\eeq  
we compute (here expectations are with respect to $\tbX_{\cdot}$, $\bU$, $\bL$, $Y$, at $\bProd_0$, $\tbProd_0$, $\bZ_{\hS}$ fixed), for $t\ge 0$,
\begin{align}
\balpha_{t+1} & = \E \{ f(\tilde{\bX}_t, Y;t) \bU^\sT \}\bLambda_S - \E \{ f(\tilde{\bX}_t, Y;t)\bL^{\sT} \} \bProd_0\bLambda_{S}\, ,   \label{eq:AlphaRec_gen} \\
\bbeta_{t+1} & =    \E \{ f(\tilde{\bX}_t, Y;t)\bL^{\sT} \} (\bZ_{\hS}+  \bProd_0\bLambda_{S}\bProd_0^{\sT}) - \E \{ f'(\tilde{\bX}_t, Y;t) \} \E \{ f( \tilde{\bX}_{t-1}, Y ;  t-1)\bL^{\sT} \}  \nonumber \\ 
& \quad - \E \{ f(\tilde{\bX}_t, Y;t) \bU^{\sT} \}  \bLambda_{S}\bProd_0^{\sT},\label{eq:BetaRec_gen} 
\\
\btau_{t+1}^2 & = \E\{ f(\tilde{\bX}_t, Y;t) f(\tilde{\bX}_t, Y;t)^{\sT}  \} -  \E \{ f(\tilde{\bX}_t, Y;t)\bL^{\sT} \} \E \{\bL  f(\tilde{\bX}_t, Y;t)^{\sT}  \}\, . \label{eq:TauRec_gen}
\end{align}
In \myeqref{eq:BetaRec_gen}, we have used $f'$ as shorthand to denote the Jacobian matrix $ \frac{\partial f}{\partial \bx} \in \reals^{q\times q}$.
Further, we set by convention $\E \{ f(\tilde{\bX}_{-1}, Y;t=-1)\bL^{\sT} \}= [ \bLambda_S^{-1} \mid \mathbf{0}_{k_* \times (q- k_*) }]^{\sT}$.

We will prove \myeqref{eq:main_SE_result}  by establishing the two lemmas below.

\begin{lemma}\label{lemma:ConvergenceToMean}
For any pseudo-Lipschitz function  $\psi:\reals^{q+2k_*+1}\to \reals$, 
the following limit holds almost surely:
\begin{align}
\lim_{n \to \infty}\left|\frac{1}{n}\sum_{i=1}^n\psi(\bx_i^t,\tbu_i,  \sqrt{n}\bPhi_{\hS,i},y_i) -
\E\big\{\psi(\balpha_t \bU  +  \bbeta_t \bL  + \btau_t \bG_0, \, \bU, \,  \bL, \,Y) \big\}\right| = 0\, .\label{eq:SE_conv1_gen}
\end{align}
\end{lemma}

\begin{lemma}\label{lemma:ExpConvergence}
For any pseudo-Lipschitz function  $\psi:\reals^{q+k_*+1}\to \reals$, 
the following limit holds almost surely:
\begin{align}
\lim_{n\to\infty}\left|\E\big\{\psi(\balpha_t \bU  +  \bbeta_t \bL  + \btau_t \bG_0, \, \bU,  \,Y) \big\} -
\E \left\{ \psi(\bM_t\bU+\bQ_t^{1/2}\bG,  \bU,  \,Y) \right\}\right| = 0\, .  \label{eq:SE_conv2_gen}
\end{align}
\end{lemma}

\subsection{Proof of Lemma \ref{lemma:ExpConvergence}}

For $(\balpha_t, \bbeta_t, \btau_t^2)$ defined via the recursion in  Eqs.~\eqref{eq:AlphaRec_gen} -- \eqref{eq:TauRec_gen}.
We show below that for $t \geq -1$, almost surely,
\begin{align}
\lim_{n\to\infty}\big\|\bbeta_{t+1} -  \E \{ f(\tilde{\bX}_t; t)  \bL^{\sT} \}\, \bLambda_{S}\big\|_F = 0\, .  \label{eq:betat1_prop_gen}
\end{align}
In order to see how this implies the lemma, denote the functions that enter the state evolution recursion (\ref{eq:Mt_def}), (\ref{eq:Qt_def}) by
\begin{align}
\cuF_t(\bM,\bQ)& \equiv \E\Big\{f(\bM\bU+\bQ^{1/2}\bG,Y;t)\bU^{\sT}\Big\} \bLambda_S\, , \label{eq:cuF_t} \\
\cuG_t(\bM,\bQ)& \equiv \E\Big\{f(\bM\bU+\bQ^{1/2}\bG,Y;t) f(\bM\bU+\bQ^{1/2}\bG,Y;t)^{\sT}\big\} \, . \label{eq:cuG_t}
\end{align}
Note that these are continuous functions by the Lipschitz continuity of $f(\cdots)$. Further let 
\begin{align}
\tbM_t & \equiv \balpha_t+\bbeta_t\bProd_0\, ,\\
\tbQ_t & \equiv \btau_t^2+\bbeta_t\tbProd_0^2\bbeta_t\, ,
\end{align}
and notice that, by construction,
\begin{align}
\E\big\{\psi(\balpha_t \bU  +  \bbeta_t \bL  + \btau_t \bG_0, \, \bU,  \,Y) \big\}=
\E \left\{ \psi(\tbM_t\bU+\tbQ_t^{1/2}\bG,  \bU,  \,Y) \right\}\, .
\end{align}
Using Eq.~(\ref{eq:betat1_prop_gen}) together with $\|\bProd_0\bLambda_S-\bLambda_S\bProd_0\|_F\le C\eta_n\to 0$ (which holds eventually almost surely 
since $\bProd\in\cG_n(\bLambda)$) and $\|\bProd_0\bProd_0^{\sT}-(\id-\bLambda_S^{-2})\|_F\le C\eta_n \to 0$  (which also holds because $\bProd\in\cG_n(\bLambda)$) in
Eqs.~(\ref{eq:AlphaRec_gen}) to  (\ref{eq:TauRec_gen}), we get
\begin{align}
\lim_{n\to \infty}\big\|\tbM_{t+1}-\cuF_t(\tbM_t,\tbQ_t)\big\|_F & = 0\, ,  \label{eq:lim_tMt1} \\
\lim_{n\to \infty}\big\|\tbQ_{t+1}-\cuG_t(\tbM_t,\tbQ_t)\big\|_F & = 0\, .     \label{eq:lim_tQt1}
\end{align}
The functions $\cuF_t(\cdot), \cuG_t(\cdot)$ in Eqs. \eqref{eq:cuF_t} -- \eqref{eq:cuG_t} are exactly the ones defining the recursion for $\bM_{t+1}, \bQ_{t+1}$ in Eqs. \eqref{eq:Mt_def} -- \eqref{eq:Qt_def}.
From the initialization in \myeqref{eq:alph_beta_ini},  we also have $\tbM_0= \bM_0$. Hence, by induction, and using the continuity of $\cuF_t(\,\cdot\, )$, $\cuG_t(\,\cdot\, )$,  Eqs. \eqref{eq:lim_tMt1} -- \eqref{eq:lim_tQt1} imply that $\lim_{n\to\infty}\|\tbM_t-\bM_t\|_F=0$, $\lim_{n\to\infty}\|\tbQ_t-\bQ_t\|_F=0$ 
(almost surely), therefore implying Eq.~(\ref{eq:SE_conv2_gen}).

 We are now left with the task of proving Eq. ~(\ref{eq:betat1_prop_gen}), which we do by induction. The base case follows immediately from the initialization 
\myeqref{eq:alph_beta_ini} and our convention on $\E \{ f(\tilde{\bX}_t; t)  \bL^{\sT} \}$. Assuming towards induction that (\myeqref{eq:betat1_prop_gen}) is true for $t=0, \ldots, r$, and
using the definition of  $\bL$ in \eqref{eq:Ldef_gen} we have
\begin{align}
 \E \{ f(\tilde{\bX}_r, Y;r)\bL^{\sT} \}   & = \E \{ f(\tilde{\bX}_r, Y;r)\bU^{\sT} \}\bProd_0^{\sT} +   \E \{ f(\tilde{\bX}_r, Y;r) \bG_1^{\sT} \}  \tbProd_0^{\sT}  \nonumber \\
& =\E \{ f(\tilde{\bX}_r, Y;r) \bU^{\sT}  \} \bProd_0^{\sT} +  \E \{ f'(\tilde{\bX}_r, Y;r)  \} \bbeta_r  \tbProd_0\tbProd_0^{\sT}\, ,
\end{align}
where the last identity follows from Stein's lemma. Substituting In Eq.~(\ref{eq:BetaRec_gen}), we get
\begin{align}
\bbeta_{r+1}-\E \{ f(\tilde{\bX}_t; t)  \bL^{\sT} \}\, \bLambda_{S} & =    \E \{ f(\tilde{\bX}_r, Y;r)\bU^{\sT} \} \big[\bProd_0^{\sT}(\bZ_{\hS}+  \bProd_0\bLambda_{S}\bProd_0^{\sT}-\bLambda_S)-\bLambda_S\bProd_0^{\sT}\big]\nonumber \\
& \  +  \E \{ f'(\tilde{\bX}_t, Y;t) \} \, \bbeta_r\tbProd_0\tbProd_0^{\sT}\big(\bZ_{\hS}+  \bProd_0\bLambda_{S}\bProd_0^{\sT}-\bLambda_S\big)\\
 & \ - \E \{ f'(\tilde{\bX}_t, Y;t) \} \E \{ f( \tilde{\bX}_{t-1}, Y ;  t-1)\bL^{\sT} \}  \nonumber 
\end{align}
The claim then follows by using the induction hypothesis, together with the fact that, almost surely: $\|\bProd_0\bLambda_S-\bLambda_S\bProd_0\|_F\to 0$; 
$\|\bProd_0\bProd_0^{\sT}-(\id-\bLambda_S^{-2})\|_F\to 0$; $\|\bZ_{\hS}-(\bLambda_S-\bLambda_S^{-1})\|_F\to 0$.

\subsection{Proof of Lemma \ref{lemma:ConvergenceToMean}}

Let 
\begin{align}
\tilg(\bx^t, \by ; t) & = f(\bx^t,\by;t)  - \sqrt{n} \bPhi_{\hS}\, \E \{ \bL f(\balpha_t \bU  +  \bbeta_t \bL  + \btau_t \bG_0, \, Y \, ;t)^\sT \},\label{eq:tilg_def_gen}
\end{align}
and define the iteration $(\bfs^t)_{t \geq 0}$ as follows.
\begin{align}
\bfs^{t+1}&  =  \tbW \tilg(\bfs^t +  \tbV \balpha_t^{\sT}  +   \sqrt{n} \bPhi_{\hS} \, \bbeta_t^{\sT}, \, \by; t ) -  
\tilg (\bfs^{t-1} +  \tbV \balpha_{t-1}^{\sT} +  \sqrt{n} \bPhi_{\hS} \, \bbeta_{t-1}^{\sT}, \, \by ; t-1) \tilde{\sB}_t^{\sT}, \label{eq:st1_def_gen} \\ 
\tilde{\sB}_t &= \frac{1}{n}\sum_{i=1}^n\frac{\partial f}{\partial \bx}( (\bfs^t +  \tbV \balpha_t^{\sT}  +   \sqrt{n} \bPhi_{\hS} \, \bbeta_t^{\sT})_i , y_i ; t)   
%
\end{align}
The iteration is initialized with 
\beq
\bfs^0   =  \bx^0 - \tbV \balpha_0^{\sT}  -   \sqrt{n} \bPhi_{\hS} \, \bbeta_0^{\sT} = \mathbf{0},   
\label{eq:s0_def_gen}
\eeq
where the last equality follows from assumption \ref{ass:SpikesInit} which sets $\bx^0 = \sqrt{n} \bPhi_{\hS} \, \bbeta_0^{\sT}$, and from the definition of  $\balpha_0, \bbeta_0$ in \eqref{eq:alph_beta_ini}.

Since the empirical distribution of $(\tbV, \sqrt{n} \bPhi_{\hS})$
converges in $W_2$ to the distribution of $(\bU, \bL)$, and the iteration for $\bfs^t$ is of the standard AMP form\footnote{The convergence statement in  \cite[Theorem 1]{berthier2017state}, is in probability.
However exploiting the additional separability structure as in \cite[Theorem 1]{javanmard2013state} yields almost sure convergence.} in \cite[Theorem 1]{berthier2017state}, for any pseudo-Lipschitz function $\tilde{\psi}$ we have:
\begin{align}
\lim_{n \to \infty}  \left|\frac{1}{n} \sum_{i=1}^n \tilde{\psi} (\bfs^t_i, \tbv_{i}, \sqrt{n} \bPhi_{\hS,i}, y_i) -
\E \left\{ \tilde{\psi} ( \btau_t  \bG_0, \bU, \bL, Y ) \right\}\right|=0\, . \label{eq:st_SE_gen}
\end{align}
where  $\btau_t$ is determined by the  recursion \myeqref{eq:TauRec_gen}. Therefore,  choosing 
\begin{align}
\tilde{\psi} (\bfs^t_i, \tbv_{i}, \sqrt{n} \bPhi_{\hS,i}, y_i) = \psi(\bfs^t_i +   \balpha_t\tbv_{i} + \sqrt{n} \bbeta_t\bPhi_{\hS,i}, \tbv_{i}, \sqrt{n} \bPhi_{\hS,i}, y_i)
\end{align}
 for a pseudo-Lipschitz function $\psi$, \myeqref{eq:st_SE_gen} implies that almost surely
\begin{align}
\lim_{n \to \infty}  \left|\frac{1}{n} \sum_{i=1}^n \psi(\bfs^t_i +   \balpha_t\tbv_{i} + \sqrt{n} \bbeta_t\bPhi_{\hS,i}, \tbv_{i}, \sqrt{n} \bPhi_{\hS,i}, y_i) -
\E \left\{ \psi ( \btau_t \bG_0 + \balpha_t \bU + \bbeta_t \bL, \bU, \bL, Y ) \right\}\right| = 0.  \label{eq:st_SE0_gen}
\end{align}
Therefore to prove \myeqref{eq:SE_conv1_gen} it suffices to show that almost surely
\begin{align}
\lim_{n \to \infty}  \left|  \frac{1}{n}\sum_{i=1}^n\psi(\bx_i^t,\tbv_i, \sqrt{n}\bPhi_{\hS,i} ,y_i)   -  
\frac{1}{n} \sum_{i=1}^n \psi(\bfs^t_i +   \balpha_t\tbv_{i} + \sqrt{n} \bbeta_t\bPhi_{\hS,i}, \tbv_{i}, \sqrt{n} \bPhi_{\hS,i}, y_i) \right| =0. 
\label{eq:psi_xs_gen}
\end{align}
We define the discrepancy  $\bDelta^t\in\reals^{n\times q}$ by
\begin{align}
\bDelta^t= \bx^t - \left( \bfs^t +  \tbV \balpha_t +  \sqrt{n} \bPhi_{\hS} \bbeta_t^{\sT} \right),  \label{eq:bDel_def_gen}
\end{align}
and inductively prove \myeqref{eq:psi_xs_gen} together with the following claims:
\begin{align}
 & \lim_{n \to \infty} \frac{1}{n} \norm{\bDelta^t}_{F}^2  = 0, \label{eq:norm_Del_gen} \\
&  \limsup_{n \to \infty} \frac{1}{n} \norm{\bx^t}_{F}^2 < \infty, \label{eq:norm_finite_gen_1}\\
& \limsup_{n \to \infty} \frac{1}{n} \norm{ \bfs^t +  \tbV \balpha_t^{\sT}  +  \sqrt{n} \bPhi_{\hS} \bbeta_t^{\sT} }_{F}^2 < \infty. \label{eq:norm_finite_gen_2}
\end{align}
The base case of $t=0$ is easy to verify. Indeed, from the definition of $\bfs^0$ in \myeqref{eq:s0_def_gen},  we have $\bDelta^0 = \mbf{0}$
and the equality \myeqref{eq:psi_xs_gen} holds. Furthermore,  Eqs.~\eqref{eq:norm_finite_gen_1} and \eqref{eq:norm_finite_gen_2} also hold for $t=0$ since the initial condition  
$\bx^0=  \sqrt{n} \bPhi_{\hS} \, \bbeta_0^{\sT}$ and the definitions of $\balpha_0, \bbeta_0$ in \eqref{eq:alph_beta_ini} imply
\begin{align}
 \frac{1}{n} \norm{\bx^0}_{F}^2 =  \frac{1}{n} \norm{\bfs^0 +  \tbV \balpha_0^{\sT}  +  \sqrt{n} \bPhi_{\hS} \bbeta_0^{\sT} }_F^2  = \Tr(\bbeta_0 \bbeta_0^{\sT}) = k_*\, .
\end{align}

With the induction hypothesis that Eqs.~\eqref{eq:psi_xs_gen} to \eqref{eq:norm_finite_gen_2} hold for $t=0, 1, \ldots, r$, we now prove the claim for $t=r+1$.   By the pseudo-Lipschitz property  of $\psi$, for some constant $C>0$  we have:
\begin{align}
& \frac{1}{n} \abs{  \sum_{i=1}^n \left[ \psi(\bx_i^{r+1},\tbv_i, \sqrt{n}\bPhi_{\hS,i} ,y_i)   -  
 \psi(\bfs^{r+1}_i +   \balpha_{r+1}\tbv_i + \sqrt{n} \bbeta_r\bPhi_{\hS,i}, \tbv_{i}, \sqrt{n} \bPhi_{\hS,i}, y_i)  \right]}  \nonumber \\
&   \leq C \left[  \frac{\norm{\bDelta^{r+1}}_F}{\sqrt{n}} + \frac{\norm{\bDelta^{r+1}}_F}{\sqrt{n}} \frac{\norm{\bx^{r+1}}_F}{\sqrt{n}} 
+ \frac{\norm{\bDelta^{r+1}}_F}{\sqrt{n}} \frac{\norm{ \bfs^{r+1} +  \tbV \balpha_{r+1}^{\sT}  +  \sqrt{n} \bPhi_{\hS} \bbeta_{r+1}^{\sT} }_F}{\sqrt{n}} \right]. \label{eq:psiPL00_gen}
\end{align}
Substituting the expressions for $\bx^{r+1}$ and $\bfs^{r+1}$ from \myeqref{eq:xt1_exp_gen} and  \myeqref{eq:st1_def_gen} into definition of $\bDelta^{r+1}$ from \myeqref{eq:bDel_def_gen},  we get 
\begin{align}
& \bDelta^{r+1}  \nonumber \\
&=   \sqrt{n} \bPhi_{\hS} \left[ \bZ_{\hS}   \frac{\sqrt{n} \bPhi_{\hS}^{\sT} f(\bx^r, \by; r)}{n} -  
\frac{\sqrt{n} \bPhi_{\hS}^{\sT} f(\bx^{r-1}, \by; r-1)}{n} \sB_r^{\sT} 
 -   \frac{\bPhi_{\hS}^{\sT} \tbV}{\sqrt{n}} \left(\frac{\bLambda_S\tbV^{\sT} \bPp f(\bx^r, \by; r)}{n} \right) - \bbeta_{r+1}^{\sT} \right]  \nonumber \\
& \  + \tbV \left[   \frac{\bLambda_S\tbV^\sT \bPp f(\bx^r, \by; r)}{n}  - \balpha_{r+1}^{\sT} \right]   \, 
+ \, \tbW \left[   g(\bx^r,\by;r) -  \tilg(\bfs^r +  \tbV \balpha_r^{\sT}  +   \sqrt{n} \bPhi_{\hS} \, \bbeta_r^{\sT}, \, \by; r )  \right] \nonumber \\ 
& \  + \left[  \tilg (\bfs^{r-1} +  \tbV \balpha_{r-1}^{\sT} +  \sqrt{n} \bPhi_{\hS} \, \bbeta_{r-1}^{\sT}, \, \by ; r-1) \tilde{\sB}_r^{\sT} - g(\bx^{r-1},\by;r-1)\, \sB_r^{\sT} \right] - \bdel^r.
\end{align}

We now show that $\norm{\bDelta^{r+1}}^2/n  \to 0$ almost surely by proving that the following limits hold almost surely.
\begin{align}
& \lim_{n \to \infty}  \left[ \bZ_{\hS} \,  \frac{\sqrt{n} \bPhi_{\hS}^{\sT} f(\bx^r, \by; r)}{n} -  
\frac{\sqrt{n} \bPhi_{\hS}^{\sT} f(\bx^{r-1}, \by; r-1)}{n} \sB_r^{\sT} 
  \right. \nonumber \\
& \phantom{AAAAAAA} \left. -  \frac{\bPhi_{\hS}^{\sT} \tbV}{\sqrt{n}} \left(\frac{\bLambda_S\tbV^{\sT} \bPp f(\bx^r, \by; r)}{n} \right) -\bbeta_{r+1}^{\sT} \right] =0,  \label{eq:lim_beta_r1_gen}  
\end{align}
\begin{align}
& \lim_{n \to \infty}  \left[\frac{\bLambda_S\tbV^\sT \bPp f(\bx^r, \by; r)}{n}  - \balpha_{r+1}^{\sT}\right] = 0, \label{eq:lim_alpht_gen} \\
& \lim_{n \to \infty}  \frac{1}{n} \norm{\tbW \left[   g(\bx^r,\by;r) -  \tilg(\bfs^r +  \tbV \balpha_r^{\sT}  +   \sqrt{n} \bPhi_{\hS} \, \bbeta_r^{\sT}, \, \by; r )  \right] }_{F}^2 =0, \label{eq:Wg_diff_norm_gen} \\
& \lim_{n \to \infty}  \frac{1}{n} \norm{ \tilg (\bfs^{r-1} +  \tbV \balpha_{r-1}^{\sT} +  \sqrt{n} \bPhi_{\hS} \, \bbeta_{r-1}^{\sT}, \, \by ; r-1) \tilde{\sB}_t^{\sT} - g(\bx^{r-1},\by;r-1)\, \sB_r^{\sT} }_F^2=0,  \label{eq:btg_diff_norm_gen} \\
& \lim_{n \to \infty}  \frac{1}{n} \norm{\bdel^r}_F^2 =0. \label{eq:bdel_norm_gen}
\end{align}

We now proceed to prove  Eqs.~(\ref{eq:lim_beta_r1_gen}) to (\ref{eq:bdel_norm_gen}). In the following, expectations are understood to be taken with respect to
the random variables $\bU, \bL, \bG_0$. To lighten notation, given two sequences $A_n$, $B_n$, we write $A_n=B_n+o_n(1)$ if $\lim_{n\to\infty}|A_n-B_n|=0$ almost surely (and we will
not mention `almost surely' explicitly).

\vspace{0.1cm}

\noindent\emph{Proof of \myeqref{eq:lim_beta_r1_gen}}.  From standard results on spiked random matrices, we have $ \bZ_{\hS} = \bLambda_{S} + \bLambda_{S}^{-1}+o_n(1)$, see e.g. 
\cite{benaych2011eigenvalues}.  Further, by definition, we have that 
\begin{align}
\frac{\bPhi_{\hS}^{\sT} \tbV}{\sqrt{n}} =\bProd_0\label{eq:bPhitbU}
\end{align}
Using the induction hypothesis for the pseudo-Lipschitz function 
$\psi(\bx_i^t,\tbv_i, \bPhi_{\hS,i} ,y_i) = $ $\sqrt{n}\bPhi_{\hS,i} f(\bx^t_i, y_i; t)^\sT$ (where $\bPhi_{\hS,i} \in \reals^{k_*}$ and $f(\bx^t_i, y_i; t) \in \reals^{q}$ are column vectors), we obtain,
for all $t\in\{0,\dots,r\}$,
\begin{align}
\frac{\sqrt{n} \bPhi_{\hS}^{\sT} f(\bx^t, \by; t)}{n} = 
\E[\bL f( \balpha_t \bU + \bbeta_t \bL + \btau_t \bG_0 ,Y; t)^{\sT}]+o_n(1)\, ,  \label{eq:bPhifT}
\end{align}
By a similar application of the induction hypothesis for the pseudo-Lipschitz functions $\tbv_i f(\bx^t_i, y_i; t)^\sT$ and $\sqrt{n}\bPhi_{\hS,i} f(\bx^t_i, y_i; t)^\sT$ we get
\begin{align}
& \frac{\tbV^{\sT} \bPp f(\bx^r, \by; r)}{n}   = \frac{\tbV^{\sT} f(\bx^r, \by; r) }{n} - \frac{\tbV^{\sT} \bPhi_{\hS}}{\sqrt{n}} 
\frac{\sqrt{n}\bPhi_{\hS}^{\sT} f(\bx^r, \by; r) }{n} \nonumber  \\
&   =  \E \{  \bU f(\balpha_r \bU + \bbeta_r \bL + \btau \bG_0, Y ; r )^\sT \}  -  \bProd_0^{\sT} \, \E \{  \bL  f(\balpha_r \bU + \bbeta_r \bL + \btau \bG_0, Y ; r )^\sT \}  +o_n(1)\, . \label{eq:tbUsTbPp}
\end{align}
The induction hypothesis implies that the empirical distribution of $\bx^r$ converges in $W_2$ to the distribution of 
 $\balpha_r \bU + \bbeta_r \bL + \btau_r \bG_0$. Combining this with the Lipschitz property of $f( \cdot, \cdot; r)$,  as in \cite[Lemma 5]{BM-MPCS-2011} we obtain
 \begin{align}
 \sB_r =  \E \{ f'(\balpha_r \bU + \bbeta_r \bL + \btau_r \bG_0, Y; r) \}  +o_n(1)\, ,\label{eq:fpr_conv_gen}
 \end{align}
 where $f'$ denotes the Jacobian $\frac{\partial f}{\partial \bx} \in \reals^{q \times q}$.
 
 Finally, combining the results in \myeqref{eq:bPhitbU} -- \myeqref{eq:fpr_conv_gen}, we obtain
 \begin{align}
 & \left[ \bZ_{\hS} \,  \frac{\sqrt{n} \bPhi_{\hS}^{\sT} f(\bx^r, \by; r)}{n} -  
\frac{\sqrt{n} \bPhi_{\hS}^{\sT} f(\bx^{r-1}, \by; r-1)}{n} \sB_r^{\sT} 
 -   \frac{\bPhi_{\hS}^{\sT} \tbV}{\sqrt{n}} \left(\frac{\bLambda_S\tbV^{\sT} \bPp f(\bx^r, \by; r)}{n} \right)  \right]  \nonumber \\
& = \bZ_{\hS}\E[\bL f( \balpha_r \bU + \bbeta_r \bL + \btau_r \bG_0 ,Y; r)^{\sT}]  \nonumber \\
& \    - \E[\bL f( \balpha_{r-1} \bU + \bbeta_{r-1} \bL + \btau_{r-1} \bG_0 ,Y; r-1)^{\sT}] \E \{ f'(\balpha_r \bU + \bbeta_r \bL + \btau_r \bG_0, Y; r)^{\sT} \}  \nonumber \\
&  \    -  \bProd_0\bLambda_S \E \{  \bU f(\balpha_r \bU + \bbeta_r \bL + \btau \bG_0, Y ; r )^\sT \}  + \bProd_0\bLambda_S\bProd_0^{\sT} \, \E \{  \bL  f(\balpha_r \bU + \bbeta_r \bL + \btau \bG_0, Y ; r )^\sT \}  +o_n(1) \nonumber \\
& = \bbeta_{r+1}^{\sT} + o_n(1), \nonumber
 \end{align}
 where  the last equality follows from the definition of $\bbeta_{r+1}$ in \myeqref{eq:BetaRec_gen}.
 
\vspace{0.1cm}

\noindent\emph{Proof of \myeqref{eq:lim_alpht_gen}}. Follows from  \myeqref{eq:tbUsTbPp} and the definition of $\balpha_{r+1}$ in Eq.~(\ref{eq:AlphaRec_gen}).

\vspace{0.1cm}

\noindent\emph{Proof of \myeqref{eq:Wg_diff_norm_gen}}. We have
\begin{align}
&  \lim_{n \to \infty}  \frac{1}{n} \norm{\tbW \left[   g(\bx^r,\by;r) -  \tilg(\bfs^r +  \tbV \balpha_r^{\sT}  +   \sqrt{n} \bPhi_{\hS} \, \bbeta_r^{\sT}, \, \by; r )  \right] }_{F}^2\nonumber \\
 & \leq   {\| \tbW \|}_{\op}^2 \, \frac{2}{n} \left[ \| g(\bx^r,\by;r) - \tilg(\bx^r,\by;r) \|_F^2 + 
  \| \tilg(\bx^r,\by;r) - \tilg(\bfs^r +  \tbV \balpha_r^{\sT}  +   \sqrt{n} \bPhi_{\hS} \, \bbeta_r^{\sT}, \, \by; r ) \|_F^2 \right].
\end{align}
Recalling that ${\| \tbW \|}_{\op}^2 \to 4$ almost surely, we bound 
\begin{align}
\frac{2}{n} \| g(\bx^r,\by;r) - \tilg(\bx^r,\by;r) \|_F^2  &  \leq \| \bPhi_{\hS} \|_F^2 \norm{ \frac{\bPhi_{\hS}^{\sT}f(\bx^r,\by;r) }{n}  - 
\E \{ \bL f(\balpha_t \bU  +  \bbeta_t \bL  + \btau_t \bG_0, \, Y \, ;t)^\sT \} }_F^2  \nonumber \\
&  =o_n(1),
\end{align} 
where we used \myeqref{eq:bPhifT} together with $\| \bPhi_{\hS} \|_F^2 = k_*$. Finally, using the Lipschitz property of $f$ we have 
\begin{align}
& \frac{1}{n} \| \tilg(\bx^r,\by;r) - \tilg(\bfs^r +  \tbV \balpha_r^{\sT}  +   \sqrt{n} \bPhi_{\hS} \, \bbeta_r^{\sT}, \, \by; r ) \|_F^2 \\
&  =  \frac{1}{n}\| f(\bx^r,\by;r) - f(\bfs^r +  \tbV \balpha_r^{\sT}  +   \sqrt{n} \bPhi_{\hS} \, \bbeta_r^{\sT}, \, \by; r ) \|_F^2 
 \leq C \frac{\| \bDelta_r \|^2_F}{n} 
\end{align}
The proof is completed by noting that ${\| \bDelta_r \|^2_F}/{n} \to 0$ a.s. by the induction hypothesis.

\vspace{0.1cm}

\noindent\emph{Proof of \myeqref{eq:btg_diff_norm_gen}}:  We have 
\begin{align}
 \frac{1}{n} \norm{ \tilg (\bfs^{r-1} +  \tbV \balpha_{r-1}^{\sT} +  \sqrt{n} \bPhi_{\hS} \, \bbeta_{r-1}^{\sT}, \, \by ; r-1) \tilde{\sB}_r^{\sT} - g(\bx^{r-1},\by;r-1)\, \sB_r^{\sT} }_F^2   \leq 3 \left[ T_1 +T_2 + T_3 \right]  \label{eq:T1T2T3_gen}
\end{align}
where
\begin{align}
 T_1 & = \frac{1}{n} \norm{ \tilg ( \bfs^{r-1} +  \tbV \balpha_{r-1}^{\sT} +  \sqrt{n} \bPhi_{\hS} \, \bbeta_{r-1}^{\sT}, \, \by ; r-1) 
 ( \tilde{\sB}_r^{\sT}  - {\sB}_r^{\sT}) }_F^2, \nonumber \\
T_2 & =  \frac{1}{n} \norm{ \tilg (\bfs^{r-1} +  \tbV \balpha_{r-1}^{\sT} +  \sqrt{n} \bPhi_{\hS} \, \bbeta_{r-1}^{\sT}, \, \by ; r-1)  - \tilg(\bx^{r-1},\by;r-1)}_F^2  {\| {\sB}_r^{\sT} \|}_F^2, \nonumber \\ 
T_3 & =  \frac{1}{n} \norm{ g(\bx^{r-1},\by;r-1) -  \tilg(\bx^{r-1},\by;r-1)}_F^2  {\| {\sB}_r^{\sT} \|}_F^2. \label{eq:T3_def_gen}
\end{align}
We now show that $T_1, T_2, T_3$ each  tend to zero almost surely. Using the Lipschitz property of $\tilde{g}$,  the term $T_1$ can be bounded as 
\begin{align}
T_1 
\leq \frac{C}{n} \norm{\bfs^{r-1} +  \tbV \balpha_{r-1}^{\sT} +  \sqrt{n} \bPhi_{\hS} \, \bbeta_{r-1}^{\sT} }_F^2 \norm{ \tilde{\sB}_r  -  \sB_r }_F^2. 
\label{eq:bT1_gen}
\end{align}
From the induction hypothesis in \myeqref{eq:norm_finite_gen_2} for $t=(r-1)$, we have
$ \limsup_{n \to \infty} \frac{1}{n} \| \bfs^{r-1} +  \alpha_{r-1} \bx_0  + \beta_{r-1} \|_F^2$ $< \infty$, almost surely.
The result in \myeqref{eq:st_SE0_gen}  implies that the empirical distribution of $\bfs^{r} +  \tbV \balpha_{r}^{\sT} +  \sqrt{n} \bPhi_{\hS} \, \bbeta_{r}^{\sT}$ converges in $W_2$ to the distribution  of $\balpha_r \bU + \bbeta_r \bL + \btau_r \bG_0$. Combining this with the Lipschitz property of $f( \,\cdot\, ; r)$ as in \cite[Lemma 5]{BM-MPCS-2011}, we have 
\begin{align}
\tilde{\sB}_r    =  
 \frac{1}{n}\sum_{i=1}^n f'( (\bfs^r +  \tbV \balpha_r^{\sT}  +   \sqrt{n} \bPhi_{\hS} \, \bbeta_r^{\sT})_i , y_i ; r)   
 =  \E \{ f'(\balpha_r \bU + \bbeta_r \bL + \btau_r \bG_0, Y; r) \}  +o_n(1).
\label{eq:fprs_conv_gen}
\end{align}
Noting from Eqs.  \eqref{eq:fpr_conv_gen} and \eqref{eq:fprs_conv_gen} that $\sB_r$ and $\tilde{\sB}_r$ have the same limiting value,  from \myeqref{eq:bT1_gen}  we conclude that
$T_1 = o_n(1)$.    Next, using the Lipschitz property of  $\tilg$,  $T_2$ can be bounded as
  \begin{align}
T_2  \leq
  \frac{ C}{n} \norm{\bDelta^{r-1}}_F^2 \norm{\sB_r}_F^2.  
  \label{eq:tilgs_tilgx_gen}
  \end{align}
By the induction hypothesis, we have ${\| \bDelta^{r-1} \|}^2/n \to 0$ almost surely. Furthermore, $\norm{\sB_r}_F^2$ tends to a  finite limit almost surely (due to \myeqref{eq:fpr_conv_gen}).  We therefore have  $T_2 = o_n(1)$. 
Finally, we have
 \begin{align}
 T_3 & \leq 
{\| {\sB}_r^{\sT} \|}_F^2   \norm{\bPhi_{\hS}}_F^2 \norm{ \E[\bL f( \balpha_{r-1} \bU + \bbeta_{r-1} \bL + \btau_{r-1} \bG_0 ,Y; r-1)^{\sT}] \} -   \frac{\sqrt{n} \bPhi_{\hS}^{\sT} f(\bx^{r-1}, \by; r-1)}{n} }_F^2   \nonumber  \\
& = o_n(1), \label{eq:br_gtil_g_gen}
 \end{align}
where the last inequality follows from  \myeqref{eq:bPhifT}.   Therefore, we have shown that $T_1, T_2, T_3$ are all $o_n(1)$ and the result follows from \myeqref{eq:T1T2T3_gen}.

\vspace{0.1cm}

\noindent
\emph{Proof of \myeqref{eq:bdel_norm_gen}}: Using the definition of $\bdel^t$ in \myeqref{eq:del_def_gen} and the recursion for $\bfs^{t+1}$ defined in \myeqref{eq:st1_def_gen}, we can write
\begin{align}
 \bdel^r  & = \bPhi_{\hS} \left[  \bPhi_{\hS}^{\sT} \tbW g(\bx^r,\by;r) \right] \nonumber \\
& = \bPhi_{\hS} \Big[   \bPhi_{\hS}^{\sT} \Big( \bfs^{r+1} + \tilg( \bfs^{r-1} +  \tbV \balpha_{r-1}^{\sT} +  \sqrt{n} \bPhi_{\hS} \, \bbeta_{r-1}^{\sT}, \, \by ; r-1) \tilde{\sB}_r^{\sT} \Big)  \nonumber  \\ 
& \quad  +  \bPhi_{\hS}^{\sT} \tbW \Big( g(\bx^r,\by;r) -   \tilg (\bfs^{r} +  \tbV \balpha_{r}^{\sT} +  \sqrt{n} \bPhi_{\hS} \, \bbeta_{r}^{\sT}, \, \by ; r) \Big) \Big]. 
\end{align}
Therefore $\norm{\bdel^r}_F^2/n  \leq 3( T_1 + T_2 + T_3)$, where
\begin{align}
&  T_1=  \frac{1}{n} \| \bPhi_{\hS}^{\sT} \bfs^{r+1} \|_F^2, \qquad 
T_2  =   \frac{1}{n} \| \bPhi_{\hS}^{\sT}  \tilg (\bfs^{r-1} +  \tbV \balpha_{r-1}^{\sT} +  \sqrt{n} \bPhi_{\hS} \, \bbeta_{r-1}^{\sT}, \, \by ; r-1)  \|_F^2 
\| \tilde{\sB}_r^{\sT} \|_F^2, \nonumber  \\
& T_3 = \frac{1}{n} \| \bPhi_{\hS}^{\sT}    \tbW ( g(\bx^r,\by;r) -   \tilg (\bfs^{r} +  \tbV \balpha_{r}^{\sT} +  \sqrt{n} \bPhi_{\hS} \, \bbeta_{r}^{\sT}, \, \by ; r) )   \|_F^2. \label{eq:T1T2T3_last}
\end{align}
We now show that $\norm{\bdel^r}_F^2/n=0$ by showing that $T_1, T_2, T_3$  are each $o_n(1)$.  

Applying the state evolution result  in \myeqref{eq:st_SE0_gen} to the pseudo-Lipschitz function $\psi(\bfs_i^{r+1},\tbv_i, \sqrt{n} \bPhi_{\hS,i} ,y_i)$  $= \sqrt{n} \bPhi_{\hS,i}^{\sT} \bfs_i^{r+1}$, we obtain
\beq
\frac{1}{n} \sqrt{n} \bPhi_{\hS}^{\sT} \bfs^{r+1}  = \E \{ \bL (\btau_{r+1} \bG_{0})^{\sT} \} + o_n(1) = o_n(1),
\eeq
where the last inequality holds because $\bL$ and $\bG_0$ are independent. Therefore $T_1 = o_n(1)$.

For the second term $T_2$, using the definition of $\tilg$ in \myeqref{eq:tilg_def_gen} (and for brevity, dropping the time-index $(r-1)$ within the function) we write
\begin{align}
& \frac{1}{\sqrt{n}} \bPhi_{\hS}^{\sT}  \tilg (\bfs^{r-1} +  \tbV \balpha_{r-1}^{\sT} +  \sqrt{n} \bPhi_{\hS} \, \bbeta_{r-1}^{\sT}, \, \by) \nonumber \\
&  = \frac{1}{\sqrt{n}}    \bPhi_{\hS}^{\sT} \left[ f( \bfs^{r-1} +  \tbV \balpha_{r-1}^{\sT} +  \sqrt{n} \bPhi_{\hS} \, \bbeta_{r-1}^{\sT}, \, \by)  - \sqrt{n} \bPhi_{\hS}\, \E \{ \bL f(\balpha_{r-1} \bU  +  \bbeta_{r-1} \bL  + \btau_{r-1} \bG_0, \, Y )^\sT \} \right] \nonumber \\ 
& = \frac{ \sqrt{n} \bPhi_{\hS}^{\sT} f(\bfs^{r-1} +  \tbV \balpha_{r-1}^{\sT} +  \sqrt{n} \bPhi_{\hS} \, \bbeta_{r-1}^{\sT}, \by)}{n} -  
\E \{ \bL f(\balpha_{r-1} \bU  +  \bbeta_{r-1} \bL  + \btau_{r-1} \bG_0, \, Y )^\sT \}  \nonumber \\
& = o_n(1), \label{eq:phi1_gtil_gen}
\end{align}
where the last  equality is obtained by applying the state evolution result  \myeqref{eq:st_SE0_gen} to the pseudo-Lipschitz function $\psi(\bfs_i^{r+1},\tbv_i, \sqrt{n} \bPhi_{\hS,i} ,y_i) = \sqrt{n} \bPhi_{\hS,i} f( \bfs^{r-1}_i + \balpha_{r-1}\tbv_i + \bbeta_{r-1} \sqrt{n}  \bPhi_{\hS,i}  )^{\sT}$. Using \myeqref{eq:phi1_gtil_gen} and recalling from \myeqref{eq:fprs_conv_gen} that $\tilde{\sB}_r$ converges to a finite value, we conclude that $T_2= o_n(1)$. Finally we bound $T_3$ in \myeqref{eq:T1T2T3_last} using Cauchy-Schwarz as follows:
\begin{align}
& T_3  \leq  \| \tbW \|^2_{\op} \,  \frac{1}{n} \norm{g(\bx^r,\by;r) -   \tilg (\bfs^{r} +  \tbV \balpha_{r}^{\sT} +  \sqrt{n} \bPhi_{\hS} \, \bbeta_{r}^{\sT}, \, \by ; r)}_F^2 \nonumber \\
&  \leq  \| \tbW \|^2_{\op} \left[  \frac{2}{n} \norm{g(\bx^r,\by;r) -  \tilg(\bx^r,\by;r)}_F^2 +    
\frac{2}{n} \norm{\tilg(\bx^r,\by;r) - \tilg (\bfs^{r} +  \tbV \balpha_{r}^{\sT} +  \sqrt{n} \bPhi_{\hS} \, \bbeta_{r}^{\sT}, \, \by ; r) }_F^2 \right]  \nonumber  \\
& = o_n(1),
\end{align}
where the last inequality holds because  $\| \tbW \|^2_{\op}=4 + o_n(1)$, and   $\frac{1}{n} \norm{g(\bx^r,\by;r) -  \tilg(\bx^r,\by;r)}_F^2$ and $\frac{1}{n} \norm{\tilg(\bx^r,\by;r) - \tilg (\bfs^{r} +  \tbV \balpha_{r}^{\sT} +  \sqrt{n} \bPhi_{\hS} \, \bbeta_{r}^{\sT}, \, \by ; r) }_F^2$ are each $o_n(1)$ from the arguments in \myeqref{eq:T3_def_gen} -- \myeqref{eq:br_gtil_g_gen}. This finishes the proof of \myeqref{eq:bdel_norm_gen}.

\vspace{0.1cm}

Thus based on the induction hypothesis we have shown that \myeqref{eq:lim_beta_r1_gen} -- \myeqref{eq:bdel_norm_gen} hold,  therefore \myeqref{eq:psi_xs_gen} and \myeqref{eq:norm_Del_gen} hold for $t=(r+1)$.  Finally, we need to verify that the conditions in Eqs.~\eqref{eq:norm_finite_gen_1},  \eqref{eq:norm_finite_gen_2}  also hold for $t=(r+1)$. But these immediately follow from \myeqref{eq:st_SE0_gen} and \myeqref{eq:psi_xs_gen} with  $t=r$  by considering  the pseudo-Lipschitz function $\psi(\bx_i^t,\tbu_i, \bPhi_{\hS,i} ,y_i)=  (\bx_i^t)^{\sT} \bx_i^t$.  This completes the proof of Theorem \ref{thm:Main}.

\subsection{Conditioning lemma}

Let $\bA$ be a spiked random matrix with distribution as per Eq.~(\ref{eq:SpikedDef}), with
$\lambda_1\ge \dots \lambda_{k_+}>1>\lambda_{k_+}$ and $\lambda_{k-k_-}>-1> \lambda_{k-k_-+1}\ge \dots\ge \lambda_k$. 
Recall that $\bz = (z_1,\dots,z_{n})$ are the  ordered eigenvalues of $\bA$ with $\bphi_1$,\dots $\bphi_n$
being the corresponding eigenvectors. Also recall that $\hS =  \{1,\dots, k_+\}\cup \{n-k_-+1,\dots,n\}$, $S = \{1,\dots, k_+\}\cup \{k-k_-+1,\dots,k\}$, and $k_* = k_++k_-$.   Let $\blambda_S = (\lambda_i)_{i\in S}$,  $\bz_{\hS} = (z_i)_{i\in\hS}$,  and 
$\bPhi_{\hS} =(\bphi_i)_{i\in\hS}$ (we will view $\bPhi_{\hS}$ as a matrix with dimensions $n\times k_*$, with columns given by the $\bphi_i$'s).
\begin{lemma}\label{lemma:Distr}
With the above definitions,  let 
\begin{align}
\tbA &\equiv \sum_{i\in \hS}z_i\bphi_i\bphi_i^{\sT} + \sum_{i=1}^k\lambda_i \bPp\bv_i\bv_i^{\sT} \bPp+\bPp\tbW\bPp\, ,\label{eq:SpikedModelModified}
\end{align}
where $\bPp$ is the projector onto the orthogonal complement of the space spanned by $(\bphi_i)_{i\in\hS}$,
and $\tbW\sim\GOE(n)$ is independent of $\bW$.
Let $\rho(x) = x+x^{-1}$, $\eta(x) = 1-x^{-2}$, and define the event
\begin{align}
\cE_{\eps} \equiv\left\{ \; \max_{i\le k_*}\big|(\bz_{\hS})_i-\rho(\lambda_{S,i})\big|\le \eps,\;\; \min_{i\in S}\big\|\bPhi_{\hS}^{\sT}\bv_i\big\|_2^2-\eta(\lambda_i)\ge
-\eps\right\}\, .  \label{eq:cE_eps}
\end{align}
Then there exists a constant $\eps_0>0$ such that for all $\eps\in (0,\eps_0)$ there is $c(\eps)>0$, such that
\begin{align}
\prob\big\{\cE_{\eps}\big\} \ge 1-\frac{1}{c(\eps)}\, e^{-n c(\eps)}\, .  \label{eq:SpikeEvent}
\end{align} 
Further (for a suitable version of the conditional probabilities):
\begin{align}
\sup_{(\bz_{\hS},\bPhi_{\hS})\in \cE_{\eps}}\Big\|\prob\big(\bA\in\,\cdot\,\big|\bz_{\hS},\bPhi_{\hS}\big) -
\prob\big(\tbA\in\,\cdot\,\big|\bz_{\hS},\bPhi_{\hS}\big)\Big\|_{\sTV}\le \frac{1}{c(\eps)}\, e^{-n c(\eps)}\, .  \label{eq:CompareRM}
\end{align}
\end{lemma}
\begin{proof}
The probability lower bound \myeqref{eq:SpikeEvent} follows for instance from \cite{benaych2012large}. 

In order to prove \myeqref{eq:CompareRM}, we will proceed in two steps: first conditioning on a
given set of eigenvectors (without ordering) and then conditioning on the event that these are actually the outlier eigenvectors. 
To reduce  book-keeping, we will assume that $k_-=0$ (and hence $k_+=k_*$): all large rank-one perturbations are positive semidefinite.

Fix real numbers $\xi_1>\xi_2>\dots > \xi_{k_*}$ and $\bu_1,\dots, \bu_{k_*}\in\reals^n$ an orthonormal set.
We claim that the conditional distribution\footnote{Formally, we consider a realization of the conditional distribution of $\bA$ given the random vectors $(\bx_i)_{i\le k_*}$ defined by
$\bx_1 = \bA\bu_1$, \dots $\bx_{k_*}=\bA\bu_{k_*}$, evaluated at $\bx_i = \xi_i \bu_i$, $i\le k_*$.} of $\bA$ given that $\bxi= (\xi_i)_{i\le k_*}$ are eigenvalues with eigenvectors $\bU =(\bu_i)_{i\le k_*}$ 
is the same as the one of 
\begin{align}
\tbA_{\bxi,\bU} &\equiv \sum_{i\in \hS}\xi_i\bu_i\bu_i^{\sT} + \sum_{i=1}^k\lambda_i \bPp_{\bU}\bv_i\bv_i^{\sT} \bPp_{\bU}+\bPp_{\bU}\tbW\bPp_{\bU}\, ,
\end{align}
where $\bPp_{\bU}$ is the projector onto the orthogonal complement of ${\rm span}(\bu_1,\dots,\bu_{k_*})$.  To prove this claim, note that by rotational 
invariance of the $\GOE(n)$ distribution, it is sufficient to consider the case in which the eigenvectors coincide with the first $k_*$ vectors of the canonical 
basis: $\bu_1=\be_1$, \dots $\bu_{k_*} = \be_{k_*}$. Conditioning on this is equivalent to conditioning on the event that the entries in the first $k_*$ rows and columns 
of $\bA$ are all equal to $0$ except on the diagonal where they are $A_{ii}=\xi_i$.  By independence of the entries of $\bA$, the distribution of the
block $\bA_{QQ}$ with $Q=\{k_*+1,\dots, n\}$ (or, equivalently, the matrix $\bPp_{\bU}\bA\bPp_{\bU}$) is not changed by the conditioning. Hence the  distribution of the block $\bA_{QQ}$ is the same as the distribution of
$\sum_{i=1}^k\lambda_i (\bv_i\bv_i^{\sT})_{QQ} +\tbW_{QQ}$ for $\tbW$ independent of $\bW$, which proves our claim.

The conditional distribution of $\bA$ given the ordered eigenvalues $(z_i)_{i\le k_*}$ and the corresponding eigenvectors $(\bphi_i)_{i\le k_*}$ is therefore the same as the one of $\tbA$ of Eq.~(\ref{eq:SpikedModelModified})
conditioned on the event that the largest eigenvalues of $\tbA$ are $z_1,\dots,z_{k_*}$. Letting $\tbA^{\perp}= \sum_{i=1}^k\lambda_i \bPp\bv_i\bv_i^{\sT} \bPp+\bPp\tbW\bPp$, 
and denoting by $z_{\max}(\tbA^{\perp})$ its top eigenvalue, we therefore have
\begin{align}
\prob\big(\bA\in\; \cdot \; \big|\bz_{\hS},\bPhi_{\hS}\big) =\prob\big(\tbA\in\; \cdot \; \big|\bz_{\hS},\bPhi_{\hS}, z_{\max}(\tbA^{\perp}) < z_{k_*}\big) \, ,\label{eq:RepresentationConditional}
\end{align}
Note that, defining $\tbv_i = \bPp\bv_i/\|\bPp\bv_i\|_2$, $\tlambda_i =\lambda_i \|\bPp\bv_i\|^2_2$
\begin{align}
\tbA^{\perp}= \sum_{i=1}^k\tlambda_i \tbv_i\tbv_i^{\sT}  +\bPp\tbW\bPp\, .
\end{align}
Note that for $\bz_{\hS},\bPhi_{\hS}\in \cE_{\eps}$, and all $\eps$ small enough, we have 
\beq \{z_{\max}(\tbA^{\perp}) < z_{k_*}\}\supseteq \{ z_{\max}(\tbA^{\perp}) < 2+\eps\}. \label{eq:zmaxE} \eeq
However, by rotational invariance of $\GOE(n)$, the eigenvalues of $\tbA^{\perp}$ are distributed as the ones of $\hbA\in\reals^{(n-k_*)\times (n-k_*)}$, defined by
\begin{align}
\hbA= \sum_{i=1}^k\tlambda_i \br_i\br_i^{\sT} +\sqrt{\frac{n}{n-k_*}}\hbW\, .
\end{align}
where $\br_1,\dots,\br_{k_*}\in\reals^{n-k_*}$ are an orthonormal set and $\hbW\sim\GOE(n-k_*)$. 
Note that, on $\cE_{\eps}$, and for all $\eps$ small enough, $0\le \tlambda_i =  \lambda_i(1-\|\bPhi_{\hS}^{\sT}\bv_i\|_2^2)\le (1-\eps)$
for $i\le k_*$. Hence $\hbA$ is a subcritical spiked model. Consequently, using \myeqref{eq:zmaxE} and applying again the result  from \cite{benaych2012large}, we obtain, on $\cE_{\eps}$,
\begin{align}
\prob\big(z_{\max}(\tbA^{\perp}) \ge  z_{k_*}|\bz_{\hS},\bPhi_{\hS}\big) \le \prob\big(\|\hbA\|_{\op}\ge 2+\eps |\bz_{\hS},\bPhi_{\hS}\big)\le \frac{1}{c(\eps)}\, e^{-c(\eps)n}\, .
\end{align}
Finally, using \myeqref{eq:RepresentationConditional}, we get, for a suitable $c_*(\eps)$,
\begin{align}
\prob\big(\bA\in\; \cdot \; \big|\bz_{\hS},\bPhi_{\hS}\big)\le \frac{\prob\big(\tbA\in\; \cdot \; \big|\bz_{\hS},\bPhi_{\hS}\big) }
{1-\prob\big(z_{\max}(\tbA^{\perp}) \ge z_{k_*}\big|\bz_{\hS},\bPhi_{\hS} \big) }\le \prob\big(\tbA\in\; \cdot \; \big|\bz_{\hS},\bPhi_{\hS}\big) + \frac{1}{c_*(\eps)}\, e^{-nc_*(\eps)}\, ,
\end{align}
and
\begin{align}
\prob\big(\bA\in\; \cdot \; \big|\bz_{\hS},\bPhi_{\hS}\big) \ge \prob\big(\tbA\in\; \cdot \; , z_{\max}(\tbA^{\perp}) < z_{k_*}\big|\bz_{\hS},\bPhi_{\hS}\big)\ge 
\prob\big(\tbA\in\; \cdot \; \big|\bz_{\hS},\bPhi_{\hS}\big) -\frac{1}{c_*(\eps)}\, e^{-nc_*(\eps)}\, .
\end{align}
This completes the proof of \myeqref{eq:CompareRM}.
\end{proof}

\section{Asymptotics of the eigenvectors of spiked random matrices}

In this appendix, we collect some consequences of known facts about the eigenvectors of random matrices
distributed according to the spiked model (\ref{eq:SpikedDef}). We copy the definition here for the reader's convenience:
\begin{align}
\bA &= \sum_{i=1}^k\lambda_i(n) \bv_i\bv_i^{\sT} + \bW\,  \nonumber \\
& = \bV \bLambda \bV^{\sT} + \bW\,  \label{eq:SpikedDef2a} \\
& \equiv \bA_0 +\bW\, . \label{eq:SpikedDef2}
\end{align}
Here $\bW\sim \GOE(n)$, $\bv_i = \bv_i(n)\in\reals^n$ are orthonormal vectors and the values $\lambda_i(n)$ have finite limits as $n\to \infty$, that we denote by $\lambda_i$. 
Further, $\lambda_1\ge \dots \lambda_{k_+}>1>\lambda_{k_++1}$ and $\lambda_{k-k_-}>-1> \lambda_{k-k_-+1}\ge \dots\ge \lambda_k$.
Let $S \equiv (1,\dots, k_+,k-k_-+1,\dots,k)$, $k_* = k_++k_-$ and $\hS \equiv  (1,\dots, k_+,n-k_-+1,\dots,n)$. 
Denote by $\bphi_1,\dots,\bphi_n$ the eigenvectors of $\bA$, with corresponding eigenvalues $z_1\ge  z_2\ge \dots\ge z_n$.

The sets of matrices $\cR(\bLambda)\subseteq \reals^{S\times [k]}$ and $\cR(\bLambda)\subseteq \reals^{S\times S}$ are defined 
as in Section \ref{sec:Symmetric}.
\begin{lemma}\label{lem:ConvergenceEigenvectors}
Let  $\bA$ be the random matrix of Eq.~(\ref{eq:SpikedDef2}).
For $\eps>0$ and $\eta_n\ge n^{-1/2+\eps}$ such that $\eta_n\to 0$ as  $n\to\infty$, define the set of matrices
\begin{align}
\cG_n(\bLambda)\equiv \Big\{\bM\in \reals^{k_*\times k}:\, \min_{\bR\in\cR(\bLambda)}\|\bM- (\id-\bLambda_{k_*}^{-2})^{1/2}\bR\|_{F}\le \eta_n\Big\}\, ,
\end{align}
Further, assume that the joint empirical distribution of the vectors $(\sqrt{n}\bv_\ell(n))_{\ell\in S}$, has a limit
in Wasserstein-$2$ metric. 
Namely, if we let $\tbv_i = (\sqrt{n} v_{\ell,i})_{\ell\in S}\in\reals^{k_*}$, then there exists a random vector $\bU$ taking values
in $\reals^{k_*}$  with law $\mu_{\bU}$, such that
\begin{align}
\frac{1}{n}\sum_{i=1}^n\delta_{\tbv_i} \towass \mu_{\bU}\, .
\end{align}
Further define, for $i\le n$,  $\tbphi_i = (\sqrt{n}\varphi_{\ell,i})_{\ell\in\hS}\in \reals^{k_*}$. 

Let $\bProd\equiv \bPhi_{\hS}^{\sT}\bV\in\reals^{k_*\times k}$ where $\bPhi_{\hS}\in\reals^{n\times k_*}$ is the matrix with columns $(\bphi_i)_{i\in\hS}$ and $\bV\in \reals^{n\times k}$ is the matrix 
with columns $(\bv_i)_{i\in [k]}$. Denote by $\bProd_0\in \reals^{k_*\times k_*}$ the submatrix corresponding to the  $k_*$ columns of $\bProd$ with index in $S$, and let $\tbProd_0 = (\id-\bProd_0\bProd_0^{\sT})^{1/2}$.
Then, for any pseudo-Lipschitz function $\psi:\reals^{k_*+k_*}\to \reals$,  we have (almost surely)
\begin{align}
\lim_{n\to\infty}\left|\frac{1}{n}\sum_{i=1}^n\psi(\tbphi_i,\tbv_i)
      - \E\big\{\psi(\bProd_0\bU+\tbProd_0\bG,\bU)\big\}\right|= 0\, .
\label{eq:ConvergenceEvectors}
\end{align}
where  expectation is with respect to $\bU\sim \mu_{\bU}$ independent of
$\bG\sim\normal(0,\id_{k_*})$.

Further $\prob(\bProd\in \cG_n(\bLambda)) \ge 1-n^{-A}$ for any $A>0$ provided $n>n_0(A)$,  and $\bProd$ converges in distribution to  $(\id-\bLambda_{k_*}^{-2})^{1/2}\bR$, with $\bR$  Haar distributed on $\cR(\bLambda)$.
\end{lemma}
Before proving this lemma, we state and prove a simple but useful estimate.
\begin{lemma}\label{lemma:Simple}
Let $\psi:\reals^q\to\reals$, $\psi\in\PL(2)$, and $\bU,\bDelta\in\reals^{n\times q}$ be matrices with rows denoted by $\bu_i$, $\bdelta_i$, respectively, for $i \in [n]$.
Then, there exists a constant $C$ (uniquely dependent on $q$ and on the function $\psi$) such that 
\begin{align}
\left|\frac{1}{n}\sum_{i=1}^n\psi(\bu_i+\bdelta_i) - \frac{1}{n}\sum_{i=1}^n\psi(\bu_i)  \right|\le C\, \frac{\|\bDelta\|_F}{\sqrt{n}}\,\left(1+\frac{\|\bDelta\|_F}{\sqrt{n}}+\frac{\|\bU\|_F}{\sqrt{n}}
\right)\, .
\end{align}
\end{lemma}
\begin{proof}
Since $\psi\in\PL(2)$, we have
\begin{align}
\left|\frac{1}{n}\sum_{i=1}^n\psi(\bu_i+\bdelta_i) - \frac{1}{n}\sum_{i=1}^n\psi(\bu_i)  \right|&
\le \frac{1}{n}\sum_{i=1}^n\big|\psi(\bu_i+\bdelta_i) - \psi(\bu_i) \big|\\
& \le \frac{C}{n}\sum_{i=1}^n\|\bdelta_i\|_2\big(1+\|\bu_i\|_2+\|\bdelta_i\|_2\big)\, ,
\end{align}
and the claim follows by applying Cauchy-Schwarz inequality.
\end{proof}
\begin{proof}[Proof of Lemma \ref{lem:ConvergenceEigenvectors}]
Decomposing  $\bPhi = \bPhi_{\hS}$ in the component along $\bV$ and the one orthogonal, we have
$\bPhi = \bV\bProd^{\sT} + \bPhi_{\perp}$ where $\bProd = \bPhi^{\sT} \bV\in \reals^{k_*\times k}$  and $\bV^{\sT}\bPhi_{\perp} = \bzero$.  Further
taking the singular value decomposition $\bPhi_{\perp} =  \bV_{\perp} \bSigma_{\perp}\bU_{\perp}^{\sT}$, we get
\begin{align}
\bPhi = \bV\bProd^{\sT} + \bV_{\perp} \tbProd^{\sT}\label{eq:DecompositionEvectors}
\end{align}
where $\bV_{\perp}\in\reals^{n\times k_*}$ is an orthogonal matrix with $\bV^{\sT}\bV_{\perp} = \bzero$,
$\bV^{\sT}_{\perp}\bV_{\perp} = \id_{k_*}$ and $\tbProd\in\reals^{k_*\times k_*}$. Notice that $\tbProd$ is only defined up to right multiplication
by a $k_*\times k_*$ orthogonal matrix. In order to fix this freedom, notice that, by orthogonality of $\bPhi$, we get $\tbProd\tbProd^{\sT} = \id_{k_*}-
\bProd\bProd^{\sT}$. We therefore select $\tbProd$ to  be a symmetric positive semi-definite square root   $\tbProd= (\id_{k_*}-
\bProd\bProd^{\sT})^{1/2}$.

Let $W_0, W_1,\dots ,W_{\ell}\subseteq \reals^n$ be  the eigenspaces corresponding to distinct eigenvalues of $\bA_0$  (cf.~ Eq.(\ref{eq:SpikedDef2}),
with $W_0$ corresponding to the null eigenvalue.
In other words, $W_0$ is the orthogonal complement of $\spn(\bv_1,\dots,\bv_k)$. Further letting $\lambda_{(1)} >\lambda_{(2)} >\dots \lambda_{(\ell)}$ the distinct
eigenvalues of $\bA_0$, and $S(j) \equiv\{i\in [k]:\; \lambda_i = \lambda_{(j)}\}$, each $W_j =\spn(\bv_i:\; i\in S(j)\}$.

Let $\tcR(\bLambda) \subseteq \reals^{k\times k}$ denote the group of orthogonal matrices $\bR$ such that $R_{ij}=0$ if $\lambda_i\neq\lambda_j$.
We note that replacing $\bV$ by $\bV \bR$ in \myeqref{eq:SpikedDef2a} for any $\tcR(\bLambda)$ leaves $\bA$ unchanged.  Note that each such $\bR$ corresponds to a unique orthogonal matrix $\bT\in\reals^{n\times n}$ which leaves invariant $W_0,\dots, W_{\ell}$, and $\bV \bR = \bT \bV$. We therefore have $\bT\bA\bT^{\sT} \ed \bA$, and therefore $\bT\bPhi\ed \bPhi$. 
Also, this symmetry group acts transitively on the Stiefel manifold of orthogonal matrices $\bV_{\perp}\in\reals^{n\times k_*}$ with columns in $W_0$, to be denoted by $\cS_{k_*}(W_0)$.
We conclude that $\bV_{\perp}$ is Haar-distributed on $\cS_{k_*}(W_0)$.
Further,  $\bProd \ed \bProd \bR$ where $\bR$ is Haar distributed on $\tcR(\bLambda)$.

It follows from \cite[Proposition 5.1.(a)]{benaych2011eigenvalues} and  \cite[Theorem 3.3]{knowles2014outliers} that for any $A>0$, the following holds with probability 
larger than $1-n^{-A}$ for $n\ge n_0(A)$:
\begin{align}
\sum_{j\in [k]: \lambda_{j}\neq \lambda_{i}}|\Omega_{ij}|& \le \eta_n\, \;\;\;\;\; \forall i\in [k_*]\, ,\\
\left|\sigma_{\max}(\bProd_{S(i),S(i)}) -\Big(1-\frac{1}{\lambda_{(i)}^2}\Big)\right| &\le \eta_n\, \;\;\;\;\; \forall i\in [\ell], \, |\lambda_{(i)}|>1\, ,\\
\left|\sigma_{\min}(\bProd_{S(i),S(i)}) -\Big(1-\frac{1}{\lambda_{(i)}^2}\Big)\right| &\le \eta_n\, \;\;\;\;\; \forall i\in [\ell], \, |\lambda_{(i)}|>1\, .
\end{align}
This implies that there exists orthogonal matrices $\bL_0\in \cR_*(\bLambda)$, $\bR_0\in \tcR(\bLambda)$ such that
\begin{align}
\|\bProd- \bL_0(\id_{k_*}-\bLambda^{-2}_{k_*})^{1/2}\bR_0\|_{F}\le C\, \eta_n\, .  \label{eq:OmL0R0}
\end{align}
Since $\bL_0$ commutes with $\bLambda_{k_*}$,  \myeqref{eq:OmL0R0} implies (after rescaling $\eta_n$ by a constant)  that $\bProd\in\cG_n(\bR)$ with the claimed probability. 
Further, letting $\bR$ be  Haar distributed on $\tcR(\bLambda)$,
\begin{align}
\bProd &\ed \bProd \bR =  \bL_0(\id_{k_*}-\bLambda^{-2}_{k_*})^{1/2}\bR_0\bR +\bDelta\\
    & \ed (\id_{k_*}-\bLambda^{-2}_{k_*})^{1/2}\bR +\bDelta\, ,
\end{align}
with $\prob(\|\bDelta\|_F> C\,\eta_n)\le n^{-A}$, which implies the claimed convergence in distribution of $\bProd$.

We are now left with the task of proving the convergence result (\ref{eq:ConvergenceEvectors}). Notice that, by the decomposition  (\ref{eq:DecompositionEvectors}),
we have
\begin{align}
\tbphi_i &= \bProd_0\tbv_i + \bProd_1\hbv_i + \tbProd \bv_i^{\perp}\, ,
\end{align}
where $\hbv_i = (\sqrt{n} v_{\ell,i})_{\ell\in [k]\setminus S}\in\reals^{k-k_*}$, $\bProd_1\in\reals^{k_*\times (k-k_*)}$ is the submatrix of $\bProd$ with columns
indexed by $[k]\setminus S$, and   $\bv_i^{\perp} =
\sqrt{n}\bV_{\perp}^{\sT} \be_i$. We will also write $\obv_i= \sqrt{n}\bV^{\sT}\be_i$ for the rescaled $i$-th row of $\bV$.
Let $\bG\in\reals^{n\times k_*}$ be a random matrix with i.i.d. entries $G_{ij}\sim\normal(0,1)$. 
Then we can construct $\bV_{\perp} = n^{-1/2}\bP_{\perp}\bG \bSigma^{-1/2}$, where $\bP_{\perp}= \id-\bV\bV^{\sT}$ is the projector orthogonal to $\bV$ and
$\bSigma = \bG^{\sT}\bP_\perp\bG/n$. 
Denoting by $\bg_i$ the $i$-th row of $\bG$, we have 
\begin{align}
\bv^{\perp}_i &= \bSigma^{-1/2}\bg_i -\bSigma^{-1/2} \left(\frac{1}{n}\sum_{j=1}^n\bg_j\obv_j^{\sT}\right)\obv_i\, ,\\
\bSigma & = \frac{1}{n}\sum_{i=1}^n\bg_i\bg_i^{\sT}- \left(\frac{1}{n}\sum_{j=1}^n\bg_j\obv_j^{\sT}\right) \left(\frac{1}{n}\sum_{j=1}^n\bg_j\obv_j^{\sT}\right)^{\sT}\, .
\end{align}
By the law of large numbers, we have  the almost sure limits $\lim_{n\to\infty}\bSigma=\bfone$ and $\lim_{n\to\infty} n^{-1} \sum_{j=1}^n\bg_j\obv_j^{\sT}$  $= \bzero$ 
(which hold conditional on  $\bProd$). Further using  the fact that $\lim\sup_{n\to\infty} \|\bG\|^2_{F}/n<\infty$ and $\|\bV\|_F^2=k$, again almost surely, we obtain
\begin{align}
\lim_{n\to\infty} \frac{1}{n}\sum_{i=1}^n\big\|\bv^{\perp}_i -\bg_i\big\|_2^2 & = 0\, .
\end{align}
Further notice that on $\cG_n(\bLambda)$, $\|\bProd_1\|_{\op}\le C\eta_n\to 0$ and $\|\tbProd-\tbProd_0\|_{\op}\le C\eta_n\to 0$.
Since by Borel-Cantelli $\cG_n(\bLambda)$ holds eventually almost surely,
\begin{align}
\lim_{n\to\infty}\frac{1}{n}\sum_{i=1}^n\big\|\tbphi_i-\big(\bProd_0\tbv_i +\tbProd_0\bg_i\big)\big\|_2^2 =0\, .
\end{align}
Using Lemma \ref{lemma:Simple}, we obtain (almost surely)
\begin{align}
\lim_{n\to\infty}\left|\frac{1}{n}\sum_{i=1}^n\psi(\tbphi_i,\tbv_i)-\frac{1}{n}\sum_{i=1}^n \psi\big(\bProd_0\tbv_i +\tbProd_0\bg_i,\tbv_i\big)\right| =0\, .
\end{align}
The proof is concluded by applying the law of large numbers for triangular arrays to the sum $n^{-1}\sum_{i=1}^n \psi\big(\bProd_0\tbv_i +\tbProd_0\bg_i,\tbv_i\big)$.
\end{proof}

\begin{corollary}\label{coro:ConvergenceEigenvectors}
Under the assumptions of Lemma \ref{lem:ConvergenceEigenvectors}, further assume $\lambda_i \neq\lambda_j$ for all $i\neq j$. 
Assume that the signs of eigenvectors $(\bphi_{\ell})_{\ell\in\hS}$ are chosen so that $\<\bv_{\ell},\bphi_{\ell}\>\ge 0$. Then, almost surely the joint
empirical distribution of $(\sqrt{n}\bphi_{\ell})_{\ell\in\hS}$ and  $(\sqrt{n}\bv_\ell)_{\ell\in S}$ converges in $W_2$ to 
the law of $(\bProd_0\bU+\tbProd_0\bG,\bU)$, where $\bProd_0 = (\id_{k_*}-\bLambda_{k_*}^{-2})^{1/2}$, $\tbProd_0 = \bLambda_{k_*}^{-1}$. Namely, if we let $\tbv_i = (\sqrt{n} v_{\ell,i})_{\ell\in S}\in\reals^{k_*}$ and  $\tbphi_i = (\sqrt{n} \bphi_{\ell,i})_{\ell\in \hS}\in\reals^{k_*}$, then for any $\psi\in \PL(2)$, we have (almost surely)
\begin{align}
\lim_{n\to\infty}\left|\frac{1}{n}\sum_{i=1}^n\psi(\tbphi_i,\tbv_i) - \E\big\{\psi(\bProd_0\bU+\tbProd_0\bG,\bU)\right|=0\, .
\end{align}
\end{corollary}
\begin{proof}
This follows from the observation that in this case the group $\cR(\bLambda)$ consists only of diagonal matrices $\bR$ with $+1/-1$ entries on the
diagonal. Under the further assumption  that $\<\bv_{\ell},\bphi_{\ell}\>\ge 0$ we can restrict attention to the case $\bR =[\id|\bzero]$ (where the two blocks correspond to 
columns with indices in $S$ and $[k]\setminus S$. The claim is proved by using 
Eq.~(\ref{eq:ConvergenceEvectors}) and controlling the effect of deviation $\|\bProd-(\id- \bLambda_{k_*}^{-2})^{1/2}\|_F\le C\eta_n$.
\end{proof}

\section{Proof of Proposition \ref{thm:BayesOpt}}
\label{app:BayesOpt}

\subsection{Preliminaries}

Following \cite{andrea2008estimating}, we define a more general model where, in addition to
obsevations $\bA$, we observe a random subset of the coordinates of $\bx_0$. Namely, we define $\by\in (\reals\cup\{*\})^n$ by
\begin{align}
y_i = \begin{cases}
x_{0,i} & \mbox{ with probability $\eps$,}\\
* &  \mbox{ with probability $1-\eps$,}\\
\end{cases}
\end{align}
independently across $i\in\{1,\dots,n\}$. 
We define $\bM_n \equiv \bx_0\bx_0^{\sT}/n$ and
\begin{align}
\hbx(\bA,\by) \equiv \E\left\{\bx_0\big|\bA,\by\right\}\, ,\;\;\;\;\;\;
\hbM^{\sBayes}_n(\bA,\by) = \E\Big\{\frac{1}{n}\bx_0\bx_0^{\sT}\Big|\bA,\by\Big\}\, .
\end{align}
We also let $\hbM^{\sBayes}_n(\bA) \equiv \E\{\bx_0\bx_0^{\sT}/n|\bA\}$.
The following theorem summarizes a few results proven in  \cite{lelarge2016fundamental}.
\begin{theorem}[\cite{lelarge2016fundamental}]\label{thm:Leo}
There exists a function $\gamma_{\sBayes}:(\lambda,\eps)\to \reals$ and, 
for any $\eps \in [0,1]$, there exists a countable set $D(\eps)$ such that, the following hold:
\begin{enumerate}
\item For every $\lambda \in \reals_{\ge 0}\setminus D(0)$, $\lim_{\eps\to 0}\gamma_{\sBayes}(\lambda,\eps) = \gamma_{\sBayes}(\lambda,0)= \gamma_{\sBayes}(\lambda)$. 
\item For every $\lambda \in \reals_{\ge 0}\setminus D(\eps)$, 
\begin{align}
\lim_{n\to\infty}\E\left\{\|\bM_n-\hbM_n^{\sBayes}(\bA,\by)\|^2_{F}\right\} 
=1- \frac{\gamma^2_{\sBayes}(\lambda,\eps)}{\lambda^4}\, .\label{eq:MMSE}
\end{align}
\item For every $\lambda\in \reals_{\ge 0}$, and every $\oeps>0$,
\begin{align}
\lim_{n\to\infty}\int_0^{\oeps}\E\Big\{\Big\|\hbM_n^{\sBayes}(\bA,\by)-\frac{1}{n}\hbx(\bA,\by)\hbx(\bA,\by)^{\sT}\Big\|_F^2\Big\} \de\eps = 0\, . \label{eq:Replicas-0}
\end{align}
\item Letting $\bx^{(1)},\bx^{(2)}\sim \prob(\,\cdot\,|\bA,\by)$ denote two independent samples from the posterior, for every $\lambda\in \reals_{\ge 0}$, and every $\oeps>0$,
\begin{align}
\lim_{n\to\infty}\int_0^{\oeps}\E\left\{\left(\frac{1}{n}\<\bx^{(1)},\bx^{(2)}\>-\frac{\gamma_{\sBayes}(\lambda,\eps)}{\lambda^2}\right)^2\right\} \de\eps = 0\, . \label{eq:Replicas}
\end{align}
\end{enumerate}
\end{theorem}
Note that Eq.~\eqref{eq:MMSE} implies
\begin{align}
\lim_{n\to\infty}\E\Big\{\|\hbM^{\sBayes}_n(\bA,\by)\|_F^2\Big\} = \lim_{n\to\infty}\E\Big\{\<\hbM^{\sBayes}_n(\bA,\by),\bM_n\>\Big\} =  \frac{\gamma^2_{\sBayes}(\lambda,\eps)}{\lambda^4}\, . \label{eq:LimMBayes}
\end{align}
Further, by Jensen's inequality, $\gamma_{\sBayes}(\lambda,\eps)$ is monotone non-decreasing in $\eps$.

\subsection{Upper bound}

For the proof of the upper bound we will set $\eps=0$ (no side information $\by$ is revealed) and 
we will write $\gamma = \gamma_{\sBayes}(\lambda) = \gamma_{\sBayes}(\lambda,0)$.

We begin by proving that $\gamma/\lambda^2$ is an upper bound on the left-hand side of Eq.~(\ref{eq:BayesOpt}). Indeed assume towards contradiction that
there exists an estimator $\hbx_{n}:\reals^{n\times n}\to \sqrt{n}\sS^{n-1}$ (with $\sS^{n-1}$ the unit sphere in $n$ dimensions) and a sequence $(n(m))_{m\in\naturals}$ such that
\begin{align}
\lim_{m\to \infty} \frac{1}{n^2}\,\E\big\{\big|\<\hbx_{n(m)}(\bA),\bx_0\>\big|^2 \big\} = \frac{ \gamma_+}{\lambda^2}>\frac{\gamma}{\lambda^2}\, .
\end{align}
Given such an estimator, we define
\begin{align}
\hbM_n(\bA) = \frac{b_n}{n}\hbx_{n}(\bA) \hbx_{n}(\bA)^{\sT}\, ,\;\;\;\; b_n = \frac{1}{n^2}\E\{\<\hbx_n(\bA),\bx_0\>^2\}\, .
\end{align}
Then  we get
\begin{align}
\lim_{n\to\infty}\E\big\{\big\|\hbM_n(\bA)-\bM_n\big\|_F^2\big\} &= \lim_{n\to\infty}\left\{1-\frac{2 b_n}{n^2}\E\{\<\hbx_n(\bA),\bx_0\>^2\}+ \frac{b_n^2}{n^2}\E\{\|\hbx_n(\bA)\|_2^4\}\right\}\\
& = 1-   \left\{\lim_{n\to\infty}\frac{1}{n^2}\E\{\<\hbx_n(\bA),\bx_0\>^2\}\right\}^2\\
&= 1- \frac{\gamma^2_+}{\lambda^4} <1- \frac{\gamma^2}{\lambda^4}\, ,
\end{align}
which contradicts the fact (\ref{eq:MMSE}), thus proving our claim. 

\subsection{Lower bound}

We next prove that $\gamma_{\sBayes}(\lambda,0)/\lambda^2$ is a lower bound on the left-hand side of Eq.~(\ref{eq:BayesOpt}), by 
exhibiting an estimator $\hbx_*:\reals^{n\times n}\to \sqrt{n}\sS^{n-1}$ that achieves the claimed accuracy. Without loss of generality,
we will assume $\gamma_{\sBayes}(\lambda,0)>0$ because the claim is trivial otherwise. 
Throughout, we will assume $\lambda\in\reals_{\ge 0}\setminus D(0)$ as per Theorem \ref{thm:Leo} and write $q(\lambda,\eps) = \gamma_{\sBayes}(\lambda,\eps)/\lambda^2$
for brevity, with $q_*= q(\lambda,0)$. Further, we denote by $\E_{\eps}$ expectation with respect to $\eps\sim\Unif([0,\oeps])$,
with $\oeps$ a sufficiently small constant.

Denote by  $\bv_1(\hbM^{\sBayes}_n(\bA))$ the principal eigenvector of $\hbM^{\sBayes}_n(\bA)$, and $\lambda_1(\hbM^{\sBayes}_n(\bA))$ the corresponding eigenvalue.
We set $\hbx_*(\bA) = \sqrt{n}\, \bv_1(\hbM^{\sBayes}_n(\bA))$, whence
\begin{align}
\E\left\{\frac{\<\hbx_*(\bA),\bx_0\>^2}{\|\hbx_*(\bA)\|_2^2 \|\bx_0\|^2_2} \right\} 
&=\E\left\{ \frac{ \<\bv_1(\hbM^{\sBayes}_n(\bA)),\bM_n \bv_1(\hbM^{\sBayes}_n(\bA))\>}
{\|\bx_0\|^2_2/n}\right\}\nonumber\\
&= \E\left\{ \frac{ \<\bv_1(\hbM^{\sBayes}_n(\bA)), \, \hbM^{\sBayes}_n(\bA)\bv_1(\hbM^{\sBayes}_n(\bA))\> }{\|\bx_0\|^2_2/n}
\right\}\nonumber\\
&=\E\left\{ \frac{\lambda_1(\hbM^{\sBayes}_n(\bA))}{\|\bx_0\|^2_2/n} \right\}\, .
%
\label{eq:EigenvalueBound}
\end{align}
%

Let $(\bx^{(\ell)})_{\ell\ge 1}$ be i.i.d. samples from the posterior $\prob(\bx_0\in\, \cdot\,|\bA,\by)$.
Using Theorem \ref{thm:Leo}, see Eq.~\eqref{eq:Replicas}, we have, as $n\to\infty$,
\begin{align}
\E_{\eps}\Big\{ \Big(\E\big\{\tfrac{1}{n}\|\hbx(\bA,\by)\|^2_2\big\}-q(\lambda,\eps)\Big)^2\Big\} & =
\E_{\eps}\Big\{ \Big(\E\big\{ \tfrac{1}{n}\<\bx^{(1)},\bx^{(2)}\>\big\} -q(\lambda,\eps)\Big)^2\Big\} \nonumber\\
& \le \E_{\eps}\E\Big\{ \Big(\frac{1}{n}\<\bx^{(1)},\bx^{(2)}\>-q(\lambda,\eps)\Big)^2\Big\} \to 0\, . \label{eq:VarX}
\end{align}
Therefore, by the triangular inequality with respect to the norm $\|f-g\| = \E_{\eps}\{(f(\eps)-g(\eps))^2\}^{1/2}$,
\begin{align}
\lim_{n\to\infty}\frac{1}{n^2}\E_{\eps}\left\{\Big( \E \|\hbx(\bA,\by)\|_2^2 \Big)^2\right\}  = 
\E_{\eps}\{q(\lambda,\eps)^2\}\, .\label{eq:X2}
\end{align}
Further, using Eqs.~\eqref{eq:Replicas-0} and triangular inequality with respect to the norm 
$\|\bX\| \equiv [\E_{\eps}\E\{\|\bX\|_F^2\}]^{1/2}$, we get (with the shorthands $\hbx = \hbx(\bA,\by)$ and $\hbM = \hbM_n^{\sBayes}(\bA,\by)$)
\begin{align*}
\left|\frac{1}{n} \big(\E_{\eps}\E  \|\hbx\|_2^4\big)^{1/2} - \big( \E_{\eps}\E  \|\hbM\|_F^2\big)^{1/2}\right|
&\le \Big[ \E_{\eps}\E  \|\hbM-\frac{1}{n}\hbx\hbx^{\sT}\|_F^2 \Big]^{1/2}\to 0\, .
\end{align*}
Using Eq.~\eqref{eq:LimMBayes}, this implies
\begin{align}
\lim_{n\to\infty}\frac{1}{n^2}\E_{\eps}\E  \|\hbx(\bA,\by)\|_2^4 =
\lim_{n\to\infty}\E_{\eps}\E \|\hbM_n^{\sBayes}(\bA,\by)\|_F^2=
\E_{\eps} \big\{ q(\lambda,\eps)^2\big\}\, .\label{eq:X4}
\end{align}
Using Eqs.~\eqref{eq:VarX}, \eqref{eq:X2}, \eqref{eq:X4}, we obtain
\begin{align}
\lim_{n\to\infty}  \E_{\eps}\E\left\{ \left(\frac{1}{n}\|\hbx(\bA,\by)\|_2^2-q(\lambda,\eps)\right)^2\right\} = 0\, .\label{eq:NormConcentration}
\end{align}

Next note that, by Eq.~\eqref{eq:Replicas-0}
\begin{align*}
&\left|\frac{1}{n}\E_{\eps}\E\big\{\<\hbx(\bA,\by),\hbM_n^{\sBayes}(\bA)\hbx(\bA,\by)\>\big\}-\E_{\eps}
\E\big\{\<\hbM_n^{\sBayes}(\bA),\hbM^{\sBayes}_n(\bA,\by)\>\big\}\right|^2 \\
&\le\E\big\{  \|\hbM_n^{\sBayes}(\bA)\|_F^2\big\} \, \E_{\eps}\E\left\{\left\|\hbM_n^{\sBayes}(\bA)-\frac{1}{n}\hbx(\bA,\by) \hbx(\bA,\by)^{\sT}\right\|_F^2\right\} \to 0\, .
\end{align*}
Therefore
\begin{align}
\lim_{n\to\infty}\frac{1}{n}\E_{\eps}\E\{\<\hbx(\bA,\by),\hbM_n^{\sBayes}(\bA)\hbx(\bA,\by)\>\} &= \lim_{n\to\infty}
\E_{\eps}\E\{\<\hbM_n^{\sBayes}(\bA), \, \hbM_n^{\sBayes}(\bA,\by)\>\}\\ 
&= \lim_{n\to\infty}\E_{\eps}\E\{\|\hbM_n^{\sBayes}(\bA)\|_F^2\} =  q_*^2\, . \label{eq:KeyXMX}
\end{align}

We proceed as follows from Eq.~\eqref{eq:EigenvalueBound} for a fixed $\delta >0$:
\begin{align}
& \E \left\{ \frac{\lambda_1(\hbM^{\sBayes}_n(\bA))}{\|\bx_0 \|_2/n} \right\}\ge \E_{\eps}\E\left\{\frac{\<\hbx(\bA,\by),\hbM^{\sBayes}_n(\bA) \hbx(\bA,\by)\>}{\|\hbx(\bA,\by)\|_2^2 \ \  \|\bx_0 \|_2/n}\right\} \nonumber \\
&\ge \frac{1}{n(q_*+\delta)(1+\delta)} \E_{\eps}\E\left\{\<\hbx(\bA,\by),\hbM^{\sBayes}_n(\bA) \hbx(\bA,\by)\>\bfone_{\frac{1}{n}\|\hbx(\bA,\by)\|_2^2\le q_*+\delta} \   \bfone_{\frac{1}{n}\|\bx_0 \|_2^2 \leq 1+ \delta }  \right\} \nonumber \\
&\ge \frac{1}{n(q_*+\delta)(1+\delta)} \Bigg( \E_{\eps}\E\left\{\<\hbx(\bA,\by),\hbM^{\sBayes}_n(\bA) \hbx(\bA,\by)\>\right\} \nonumber \\
& \ - 
\Big[ \E_{\eps}\E\{\<\hbx(\bA,\by),\hbM^{\sBayes}_n(\bA) \hbx(\bA,\by)\>^2\} \Big]^{\frac{1}{2}} \Big[\E_{\eps}\left\{\prob\left(\tfrac{1}{n}\|\hbx(\bA,\by)\|_2^2 > q_*+\delta\right)\right\} + 
\prob(\tfrac{1}{n} \|\bx_0\|^2_2  \geq 1 + \delta  )  \Big]^{\frac{1}{2}}\, \Bigg).
\label{eq:EigLB}
\end{align}
Next note that
\begin{align}
 \E_{\eps}\E\{\<\hbx(\bA,\by),\hbM^{\sBayes}_n(\bA) \hbx(\bA,\by)\>^2\}&\le 
 \Big( \E_{\eps}\E \|\hbx(\bA,\by) \|_2^4\Big)^{1/2} \Big(  \E\|\hbM^{\sBayes}_n(\bA) \|_F^4 \Big)^{1/2} 
 \nonumber \\
&\le \Big(  \frac{1}{n^2} \E_{\eps}\E\{  \|\bx_0\|_2^4\} \Big) \le C\, .
\end{align}
Further for any $\delta>0$, we can choose $\oeps$ small enough so that 
$\E_{\eps}\left\{\prob\left(\frac{1}{n}\|\hbx(\bA,\by)\|_2^2> q_*+\delta\right)\right\}\to 0$. Indeed, this follows from
Eq.~\eqref{eq:NormConcentration} and Markov's inequality, together with the fact that $q(\lambda,\eps)\to q_*$ as $\eps\to 0$. We also have $\prob( \frac{1}{n}\|\bx_0\|^2_2  \geq 1 + \delta ) \to 0$ from the law of large numbers.

 Using Eqs.~\eqref{eq:EigLB} and \eqref{eq:KeyXMX} in Eq.~\eqref{eq:EigenvalueBound}, we conclude
\begin{align*}
\lim\inf_{n\to\infty}\E\left\{\frac{\<\hbx_*(\bA),\bx_0\>^2}{\|\hbx_*(\bA)\|_2^2\|\bx_0\|^2_2} \right\}  &= \lim\inf_{n\to\infty} \frac{\E \{\lambda_1(\hbM^{\sBayes}_n(\bA))\}}{1+\delta} \\
 &\ge \frac{1}{(q_*+\delta)(1+ \delta)}\lim\inf_{n\to\infty}\frac{1}{n}\E_{\eps}\E\left\{\<\hbx(\bA,\by),\hbM^{\sBayes}_n(\bA) \hbx(\bA,\by)\>\right\}\\
& \ge \frac{q_*^2}{(q_*+\delta)(1+ \delta)}\, .
\end{align*}
The desired lower bound follows since $\delta$ can be taken arbitrary small.

\section{Proofs for Section \ref{sec:SparseSpike}: Sparse spike} \label{app:sparse_spike}

\subsection{Reduction to three-points priors}
\label{sec:Reduction}

In this appendix we prove that the map $S(\gamma;\theta)$ defined in Eq.~\eqref{eq:Sdef} is indeed a lower bound on the  state evolution map.
\begin{lemma}\label{lemma:Reduction3Points}
Let $S_*, S$ be defined as in Eqs.~\eqref{eq:Sstar_def}--\eqref{eq:Sdef}. Then
\begin{align}
\inf_{\nu\in\cF_{\eps}}S_* (\gamma,\theta;\nu) = S(\gamma;\theta)\, ,\label{eq:Reduction}
\end{align}
where $\cF_{\eps} = \{\nu_{X_0}: \; \nu_{X_0}(\{0\})\ge 1-\eps, \,  \int x^2\nu_{X_0}(\de x) = 1\}$. 
\end{lemma}

By rescaling the distribution $\nu$, it is sufficient to prove this lemma for $\gamma=1$, and replacing $\cF_{\eps}$ by
$\cF_{\eps,\gamma} = \{\nu: \; \nu(\{0\})\ge 1-\eps, \,  \int x^2\nu(\de x) = \gamma\}$.  With $G \sim \normal(0,1)$, we define the functions
\begin{align}
f_1(x) &\equiv x\, \E\big\{\eta(x+G;\theta)\big\}\, ,\\
f_2(x) & \equiv  \E\big\{\eta(x+G;\theta)^2\big\}\, ,\\
F_{\balpha}(x) & \equiv \alpha_0x^2+\alpha_1 f_1(x) +\alpha_2 f_2(x)\, .
\end{align}
(We omit  the dependence on  $\gamma$, $\theta$, since they are fixed throughout the proof.)
Notice that $f_1,f_2:\reals\to\reals$ are even (namely $f_i(-x) =
f_i(x)$) and analytic on $\reals$. Further $f_1(x)>0$ for all $x\neq 0$, with $f_1(0) = 0$ and $\inf_{x\in\reals} f_2(x)>0$. Moreover
\begin{align}
S_*(1,\theta;\nu) = \frac{\big(\int f_1(x) \nu(\de x)\big)^2}{\int f_2(x) \nu(\de x)}\, .
\end{align} 
Because the $f_i$'s are even, it is sufficient to prove the lemma by considering $\nu$ with support on $\reals_{\ge 0}$, i.e. to consider the class
$\cF^+_{\eps,\gamma}= \{\nu: \; \supp(\nu)\subseteq\reals_{\ge 0}\, ,\nu(\{0\})\ge 1-\eps, \,  \int x^2\nu(\de x) = \gamma\}$.
We also define $\cF^+_{\eps}= \{\nu: \; \supp(\nu)\subseteq\reals_{\ge 0}\, ,\nu(\{0\})\ge 1-\eps  \}$ (dropping the second moment constraint).

The next two lemmas establish analytic facts that will be crucial in the proof of Lemma \ref{lemma:Reduction3Points}.
\begin{lemma}\label{lemma:Tanh}
Let $\theta\in \reals_{>0}$, $a\in\reals$ be given, and consider the equation
\begin{align}
\frac{\theta x}{\tanh \theta x} = x^2+a\, .
\end{align}
For $a\le 0$, this equation has exactly one solution for $x\in (0,\infty)$. For $a>0$ it has at most two solutions for $x\in (0,\infty)$. 
\end{lemma}
\begin{proof}
For $a\le 0$, rewrite this equation as
\begin{align}
\frac{1}{\theta}\tanh \theta x = \frac{x}{x^2+a}\, . \label{eq:Tanh}
\end{align}
The left-hand side is strictly  increasing and positive on $(0,\infty)$. The right-hand side $h(x) = x/(x^2+a)$  is strictly negative for $x\in (0,\sqrt{-a})$,
and decreasing and stricly positive on $(\sqrt{-a},\infty)$. Further, $h(x)\uparrow+\infty$ as $x\downarrow \sqrt{-a}$ and $h(x)\downarrow 0$ as 
$x\uparrow +\infty$. Hence the equation has exactly one solution $x_*$ on $(0,\infty)$ for $a \leq 0$, with $x_*\in (\sqrt{-a},\infty)$.

Next consider the case $a>0$. Define $u(x) = x/\tanh(x)$. It is easy to compute
\begin{align}
u'(x) & = \frac{1}{\tanh x}-\frac{x}{\sinh^2 x}\, ,\\
u''(x) & = \frac{2}{\sinh^2 x}\left(  \frac{x}{\tanh(x)} -1 \right)\, ,\\
u'''(x) & = - \frac{2}{\sinh^2 x}\left( \frac{x(2 + \cosh^2 x) - 3 \sinh x \cosh x}{\sinh^2 x}
 \right)\,.
\end{align}
In particular, we have $u''(x)>0$ and $u'''(x)<0$ for $x\in (0,\infty)$. Solutions of Eq.~\eqref{eq:Tanh} are zeros of $g(x) \equiv u(\theta x)-x^2-a$.
The above calculation yields $g'''(x)<0$ and
\begin{align}
g'(x) & = \frac{\theta }{\tanh \theta x}-\frac{\theta^2x}{\sinh^2 \theta x}+\theta^2 x-2x\, ,\\
g''(x)  & = \frac{2\theta^2}{\sinh^2 \theta x}\left( \frac{\theta x}{\tanh \theta x} - 1 \right) - 2\, .
\end{align}
In particular, we have $g''(0+) =(2\theta^2/3)-2$ and $g''(+\infty) = -2$. Hence $g$ is convex for $x\in(0,x_0)$, and concave for $x\in (x_0,\infty)$, 
where $x_0= 0$ for $\theta<\sqrt{3}$. Further $g'(0+) = 0$, and $g'(x)\downarrow -\infty$ for 
$x\uparrow +\infty$. Therefore $g$ is increasing on $(0,x_0]$ and has a unique local maximum $x_*$ on $(x_0,\infty)$.
Hence $g$ is strictly increasing on $(0,x_*)$ and strictly decreasing on $(x_*,\infty)$ 
It follows that $g(x) =0$ can have at most two solutions.
\end{proof}

\begin{lemma}\label{lemma:Falpha}
For any nonzero vector $\balpha\in \reals^3$, the function $x\mapsto F_{\balpha}(x)$ has at most two local maxima on $(0,\infty)$.
\end{lemma}
\begin{proof}
We compute first two derivatives of $F_{\balpha}$ to get
\begin{align}
F'_{\balpha}(x) & = 2\alpha_0 x +(\alpha_1+2\alpha_2)\E \eta(x+G;\theta) + \alpha_1 x\, \big[\Phi(x-\theta)+\Phi(-x-\theta)\big]\, ,\\
F''_{\balpha}(x) & = 2\alpha_0 + H(x)\,,\\
H(x) & \equiv b_1\, \big[\Phi(x-\theta)+\Phi(-x-\theta)\big] + b_2x\,  \big[\phi(x-\theta)-\phi(x+\theta)\big]\, ,
\end{align}
where we defined $b_1= 2(\alpha_1+\alpha_2)$, $b_2=\alpha_1$. We claim that $F''_{\balpha}(x)=0$ for at most three values of $x\in (0,\infty)$.
Hence there are at most two disjoint intervals $I_1=(a_1,b_1)$, $I_2=(a_2,b_2)\subseteq \reals$ (with, potentially, $b_2=\infty$) such that $F''_{\balpha}(x)<0$ for $x\in I_1\cup I_2$,
(because $F'''_{\balpha}$ must vanish at the boundary of these intervals).
Since $F_{\balpha}$ is concave in these intervals and convex outside, it can have at most one local maximum in each of the intervals. This proves the lemma.

In order to prove the claim that $F''_{\balpha}(x)=0$ for at most three values of $x\in (0,\infty)$), we compute the derivative 
\begin{align}
F'''_{\balpha}(x) = H'(x) = (b_1+b_2-b_2x^2) \big[\phi(x-\theta)-\phi(x+\theta)\big]+b_2\theta x\big[\phi(x-\theta)+\phi(x+\theta)\big]\,,
\end{align}
and show that $H'(x) =0$ can have at most two solutions in $(0,\infty)$. From this it follows that $H(x) = -2\alpha_0$ can have at most three solutions  in $(0,\infty)$
(because otherwise it would have more than two stationary points by the intermediate value theorem).

If $b_2=0$, then necessarily $b_1\neq 0$, and the claim that that $H'(x) =0$ has at most two solutions  is trivial. We can therefore assume $b_2\neq 0$.
Re-organizing the terms, we get $H'(x)=0$ (for $x\in (0,\infty)$) if and only if
\begin{align}
-\frac{b_1+b_2}{b_2}+x^2 = \frac{\theta x}{\tanh\theta x}\, .
\end{align}
By Lemma \ref{lemma:Tanh}, this equation can have at most two solutions in $(0,\infty)$,  which completes the proof.
\end{proof}

We are now in position to prove Lemma \ref{lemma:Reduction3Points}.
\begin{proof}[Proof of Lemma \ref{lemma:Reduction3Points}.]
Obviously  the right-hand side of Eq.~\eqref{eq:Reduction} is no smaller than the left-hand side. We will prove that the infimum on the left-hand side is achieved at 
$\nu=\pi_{p,a_1,a_2}$ for a certain three points prior, hence establishing the lemma.

Denote by $\cuP(\reals)$ the space of probability distributions on $\reals$ endowed with the $W_2$ metric \eqref{eq:WassersteinDef}. Let  $H:\cuP(\reals)\to \reals^2_{\ge 0}$, and
$Q \subseteq \reals_{\ge 0}$ be defined by
\begin{align}
H(\nu) & \equiv \Big(\int f_1(x)\, \nu(\de x), \; \int f_2(x)\, \nu(\de x)\Big)\, ,\\ 
\cQ_{\eps,\gamma}& \equiv \big\{ H(\nu):\; \nu\in \cF_{\eps,\gamma}\big\}\, .
\end{align}
Note that $H$ is continuous in the $W_2$ metric (because $f_1$, $f_2$ are continuous, with $|f_1(x)|, |f_2(x)|\le C(1+x^2)$). 
Further $\cF_{\eps,\gamma}$ is sequentially compact in the same metric, and therefore $\cQ_{\eps,\gamma}$ is compact. Since $\inf_{x\in\reals} f_2(x)>0$
and $\int f_1(x)\, \nu(\de x)>0$ unless $\nu = \delta_0$, it follows that $\cQ_{\eps,\gamma}\subseteq \reals_{>0}^2$ (i.e., is bounded away from the coordinate axes).
Finally, since $H$ is linear and $\cF_{\eps,\gamma}$ is convex, it follows that $\cQ_{\eps,\gamma}$ is convex as well.

We have 
\begin{align}
\inf_{\nu\cF_{\eps}} S_*(\gamma,\theta;\nu) = \inf_{(z_1,z_2)\in\cQ_{\eps,\gamma}} \frac{z^2_1}{z_2}\equiv S_{\min}\, .
\end{align}
Notice that the infimum on the right-hand side is achieved at a point 
$(z_{1*},z_{2*}) \in \cQ_{\eps,\gamma}$ because $\cQ_{\eps,\gamma}$ is compact and $(z_1,z_2)\mapsto z_1^2/z_2$ is continuous on $\reals^2_{>0}$. Furthermore,  $(z_{1*},z_{2*})$ must be on the boundary of $\cQ_{\eps,\gamma}$. Indeed, if this wasn't the case
$(z_{1*}-\delta,z_{2*})$ would be feasible for $\delta$ small enough, and achieve a smaller ratio  $z_1^2/z_2$.

To complete the proof, we will show that, for any point $\bz = (z_1,z_2)$ on
the boundary $\dcQ$ of $\cQ = \cQ_{\eps,\gamma}$, there exists  $\pi_{p,a_1,a_2}\in\cF_{\eps,\gamma}$ such that
$H(\pi_{p,a_1,a_2})= \bz$, whence $S_*(1,\theta;\pi_{p,a_1,a_2}) = S_{\min}$. Since $\cQ$ is convex, any point $\bz\in \dcQ$ is a maximizer of a
linear function $\alpha_1z_1+\alpha_2 z_2$ subject to $\bz\in \cQ$,
for some nonzero vector $(\alpha_1,\alpha_2)\in\reals^2$. It is therefore sufficient to show that for any $(\alpha_1,\alpha_2)\neq (0,0)$ 
the maximizer is unique and takes the form $\bz = H(\pi_{p,a_1,a_2})$. This optimization problem can be rewritten as
\begin{equation}
\begin{aligned}
\mbox{maximize}&\;\; \alpha_1\int f_1(x)\,  \nu(\de x) +  \alpha_2\int f_2(x)\,  \nu(\de x) \, ,\\
\mbox{subject to}&\;\; \nu \in\cF^+_{\eps}\, ,\\
&\;\; \int x^2 \nu(\de x) = \gamma\, .
\end{aligned}
\end{equation}
The claim follows if this problem has a unique maximizer at a three-points distribution $\gamma = \pi_{p,a_1,a_2}$.
By strong duality, there exist a Lagrange parameter $\alpha_0\in\reals$, such that all maximizers of the last optimization problem are
also maximizers of 
\begin{equation}\label{eq:Lagrangian3Points}
\begin{aligned}
\mbox{maximize}&\;\; \int F_{\balpha}(x)\,  \nu(\de x) \, ,\\
\mbox{subject to}&\;\; \nu \in\cF^+_{\eps}\, ,
\end{aligned}
\end{equation}
where we recall that $F_{\balpha}(x) = \alpha_0x^2+\alpha_1 f_1(x)+\alpha_2f_2(x)$. 
Note that the constraint $\nu \in\cF^+_{\eps}$ is equivalent to $\nu = (1-\eps)\delta_0+\eps\nu^+$ with $\nu^+\in\cuP((0,\infty))$ (a probability
distribution with support in $(0,\infty)$. Therefore,
$\nu$ is a solution of problem \eqref{eq:Lagrangian3Points} if and only if $\nu^+$ is supported on the global maxima of $F_{\balpha}$. However,
by Lemma \ref{lemma:Falpha}, the set of global maxima contains at most two points, and therefore $\nu^+$is supported on at most two points, which proves our claim.
\end{proof}

\subsection{Proof of Proposition \ref{prop:sparse_spike}} 

For $t \geq 0$, let $\gamma_t \equiv \mu_t^2/\sigma_t^2$. We will first show the inequality 
in \eqref{eq:sparse_spike_lb}, which is equivalent to showing $\gamma_{t+1} \geq \ugamma_{t+1}$. From the definitions, we have 
$\gamma_{0} = \ugamma_{0} = (\lambda^2-1)$.  Assume towards induction that 
$\gamma_{s} \geq \ugamma_{s}$ for  $0 \leq s \leq t$.   We observe that $\gamma_{t+1}$ can be computed from $\gamma_t$ as
\beq
\gamma_{t+1} = \lambda^2 S_* (\gamma_t,\theta_t;\nu_{X_0}),
\eeq
where the function $S_*$ is defined in \myeqref{eq:Sstar_def}. Indeed, since the soft-thresholding function satisfies 
$\eta(x; \theta \sigma) = \sigma \eta(x/\sigma \, ; \theta)$ for any $\theta, \sigma >0$, we have
\begin{align}
\gamma_{t+1} = \lambda^2 \,  \frac{[\E\{ X_0 \,  \eta(\mu_t X_0+\sigma_t G; \, \theta_t \sigma_t)\} ]^2}{\E \{ \eta(\mu_t X_0 + \sigma_t G; \,  \theta_t \sigma_t )^2 \}} = 
\lambda^2 \frac{ [ \E\{ X_0 \, \sigma_t \eta(\gamma_t X_0+ G; \, \theta_t )\} ]^2}{\E \{ \sigma_t^2 \eta(\gamma_t X_0+ G; \, \theta_t  )^2\}} 
= \lambda^2  S_*(\gamma_t,\theta_t;\nu_{X_0}).
\end{align}

Next, we note that $S_*(\gamma,\theta;\nu_{X_0})$ is non-decreasing in $\gamma$. To see this, we use the definition in \eqref{eq:Sstar_def} to compute the derivative:
\begin{align}
\frac{\partial  S_* (\gamma,\theta;\nu_{X_0}) }{\partial \gamma} = 
\frac{ \E\{\eta(\sqrt{\gamma}X_0+G;\theta)^2\}  \E\{ X_0^2 \, \ind( | \sqrt{\gamma} X_0 + G | > \theta ) \} \, - \, [\E\{X_0\eta(\sqrt{\gamma}X_0+G;\theta)\}]^2 }{\sqrt{\gamma} \, [\E\{\eta(\sqrt{\gamma}X_0+G;\theta)^2\} ]^2} \geq 0,
\end{align}
where the inequality is due to the Cauchy-Schwarz (after noting that $\eta(\sqrt{\gamma}X_0+G;\theta) = 
\ind(| \sqrt{\gamma}X_0+G | > \theta)\, \eta(\sqrt{\gamma}X_0+G;\theta)$). Therefore, using the induction hypothesis we have
\begin{align}
\gamma_{t+1} = \lambda^2 S_* (\gamma_t,\theta_t;\nu_{X_0}) &  \geq \lambda^2 S_* (\ugamma_t,\theta_t;\nu_{X_0}) 
 \geq \lambda^2 \inf_{\pi_{X_0} \in \cF_{\eps}}  S_* (\ugamma_t,\theta_t;\pi_{X_0}),
 \label{eq:gam_t1_lb}
\end{align}
where $\cF_{\eps} = \{\pi_{X_0}: \; \pi_{X_0}(\{0\})\ge 1-\eps, \,
\int x^2\pi_{X_0}(\de x) = 1\}$. By Lemma \ref{lemma:Reduction3Points}, the infimum is achieved on a three-points prior, whence:
\beq
\inf_{\pi_{X_0} \in \cF_{\eps}}  S_* (\ugamma_t,\theta_t;\pi_{X_0})  = S(\ugamma_t, \theta_t) \equiv 
\inf\Big\{S_* (\gamma,\theta;\pi_{p,a_1,a_2}) :\;\; pa_1^2+(1-p)a_2^2 =1, p\in[0,1]\Big\}\, .
\label{eq:three_pt_equiv}
\eeq
Recalling from Eq. \eqref{eq:ugam_rec} that  $\ugamma_{t+1}= \lambda^2 S(\ugamma_t, \theta_t)$, Eqs. \eqref{eq:gam_t1_lb} and \eqref{eq:three_pt_equiv} imply
\beq
\gamma_{t+1} \geq \ugamma_{t+1},
\eeq
as required.  

Next we prove the equality in Eq. \eqref{eq:sparse_spike_lb}. For this, we define the AMP iteration 
\begin{align}
{\bx'}^{\,t+1} &= \bA\, \hbx'^{\, t}-\sb_t \hbx'^{\, t-1}\, ,\;\;\;\; \hbx'^{\, t}  = \eta(\bx'^{\, t}; \, \theta_t \sigma_t)\, ,\label{eq:SparseAlgMod}\\
&\;\;\;\;\;\sb'_t = \frac{1}{n}\|\hbx'^{\, t}\|_0\, , \nonumber
\end{align}
initialized with $\bx'^{\, 0} = \sqrt{n} \bphi_1$. The difference between  $\hbx'^{\, t}$ and $\hbx^t$ is that the former is produced using the deterministic threshold $\theta_t \sigma_t$ (whose computation would require knowledge of the distribution  $\nu_{X_0}$), and the latter using the threshold $ \theta_t \hat{\sigma}_t$ which is computed from data.  The  result of Theorem \ref{thm:Rank1} can be directly applied to the iterates $\{ \bx'^{\, t} \}_{t\geq 0}$, but not to to the iterates $\{ \bx^{\, t} \}_{t\geq 0}$ (as the data-derived threshold makes the soft-thresholding denoiser non-separable).
We will show below that for $t \geq 0$, almost surely, 
\beq
\lim_{n \to \infty} \frac{1}{n} \| \bx^t - \bx'^{\, t} \|_2^2=0 \, .
\label{eq:xxpr_eq}
\eeq
Equation \eqref{eq:xxpr_eq}  implies that, almost surely,
\beq 
\lim_{n \to \infty}  \frac{1}{n} \| \hbx^t -\hbx'^{\, t} \|^2_2 =0 \, .
\label{eq:xxpr_eq1}
\eeq
Indeed,
\begin{align}
 \frac{1}{n} \| \hbx^t -\hbx'^{\, t} \|^2_2 & =  \frac{1}{n} \| \eta(\bx^t; \theta_t \hat{\sigma}_t) -  \eta(\bx'^{\, t}; \theta_t {\sigma}_t ) \|^2_2 
\nonumber \\
& \leq \frac{2}{n}   \| \eta(\bx^t; \theta_t \hat{\sigma}_t)  -  \eta(\bx^t; \theta_t \sigma_t) \|^2 + \frac{2}{n}    \|  \eta(\bx^t; \theta_t \sigma_t) -  \eta(\bx'^{\, t}; \theta_t {\sigma}_t ) \|^2_2 \nonumber  \\
& \leq 2 \, \theta_t^2 ( \hat{\sigma}_t - \sigma_t )^2 + \frac{2}{n} \| \bx^t -\bx'^{\, t} \|^2_2  \, \to \,  0 \quad \text{a.s.}
\label{eq:xxpr_eq2}
\end{align}
where the last inequality holds because $\eta(x; \theta)$ is Lipschitz in each argument, with $\abs{\partial_x \eta(x; \theta)} = \abs{\partial_\theta \eta(x; \theta)} = \ind( | x | > \theta)$.  

Eqs. \eqref{eq:xxpr_eq} and \eqref{eq:xxpr_eq1} imply that, almost surely
\begin{align}
\lim_{n\to\infty} \frac{|\<\hbx^t,\bx_0\>|}{\|\hbx^t\|_2\|\bx_0\|_2} & = \lim_{n\to\infty} \frac{|\<\hbx'^{\,t},\bx_0\>|}{\|\hbx'^{\, t} \|_2\|\bx_0\|_2}  \\
& = \frac{\mu_{t+1}}{\lambda \sigma_{t+1}}, \label{eq:SE_result_overlap}
\end{align}
as required. Here Eq. \eqref{eq:SE_result_overlap} is obtained by applying Theorem \ref{thm:Rank1} with the following choices for the test function $\psi: \reals \times \reals  \to \reals$. First take $\psi(u,v) = u \,\eta(v; \, \theta_t \sigma_t)$ to obtain
\beq 
\lim_{n \to \infty} \frac{1}{n} \< \bx_0, \hbx'^{\, t} \> = \E\{ X_0 \eta(\mu_t X_0  + \sigma_t G) \} = \frac{\mu_{t+1}}{\lambda} \quad \text{ a.s.} 
\eeq
Next take $\psi(u,v) = \eta(v; \, \theta_t \sigma_t)^2$ to obtain
\beq
\lim_{n \to \infty} \frac{1}{n} \| \hbx'^{\, t} \|_2^2 = \E\{ \eta(\mu_t X_0  + \sigma_t G)^2 \} = \sigma_{t+1}^2 \quad \text{ a.s.}
\label{eq:sig_t_conv}
\eeq
It is easy to check that both these choices for $\psi$ satisfy the condition required by Theorem \ref{thm:Rank1}.

Finally, it remains to prove Eq. \eqref{eq:xxpr_eq}.  For $t=0$, we have $\bx^0= \bx'^{\, 0} =\sqrt{n} \bphi_1$.  Towards induction, assume Eq. \eqref{eq:xxpr_eq} holds for $0 \leq s \leq t$. From Eqs. \eqref{eq:SparseAlg} and  \eqref{eq:SparseAlgMod}, we have
\begin{align}
\|\bx^{t+1} - \bx'^{\, t+1}\|_2 & \le \|\bA\|_{\op} \, \| \hbx^t -\hbx'^{\, t} \|_2+ |\sb_t-\sb'_t|\, \| \hbx'^{\, t-1} \|_2+
 |\sb'_t|\, \big\| \hbx^{t-1} -\hbx'^{\, t-1} \big\|_2\\
& \equiv \sqrt{n}\big(R_1(t;n)+R_2(t;n)+R_3(t;n)\big)\,  .
\end{align}
Consider the first term. Since  $\lim\sup_{n\to\infty}\|\bA\|_{\op} \le \lambda+\lim_{n\to\infty}\|\bW\|_{\op} =\lambda+2$ almost surely, from the induction hypothesis and Eq. \eqref{eq:xxpr_eq1}  it follows that $R_1(t;n)^2 \to 0$ almost surely. Similarly, by the induction hypothesis (and noting that $|\sb_t'| <1$), we also have $R_3(t;n)^2 \to 0$. 
For $R_2(t;n)$, the induction hypothesis  Eq. \eqref{eq:xxpr_eq} and Theorem \ref{thm:Rank1} together imply that 
the following holds for any test function $\psi: \reals \times \reals \to \reals$ satisfying the conditions of the theorem. For $0 \leq s \leq t$, 
\beq 
\lim_{n \to \infty} \frac{1}{n} \sum_{i=1}^n \psi (x_{0,i},x^{s}_i) = \lim_{n \to \infty} \frac{1}{n} \sum_{i=1}^n \psi (x_{0,i},x'^{\, s}_i) = \E \left\{ \psi( X_0, \mu_s X_0  +\sigma_s G) \right\} \  \quad \text{a.s.}
\label{eq:SE_PLfns}
\eeq
As in Eq. \eqref{eq:sig_t_conv}, we have $ \frac{1}{n}\| \hbx'^{\, t-1} \|_2^2  \to \sigma_{t-1}^2$.  Eq. \eqref{eq:SE_PLfns} implies that the empirical distributions of $\bx^{t}$ and $\bx'^{\, t}$ both converge weakly to the distribution of $(\mu_t X_0 + \sigma_t G)$. Furthermore since $\eta(x; \theta)$ is Lipschitz, denoting by $\partial \eta$ the derivative with respect to the first argument, \cite[Lemma 5]{BM-MPCS-2011} implies that 
\begin{align}
 \sb'_t =  \frac{1}{n} \sum_{i=1}^n \partial \eta(x'^{\, t}_i ; \theta_t {\sigma_t}) \  \stackrel{\text{a.s.}}{\to} \  \E\{ \partial \eta(\mu_t X_0 + \sigma_t G ;  \, \theta_t {\sigma_t}) \} = \P( |\mu_t X_0 + \sigma_t G | > \theta_t {\sigma_t} ).
\end{align}
Similarly, $\sb_t$ also converges to $\P( |\mu_t X_0 + \sigma_t G | > \theta_t {\sigma_t} )$. This shows that $R_2(t;n) \to 0$, and completes the proof of the proposition.

\section{Proof of Theorem \ref{thm:Special}}
\label{app:ProofRank1}

We begin by proving the following lemma, which implies Remark \ref{rem:LipExpectation}. (This stronger version will be used in Appendix 
\ref{app:CoroConfInt}).
\begin{lemma}\label{lemma:LipCE}
The function $F:\reals\times \reals_{>0}\to\reals$ of Eq.~(\ref{eq:Fdef}) is $C^{\infty}(\reals\times \reals_{>0})$.
Further assume either of the following conditions: $(i)$ $\supp(\nu_{X_0})\in [-M,M]$ for some constant $M$; $(ii)$ $\nu_{X_0}$ has log-concave density.  
Then, for any $\eps>0$ there  exist $C(\eps)<\infty$ such that, for any $y\in\reals$, $\gamma\in [\eps,\infty)$,
we have
\begin{align}
\big|\partial_y F(y;\gamma)\big|\le C(\eps),\;\;\;\; \big|\partial_{\gamma}F(y;\gamma)\big|\le C(\eps)\,\big(1+|y|\big)\, .
\end{align}
\end{lemma}
\begin{proof}
Note that (throughout this proof, we write $\mu = \nu_{X_0}$ for the law of $X_0$)
\begin{align}
F(y;\gamma) =  \frac{\int x\, e^{yx-\frac{1}{2}\gamma x^2} \mu(\de x)}{\int  e^{yx-\frac{1}{2}\gamma x^2} \mu(\de x)} \, . 
\end{align}
Hence $F\in C^{\infty}(\reals\times \reals_{>0})$ by an application of dominated convergence (alternatively 
notice that $F$ an be obtained by differentiating a log-moment generating function).

In order to bound the derivatives, we write $\mu_{y,\gamma}$ for the probability measure on $\reals$ with Radon-Nikodym derivative
\begin{align}
\frac{\de \mu_{y,\gamma}}{\de\mu}(x) =  \frac{ e^{yx-\frac{1}{2}\gamma x^2}}{\int  e^{yx-\frac{1}{2}\gamma x^2} \mu(\de x)}\, , 
\end{align}
and we write $\Ev_{y,\gamma}$ and $\Var_{y,\gamma}$ for expectation and variance with respect to this measure.
We then have
\begin{align}
\partial_y F(y;\gamma) & = \Var_{y,\gamma}(X)\, ,\\
\partial_{\gamma} F(y;\gamma) & = -\frac{1}{2}\Ev_{y,\gamma}(X^3)- \frac{1}{2}\Ev_{y,\gamma}(X^2)\Ev_{y,\gamma}(X)\, ,\\
\big|\partial_{\gamma} F(y;\gamma) \big|& \le \frac{1}{2}\sqrt{\Var_{y,\gamma}(X) \Var_{y,\gamma}(X^2)}\, ,\label{eq:PartialGamma}
\end{align}
where the last inequality follows by Cauchy-Schwarz.
Under assumption $(i)$, we have $|\partial_y F(y;\gamma)|\le M^2$, $|\partial_\gamma F(y;\gamma)|\le M^3/2$. 

Under assumption $(ii)$, note that $\mu_{y,\gamma}$ is $\eps$-strongly log-concave (i.e.   $\mu_{y,\gamma}(\de x) = \exp\{-h_{y,\gamma}(x)\}\, \de x$,
with $h_{y,\gamma}(x)$ $\eps$-strongly convex).  As a consequence, it satisfies a log-Sobolev inequality with constant $1/\eps$ \cite[Theorem 5.2]{Ledoux}, whence
$\mu_{y,\gamma}(|X-\Ev_{y,\gamma}(X)|\ge t) \le 2\, e^{-\eps t^2/2}$, and therefore  $\Var_{y,\gamma}(X)\le C_0/\eps$, for a numerical constant $C_0$. 
The same inequality implies 
\begin{align}
\mu_{y,\gamma}(|X^2- \Ev_{y,\gamma}(X)^2|\ge t) \le C_1\, \exp\left\{-\frac{\eps}{C_1}\Big(\frac{t^2}{\Ev_{y,\gamma}(X)^2}\wedge t\Big)\right\} \, .
\end{align}
Using $|\Ev_{y,\gamma}(X) |= |F(y;\gamma)|\le |F(0;\gamma)|+\|\partial_yF\|_{\infty}|y| \le C_0'(1+(|y|/\eps))$ (which follows from the above bound on
$\partial_yF(y;\gamma) =  \Var_{y,\gamma}(X)$),  immediately implies,
for $\eps\le 1$,
\begin{align}
\Var_{y,\gamma}(X^2) &\le \frac{C_2}{\eps} \Big(\frac{1}{\eps}+ \Ev_{y,\gamma}(X)^2\Big)\\
&\le \frac{C_3}{\eps}\left(\frac{1}{\eps}+\frac{y^2}{\eps^2}\right)\, .
\end{align}
Substituting in Eq~(\ref{eq:PartialGamma}), we obtain the claimed bound on $\big|\partial_{\gamma} F(y;\gamma) \big|$.
\end{proof}

We use Theorem \ref{thm:Rank1} which applies to the rank one matrix in  \myeqref{eq:SpikedDefSpecial}, with the setting
$\bv = \bx_0/\sqrt{n}$.  We conclude that the state evolution result in  \myeqref{eq:rank1_SE} applies with  $\mu_t$, $\sigma_t$  
determined via Eqs.~(\ref{eq:MutRec}), (\ref{eq:SigtRec}), and initial condition $\mu_0=(\lambda^{2} -1)$, $\sigma_0^2 =( \lambda^{2} -1)$ (because the initial
condition in Theorem \ref{thm:Special} is scaled by a factor $\lambda(\lambda^2-1)^{1/2}$ with respect to the statement of Theorem \ref{thm:Rank1}).

Further note that -- by Cauchy-Schwarz inequality -- the signal-to-noise ratio $\mu_{t+1}/\sigma_{t+1}$ is maximized by setting
$f(y;t) = f_{\sBayes}(y;t)$  (or any positive multiple of this function) where
\begin{align}
f_{\sBayes}(y;t) =\lambda\E\{X_0 \mid \mu_{t}X_0+\sigma_tG =y\}\, , \label{eq:fBayesProof}
\end{align}
whence  Eqs.~(\ref{eq:MutRec}), (\ref{eq:SigtRec}) yield
\begin{align}
\mu_{t+1} &= \lambda^2\E\big\{\E\{X_0 \mid  \mu_{t}X_0+\sigma_tG \}^2\big\} = \lambda^2\Big\{1-\mmse\big(\mu_t^2/\sigma_t^2\big)\Big\}\, ,\\
\sigma^2_{t+1} &= \lambda^2\E\big\{\E\{X_0 \mid \mu_{t}X_0+\sigma_tG \}^2\big\} = \lambda^2\Big\{1-\mmse\big(\mu_t^2/\sigma_t^2\big)\Big\}\, .
\end{align}
In particular, we have $\mu_t=\sigma_t^2$ for all $t\ge 1$, and we selected the initial condition to ensure that this holds for 
$t=0$ as well. Setting $\gamma_t = \mu_t^2/\sigma_t^2$, we obtain that $\gamma_t$ satisfies 
the state evolution equation (\ref{eq:SEspecial_2}), with initialization (\ref{eq:SEspecial_1}). Further, the identity $\mu_t=\sigma_t^2$ implies
$\mu_t = \gamma_t$, $\sigma^2_t = \gamma_t$ whence the choice (\ref{eq:fBayesProof}) concides with the one of  Eq.~(\ref{eq:OptFspecial}).
Finally Eq.~(\ref{eq:special_AMP_conv}) follows from Eq.~(\ref{eq:rank1_SE}) using the same identities.

Applying  (\ref{eq:rank1_SE}) to suitable test functions $\psi$, we obtain
\begin{align}
\lim_{n\to\infty}\frac{|\<\hbx^t(\bA),\bx_0\>|}{\|\hbx^t(\bA)\|_2\|\bx_0\|_2} &= \frac{\sqrt{\gamma_t}}{\lambda}\, ,\\
\lim_{n\to\infty}\frac{1}{n}\min_{s\in \{+1,-1\}} \| \hbx^t(\bA)-\bx_0\|_2^2&= 1- \frac{\gamma_{t}(\lambda)}{\lambda^2}\, .
\end{align}
To complete the proof, we need to prove that $\lim_{t\to\infty}\gamma_t =  \gamma_{\sALG}(\lambda)$.
To this end, let $M_{\lambda}(\gamma) \equiv \lambda^2\{1-\mmse(\gamma)\}$.
Since the minimum mean square error is bounded above by the minimum error of any linear estimator, we have
$\mmse(\gamma)\le (1+\gamma)^{-1}$ (where we used $\E(X_0)=0$, $\E(X_0^2) = 1$). Hence
\begin{align}
M_{\lambda}(\gamma) \ge \frac{\lambda^2\gamma}{1+\gamma}\, .
\end{align}
Also, $\gamma\mapsto\mmse(\gamma)$ is non-increasing.
Hence $\gamma\mapsto M_{\lambda}(\gamma)$ is a non-decreasing function with $M_{\lambda}(\gamma)>\gamma$ for $\gamma\in (0,\gamma_{\sALG})$, $\gamma_0\le \gamma_{\sALG}$,
which immediately implies the claim.

\section{Proof of Corollary \ref{coro:ConfInt}}
\label{app:CoroConfInt}

For the sake of concreteness, we will assume the construction of confidence intervals via Bayes AMP, cf. Eq.~\eqref{eq:ConfInterval}.
The proof is unchanged for the more general construction in \eqref{eq:ConfInterval_2}.

First we note that substituting the estimate of $\lambda$ given by $\hlambda(\bA)$ does not change the behavior of
$\obx^t$, $\hgamma_t$.
\begin{lemma}\label{lemma:HatX}
Under the assumptions of Corollary \ref{coro:ConfInt}, the following limits hold almost surely, for any fixed $t\ge 0$:
\begin{align}
\lim_{n\to\infty}\hlambda(\bA) &=\lambda\, ,\label{eq:HatLambdaConvergence}\\
\lim_{n\to\infty}\hgamma_t &=\gamma_t\, ,\label{eq:HatGammaConvergence}\\
\lim_{n\to\infty} \frac{1}{n}\|\obx^t-\bx^t\|^2_2& = 0 \, .  \label{eq:HatXConvergence}
\end{align}
\end{lemma}
\begin{proof}
Recall that for $\lambda>1$, we have $\lambda_{\max}(\bA) \to (\lambda+\lambda^{-1})$ almost surely \cite{benaych2012singular}. Since the function
$g(x) = (x+\sqrt{x^2-4})/2$ is continuous for $x>2$, with $g(\lambda+\lambda^{-1}) = \lambda$, we also
have  $\hlambda(\bA) = g(\lambda_{\max}(\bA))\to \lambda$.
 
In order to prove Eq.~(\ref{eq:HatGammaConvergence}), note that $\gamma\mapsto \mmse(\gamma)$ is continuous for $\gamma\in (0,\infty)$.
Indeed, it is non-increasing by the optimality of $\mmse$. Further define $Y_t = X_0+G_t$ where $(G_t)_{t\ge 0}$ is a standard Brownian motion independent of
$X_0$, with $G_0=0$. Then $\mmse(\gamma) = \E\{[X_0-\E\{X_0|Y_{1/\gamma}\}]^2\}$. Let $\gamma_1>\gamma_2\ge0$ and $t_i\equiv 1/\gamma_i$. 
By optimality of conditional expectation, we have, for any measurable function $h:\reals\to\reals$:
\begin{align}
\mmse(\gamma_1)\le\mmse(\gamma_2)& \le \E\{[X_0-h(Y_{t_2})]^2\}=\E\{[X_0-\E(X_0|Y_{t_1})+\E(X_0|Y_{t_1})-h(Y_{t_2})]^2\} \\
& = \E\{[X_0-\E(X_0|Y_{t_1})]^2\}+\E\{[\E(X_0|Y_{t_1})-h(Y_{t_2})]^2\}\\
& = \mmse(\gamma_1)+\E\{[\E(X_0|Y_{t_1})-h(Y_{t_2})]^2\}\, .
\end{align}
Note that $\E(X_0|Y_{1/\gamma}=y) = F(\gamma y;\gamma)$. By setting $h(y) = F(\gamma_1 y;\gamma_1)$, and denoting by $L(\gamma)$
the Lipschitz constant of $F(\,\cdot\, ;\gamma)$, we obtain
\begin{align}
\big|\mmse(\gamma_1)-\mmse(\gamma_2)\big| &\le \E\big\{[F(\gamma_1Y_{1/\gamma_1};\gamma_1)-F(\gamma_1 Y_{1/\gamma_2};\gamma_1)]^2\big\}\\
&\le \gamma_1^2 L(\gamma_1) \E\big\{[Y_{1/\gamma_1}-Y_{1/\gamma_2}]^2\big\}\\
&\le \gamma_1^2 L(\gamma_1) \Big|\frac{1}{\gamma_1}-\frac{1}{\gamma_2}\Big|\, .
\end{align}

We then  proceed by induction over $t$. Using  Eq.~(\ref{eq:SEspecial_2}):
\begin{align}
\lim_{n\to\infty}\hgamma_{t+1} &= \lim_{n\to\infty}\hlambda(\bA)\Big\{1- \lim_{n\to\infty}\mmse(\hgamma_t)\Big\}\\
& = \lambda\big(1-\mmse(\gamma_t)\big) = \gamma_{t+1}\, .
\end{align}

Finally  Eq.~(\ref{eq:HatXConvergence}) is also proved by induction over $t$. Note that $\obx^t$ is defined recursively 
as per Eq.~(\ref{eq:AMPspecial0}) with $\lambda$, $\gamma_t$ in the definition of $f_t$ in \myeqref{eq:OptFspecial} repalced by $\hlambda$, $\hgamma_t$.
Explicitly,
\begin{align}
\obx^{t+1} & = \bA \, \hf_t(\obx_t) - \hsb_t\, \hf_{t-1}(\obx^{t-1})\, ,\\
\hf_t(y) &= \hlambda F\Big(y; \hgamma_t\Big)\, ,\;\;\;\;\;\;
\hsb_t  = \frac{1}{n}\sum_{i=1}^n \hf'_t(\ox^t_i) \, .
\end{align}
Therefore
\begin{align}
\|\obx^{t+1}-\bx^{t+1}\|_2 & \le \|\bA\|_{\op} \, \big\|\hf_t(\obx_t)-f_t(\bx_t)\big\|_2+ |\hsb_t-\sb_t|\, \big\| f_{t-1}(\bx^{t-1})\big\|_2  \nonumber \\ 
& \quad  +  |\hsb_t|\, \big\| \hf_{t-1}(\obx^{t-1}) - f_{t-1}(\bx^{t-1})\big\|_2\\
& \equiv \sqrt{n}\big(R_1(t;n)+R_2(t;n)+R_3(t;n)\big)\,  .
\end{align}
Consider the first term. Since  $\|\bA\|_{\op}\to 2$ almost surely, for large enough $n$ we almost surely have 
\begin{align}
R_1(t;n)^2 & \le \frac{10\lambda^2 }{n} \sum_{i=1}^n\big[F(x^t_i;\hgamma_t)-F(x^t_i;\gamma_t)\big]^2 +
\frac{10}{n}\sum_{i=1}^n\big[\hlambda F(\ox^t_i;\hgamma_t)-\lambda F(x^t_i;\hgamma_t)]^2 \nonumber\\
&\stackrel{(a)}{\le} \frac{10C_t\lambda^2}{n} (1+\|\bx^t\|_2^2) |\hgamma_t-\gamma_t| + \frac{20C_t\hlambda^2}{n}\big\|\obx^t-\bx^t\big\|_2^2+
\frac{20C_t}{n}\big\|\bx^t\big\|_2^2|\hlambda-\lambda|^2\, ,
\end{align}
where step $(a)$ is obtained using Lemma \ref{lemma:LipCE}.  We next take the limit $n\to \infty$ and use the induction hypothesis together with Eqs.~(\ref{eq:HatLambdaConvergence}),  (\ref{eq:HatGammaConvergence}), 
and the fact that $\lim\sup_{n\to\infty}\|\bx^t\|_2^2/n<\infty$, which follows  by Theorem \ref{thm:Special}.
We claim that $\lim_{n\to\infty}|\hsb_t-\sb_t| = 0$, whence $\lim\sup_{n\to\infty}|\hsb_t|<\infty$ (since $\sb_t$ is asymptotically bounded, per Eq.~(\ref{eq:fpr_conv})).
Since $\| \hf_{t-1}(\obx^{t-1}) - f_{t-1}(\bx^{t-1})\|_2/\sqrt{n}\to 0$ by the same argument above, this implies $R_3(t;n)\to 0$.
Further,  $\lim\sup_{n\to\infty}\| f_{t-1}(\bx^{t-1})\|_2<\infty$,  we also get $R_2(t;n)\to 0$.

We are left with the task of showing $\lim_{n\to\infty}|\hsb_t-\sb_t| = 0$. Note that  $\lambda,\hlambda,\gamma_t, \hgamma_t\in [1/C_0,C_0]$ 
almost surely for all $n$ large enough. Hence, by Lemma \ref{lemma:LipCE},  $|f_t'(x^t_i)|,  |\hf'_t(\ox_i^t)|\le C$ for some constant $C>0$,
Therefore, for any constant $M$, the following holds almost surely for all $n$ large enough
\begin{align}
 |\sb_t-\hsb_t|  & \le \frac{1}{n}\sum_{i=1}^n\big|f_t'(x^t_i) - \hf'_t(\ox_i^t)\big|\, \ind\big(|x^t_i|, |\ox^t_i|\le M\big) +
\frac{2C}{n}\sum_{i=1}^n\ind\big(|x^t_i|\ge M\big) +\frac{2C}{n}\sum_{i=1}^n\ind\big(|\ox^t_i|\ge M\big) \\
& \le \frac{1}{n}\sum_{i=1}^n\big|\lambda\partial_y F'( x^t_i;\gamma_t) - \hlambda\partial_yF(\ox_i^t;\hgamma_t)\big|\, \ind\big(|x^t_i|, |\ox^t_i|\le M\big) +
\frac{2C}{n M^2}\|\bx^t\|_2^2+\frac{2C}{nM^2}\|\obx^t\|_2^2\, .
\end{align}
Again by Lemma \ref{lemma:LipCE}, $\partial_y F\in C^{\infty}(\reals\times (0,\infty))$ and hence Lipschitz continuous on the 
compact set $[-M,M]\times [1/C_0,C_0]$, with Lipschitz constant $L(M)$, whence 
\begin{align}
|\sb_t-\hsb_t| &\le \frac{L(M)\lambda}{n}\| \bx^t-\obx^t\|_1 +\frac{L(M)}{n}\|\obx^t\|_1|\lambda-\hlambda|+ L(M) |\gamma_t-\hgamma_t| +
\frac{2C}{n M^2}\|\bx^t\|_2^2+\frac{2C}{nM^2}\|\obx^t\|_2^2\\
&\le \frac{L(M)\lambda}{\sqrt{n}}\|\bx^t-\obx^t\|_2 +\frac{L(M)}{\sqrt{n}}\|\obx^t\|_2|\lambda-\hlambda|+ L(M) |\gamma_t-\hgamma_t| +
\frac{2C}{n M^2}\|\bx^t\|_2^2+\frac{2C}{nM^2}\|\obx^t\|_2^2\, .
\end{align}
Using $|\hlambda-\lambda|\to 0$, $|\gamma_t-\hgamma_t| \to 0$, $\|\bx^t-\obx^t\|_2/\sqrt{n}\to 0$ (proved above),
and $\lim\sup_{n\to\infty}\|\bx^t\|_2/\sqrt{n}\le C'$, $\lim\sup_{n\to\infty}\|\obx^t\|_2/\sqrt{n}\le C'$, we get
\begin{align}
\lim\sup_{n\to\infty}|\sb_t-\hsb_t| &\le \frac{C''}{M^2}\, ,
\end{align}
whence the claim follows since $M$ is arbitrary.
\end{proof}

We are now in position to prove Corollary \ref{coro:ConfInt}.
\begin{proof}[Proof of Corollary \ref{coro:ConfInt}]
First note that for any function $\psi:\reals\times\reals\to\reals$  with $|\psi(\bx)-\psi(\by)|\le C(1+\|\bx\|_2+\|\by\|_2)\|\bx - \by\|_2$,  by Theorem \ref{thm:Special}, and Lemma \ref{lemma:HatX} we have
\begin{align}
\lim_{n\to\infty}\frac{1}{n}\sum_{i=1}^n\psi(x_{0,i},\ox_i^t) = \E\big\{\psi\big(X_0,\gamma_t\, X_0+\gamma_t^{1/2}\, Z\big)\big\}\, , 
\end{align}

We begin by proving Eq.~\eqref{eq:AS-ConfInt}. Define
\begin{align}
\hB(x;\alpha,t) = \left[\frac{1}{\hgamma_t} x - \frac{1}{\sqrt{\hgamma_t}} \Phi^{-1}(1-\tfrac{\alpha}{2}), \, 
\frac{1}{\hgamma_t} x+\frac{1}{\sqrt{\hgamma_t}} \Phi^{-1}(1-\tfrac{\alpha}{2}) \right]\, , \label{eq:ConfInterval_BIS_0}\\
B(x;\alpha,t) = \left[\frac{1}{\gamma_t} x - \frac{1}{\sqrt{\gamma_t}} \Phi^{-1}(1-\tfrac{\alpha}{2}), \, 
\frac{1}{\gamma_t} x+\frac{1}{\sqrt{\gamma_t}} \Phi^{-1}(1-\tfrac{\alpha}{2}) \right]\, .\label{eq:ConfInterval_BIS}
\end{align}
For $x\in \reals$ and $S\subseteq \reals$, we let $d(x,S)\equiv \inf\{|x-y|:\; y\in S\}$, and $S^c\equiv \reals\setminus S$.
Fixing $\epsilon>0$, we define the Lipschitz-continuous functions
\begin{align}
\hpsi_{\epsilon,+}(x_0,x) = \begin{cases}
1 & \;\;\;\; \mbox{ if $x_0\in \hB(x;\alpha,t)$,}\\
0 & \;\;\;\; \mbox{ if $d(x_0, \hB(x;\alpha,t))\ge \epsilon$,}\\
1-d(x_0, \hB(x;\alpha,t))/\epsilon & \;\;\;\; \mbox{ otherwise,}
\end{cases}
\label{eq:hpsi+def}\\
\hpsi_{\epsilon,-}(x_0,x) = \begin{cases}
1 & \;\;\;\; \mbox{ if $d(x_0, \hB(x;\alpha,t)^c)\ge \epsilon$,}\\
0 & \;\;\;\; \mbox{ if $x_0\in \hB(x;\alpha,t)^c$,}\\
d(x_0, \hB(x;\alpha,t)^c)/\epsilon & \;\;\;\; \mbox{ otherwise,}
\end{cases}
\end{align}
as well as the analogous functions for $B(x;\alpha,t)$:
\begin{align}
\psi_{\epsilon,+}(x_0,x) = \begin{cases}
1 & \;\;\;\; \mbox{ if $x_0\in B(x;\alpha,t)$,}\\
0 & \;\;\;\; \mbox{ if $d(x_0, B(x;\alpha,t))\ge \epsilon$,}\\
1-d(x_0, B(x;\alpha,t))/\epsilon & \;\;\;\; \mbox{ otherwise,}
\end{cases}\\
\psi_{\epsilon,-}(x_0,x) = \begin{cases}
1 & \;\;\;\; \mbox{ if $d(x_0, B(x;\alpha,t)^c) \ge \epsilon$,}\\
0 & \;\;\;\; \mbox{ if $x_0\in B(x;\alpha,t)^c$,}\\
d(x_0, B(x;\alpha,t)^c)/\epsilon & \;\;\;\; \mbox{ otherwise,}
\end{cases}
\label{eq:psi-def}
\end{align}
By the same argument as in the proof of Lemma  \ref{lemma:HatX}, we have almost surely
\begin{align}
\lim_{n\to\infty}\frac{1}{n}\sum_{i=1}^n\hpsi_{\epsilon,\pm}(x_{0,i},\ox^t_i) & =\lim_{n\to\infty}\frac{1}{n}\sum_{i=1}^n\psi_{\epsilon,\pm}(x_{0,i}, \ox^t_i)  \\
& = \E\big\{\psi_{\epsilon,\pm}(X_{0},\gamma_tX_0+ \gamma_t^{1/2}\, Z )\big\}. \label{eq:psi_pm_lim}
\end{align}
where the second equality follows from Theorem \ref{thm:Special}. 
On the other hand,
\begin{align}
 \hpsi_{\epsilon,-}(x_{0,i},\ox^t_i)\le  \ind\big(x_{0,i}\in \hJ_i(\alpha;t)\big)  \le \hpsi_{\epsilon,+}(x_{0,i},\ox^t_i)\, ,
\end{align}
which implies
\begin{align*}
& \lim_{\epsilon\to 0} \E\big\{\psi_{\epsilon,-}(X_{0},\gamma_t\, X_0+ \gamma_t^{1/2}\, Z )\big\} \le 
\lim\inf_{n\to\infty}\frac{1}{n}\sum_{i=1}^n \ind\big(x_{0,i}\in \hJ_i(\alpha;t)\big)\\
& \le\lim\sup_{n\to\infty}\frac{1}{n}\sum_{i=1}^n \ind\big(x_{0,i}\in \hJ_i(\alpha;t)\big)  \leq 
\lim_{\epsilon\to 0} \E\big\{\psi_{\epsilon,+}(X_{0},\gamma_t\, X_0+ \gamma_t^{1/2}\, Z )\big\}\, .
\end{align*}
The proof is completed by noticing that by monotone convergence, 
\begin{align}
&\lim_{\epsilon\to 0} \E\big\{\psi_{\epsilon,-}(X_{0},\gamma_t X_0+ \gamma_t^{1/2} Z )\big\} = 
\prob\big\{|Z|\le \Phi^{-1}(1-\tfrac{\alpha}{2})\big\} = 1- \alpha\, ,\\
&\lim_{\epsilon\to 0} \E\big\{\psi_{\epsilon,+}(X_{0},\gamma_t X_0+ \gamma_t^{1/2} Z )\big\} = 
\prob\big\{|Z|< \Phi^{-1}(1-\tfrac{\alpha}{2})\big\} = 1- \alpha\, . \label{eq:lim_psi_e_upper}
\end{align}

In order to prove Eq.~\eqref{eq:ValidNull}, we use a similar argument, with a slightly different test function. Define
$u_{\delta}(x_0) = (1-|x_0|/\delta)_+$ and
\begin{align}
\hxi_{\epsilon,\delta,\pm}(x_0,x) &= \hpsi_{\epsilon,\pm} (0,x) u_{\delta}(x_0)\, ,\\
\xi_{\epsilon,\delta,\pm}(x_0,x) &= \psi_{\epsilon,\pm}(0,x) u_{\delta}(x_0)\, .
&
\end{align}
Proceeding  as above, we obtain
\begin{align}
\lim_{n\to\infty}\frac{1}{n}\sum_{i=1}^n\hxi_{\epsilon,\delta,\pm}(x_{0,i},\ox^t_i) & = \E\big\{\xi_{\epsilon,\pm}(X_{0},\gamma_tX_0+ \gamma_t^{1/2}\, Z )\big\}\, ,
\label{eq:hxi_lim}
\end{align}
Upper and lower bounding the indicator function by $\hpsi_{\epsilon,\pm}$ as in the previous proof, we then obtain that, for any
$\delta>0$
\begin{align}
\lim_{n\to\infty}\frac{1}{n}\sum_{i=1}^n \ind\big(0\in \hJ_i(\alpha;t)\big)\, u_\delta(x_{0,i}) &=\E\big\{ u_{\delta}(X_0)\, \ind\big(\sqrt{\gamma_t}X_0-c_{\alpha}\le Z\le \sqrt{\gamma_t}X_0+c_{\alpha}\big)\big\}\\
&=\E\big\{ u_{\delta}(X_0)\, \big(\Phi(\sqrt{\gamma_t}X_0+c_{\alpha})-\Phi(\sqrt{\gamma_t}X_0-c_{\alpha})\big)\big\}
\end{align}
where $c_{\alpha} \equiv \Phi^{-1}\big(1-\frac{\alpha}{2}\big)$. By taking $\delta\to 0$ and using monotone convergence, we get
\begin{align}
\lim_{\delta\to 0}\lim_{n\to\infty}\frac{1}{n}\sum_{i=1}^n \ind\big(0\in \hJ_i(\alpha;t)\big)\, u_\delta(x_{0,i}) &= (1-\eps)(1-\alpha)\, .
\end{align}
On the other hand,  
\begin{align}
0\le\frac{1}{n}\sum_{i=1}^n \ind\big(0\in \hJ_i(\alpha;t)\big)\, u_\delta(x_{0,i}) -
\frac{1}{n}\sum_{i=1}^n \ind\big(0\in \hJ_i(\alpha;t);\; x_{0,i}=0\big)\le
 \frac{1}{n}\sum_{i=1}^n u_\delta(x_{0,i}) -\frac{1}{n}\|\bx_0\|_0\, .
\end{align}
Since, using the assumption
\begin{align}
\lim_{\delta\to 0}\lim_{n\to\infty} \left\{\frac{1}{n}\sum_{i=1}^n u_\delta(x_{0,i}) -\frac{1}{n}\|\bx_0\|_0\right\} =
\lim_{\delta\to 0}\left\{\E u_{\delta}(X_0)-(1-\eps)\right\} = 0\, ,
\label{eq:spar_lim}
\end{align}
we obtain 
\begin{align}
\lim_{n\to\infty}\frac{1}{n}\sum_{i=1}^n \ind\big(0\in \hJ_i(\alpha;t);\; x_{0,i}=0\big) =\lim_{\delta\to 0}\lim_{n\to\infty}\frac{1}{n}\sum_{i=1}^n \ind\big(0\in \hJ_i(\alpha;t)\big)\, u_\delta(x_{0,i}) &= (1-\eps)(1-\alpha)\, .
\end{align}
By dominated convergence, this also implies 
\begin{align}
\lim_{n\to\infty}\frac{1}{n}\sum_{i: x_{0,i}=0} \prob\big(0\in \hJ_i(\alpha;t)\big) =  (1-\eps)(1-\alpha)\, .
\end{align}
This can equivalently be rewritten as
\begin{align}
\lim_{n\to\infty}\frac{1}{n}\sum_{i: x_{0,i}=0} \prob\big(p_i(t)\ge\alpha\big) =  (1-\eps)(1-\alpha)\, .
\end{align}
Let $S_0(n) \equiv\{i\in [n]:x_{0,i}=0\}=[n]\setminus\supp(\bx_0(n))$.
Notice that the $p$-values $(p_i(t))_{i\in S_0(n)}$ are exchangeable. Hence for any sequence $i_0(n)\in S_0(n)$,
we have
\begin{align}
\lim_{n\to\infty}\frac{1}{n}|S_0(n)| \prob\big(p_{i_0(n)}(t) \ge \alpha\big) =  (1-\eps)(1-\alpha)\, .
\end{align}
Since by assumption $|S_0(n)|/n\to (1-\eps)$, the claim \eqref{eq:ValidNull} follows.
\end{proof}

\section{Proof of Corollary \ref{corr:FDR}}
\label{app:FDR}

Again, for concreteness we assume  the construction of $p$-values via Bayes AMP, as per Eq.~\eqref{eq:pi_Bayes}.
The proof is unchanged for the more general construction in \eqref{eq:pi_def}.

Using the definitions of $p_i(t)$ and $\hFDP(s;t)$ from  Eqs. \eqref{eq:pi_Bayes} and \eqref{eq:FDP_def}, the threshold $s_*(\alpha;t)$ in \myeqref{eq:thresh_rej_set} can be expressed as 
\beq
s_*(\alpha;t) =  \inf \left\{\, s \in [0,1]: \; \frac{1}{n} \vee \left( \frac{1}{n} \sum_{i=1}^n \ind\left( |\ox^t_i | \geq \sqrt{\hgamma_t} \,\Phi^{-1}\left(1 -{s}/{2} \right) \right) \right) \le \frac{s}{\alpha} \, \right\}\, .
\label{eq:s_star_eq}
\eeq
We first show that $ \lim_{n\to \infty} s_*(\alpha;t) = \bar{s}(\alpha;t)$ almost surely, where 
\beq
\bar{s}(\alpha;t) \equiv \inf \left\{\, s \in [0,1]: \;  \P \left(  \abs{ \sqrt{\gamma_t} X_0 +  Z } \, \geq \, \Phi^{-1}\left(1 - {s}/{2} \right) \right) <  \frac{s}{\alpha} \, \right\}
\, .
\label{eq:salph_def}
\eeq
For  $s \in [0,1]$, define the sets
\begin{align}
\hC(s;t) =\left( -\infty, \, \sqrt{\hgamma_t}\,  \Phi^{-1}(1-\tfrac{s}{2}) \right] \, \cup \, [ \sqrt{\hgamma_t}\, \Phi^{-1}(1-\tfrac{s}{2}), \,  \infty), \\
C(s;t) =\left( -\infty, \, \sqrt{\gamma_t}\,  \Phi^{-1}(1-\tfrac{s}{2}) \right] \, \cup \, [ \sqrt{\gamma_t}\, \Phi^{-1}(1-\tfrac{s}{2}), \,  \infty).
\end{align}
We define the following test functions, similarly to Eqs. \eqref{eq:hpsi+def}--\eqref{eq:psi-def}.  For $x\in \reals$ and $S\subseteq \reals$, recall that $d(x,S)\equiv \inf\{|x-y|:\; y\in S\}$, and $S^c\equiv \reals\setminus S$.
For fixed $\epsilon>0$ we define the Lipschitz-continuous functions
\begin{align}
\hpsi_{\epsilon,+}(x; \, s) = \begin{cases}
1 & \;\;\;\; \mbox{ if $x\in \hC(s;t)$,}\\
0 & \;\;\;\; \mbox{ if $d(x, \hC(s;t))\ge \epsilon$,}\\
1-d(x, \hC(s;t))/\epsilon & \;\;\;\; \mbox{ otherwise,}
\end{cases}
\label{eq:hpsi+def2} \\
\hpsi_{\epsilon,-}(x; s) = \begin{cases}
1 & \;\;\;\; \mbox{ if $d(x, \hC(s;t)^c)\ge \epsilon$,}\\
0 & \;\;\;\; \mbox{ if $x \in \hC(s;t)^c$,}\\
d(x, \hC(s;t)^c)/\epsilon & \;\;\;\; \mbox{ otherwise,}
\end{cases}
\label{eq:hpsi-def2} 
\end{align}
The analogous functions for $C(s;\, t)$, denoted by $\psi_{\epsilon,+}(x; \, s)$  and $\psi_{\epsilon,-}(x; s)$, are defined by replacing $\hat{C}(s;\, t)$ with $C(s;\, t)$ in Eqs. \eqref{eq:hpsi+def2}--\eqref{eq:hpsi-def2}, respectively.

Using the same argument as in the proof of Lemma  \ref{lemma:HatX}, we have almost surely
\begin{align}
\lim_{n\to\infty}\frac{1}{n}\sum_{i=1}^n\hpsi_{\epsilon,\pm}(\ox^t_i; \, s) & =\lim_{n\to\infty}\frac{1}{n}\sum_{i=1}^n\psi_{\epsilon,\pm}(\ox^t_i; \, s)   = \E\big\{\psi_{\epsilon,\pm}(\gamma_tX_0+ \gamma_t^{1/2}\, Z; \, s )\big\}. \label{eq:psi_pm_lim_s}
\end{align}
where the second equality follows from Theorem \ref{thm:Special}.   Furthermore, we note that
\begin{align}
 \hpsi_{\epsilon,-}(\ox^t_i; \, s) \, \le \,  \ind\big(\, \ox^t_i \in \hC(s; \,t)\big) = \frac{1}{n} \sum_{i=1}^n \ind\left( |\ox^t_i | \geq \sqrt{\hgamma_t} \,\Phi^{-1}\left(1 -s/2 \right)  \right)  \,   \le \hpsi_{\epsilon,+}(\ox^t_i; \, s).
\label{eq:h-ind+}
\end{align}

For fixed $\epsilon >0$, let
\begin{align}
\hat{s}_{\epsilon,+}(\alpha;t) =  \inf \left\{\, s \in [0,1]: \; \frac{1}{n} \vee  \left( \frac{1}{n} \sum_{i=1}^n \hpsi_{\epsilon, +}( \ox^t_i; \, s ) \right) \le \frac{s}{\alpha} \, \right\}, \\
\hat{s}_{\epsilon,-}(\alpha;t) =  \inf \left\{\, s \in [0,1]: \; \frac{1}{n} \vee  \left( \frac{1}{n} \sum_{i=1}^n \hpsi_{\epsilon, -}( \ox^t_i; \, s ) \right) \le \frac{s}{\alpha} \, \right\}.
\end{align}
Since $\hpsi_{\epsilon, +}( x; \, s ), \, \hpsi_{\epsilon, -}(x; \, s ) $ and $\ind(x \in \hC(s;t))$ are all positive and increasing in $s$
(for any fixed $x$), Eq. \eqref{eq:h-ind+} implies that
\beq
\frac{\alpha}{n} \leq \hat{s}_{\epsilon,-}(\alpha;t) \le \, s_{*}(\alpha;t) \, \le \hat{s}_{\epsilon,+}(\alpha;t).
\label{eq:s_order}
\eeq
Furthermore, using Eq. \eqref{eq:psi_pm_lim_s} we obtain that
\beq
\lim_{n \to \infty}\, \hat{s}_{\epsilon, \pm}(\alpha;t) = \inf \left\{\, s \in [0,1]: \; \E\big\{\psi_{\epsilon,\pm}(\gamma_tX_0+ \gamma_t^{1/2}\, Z; \, s )\big\} < \frac{s}{\alpha} \, \right\} \quad \text{a.s.}
\eeq
By the monotone convergence theorem, we have
\begin{align}
\lim_{\epsilon\to 0} \E\big\{\psi_{\epsilon,-}(\gamma_t\, X_0+ \gamma_t^{1/2}\, Z; \, s)\big\} & = \P\left(\abs{\sqrt{\gamma_t}\, X_0+ \, Z} \geq \Phi^{-1}(1 -s/2)  \right), \\
 \lim_{\epsilon\to 0} \E\big\{\psi_{\epsilon,+}(\gamma_t\, X_0+ \gamma_t^{1/2}\, Z; \, s )\big\}
& =  \P\left(\abs{\sqrt{\gamma_t}\, X_0+ \, Z} \geq \Phi^{-1}(1 -s/2)  \right).
\label{eq:Epsi_lim}
\end{align}
Therefore, taking $\epsilon \to 0$, from Eqs. \eqref{eq:s_order}--\eqref{eq:Epsi_lim} we obtain
\beq
\lim_{n \to \infty} \, s_{*}(\alpha; \, t) \, = \, \bar{s}(\alpha;t) \equiv \inf \left\{\, s \in [0,1]: \;  \P \left(  \abs{ \sqrt{\gamma_t} X_0 +  Z } \, \geq \, \Phi^{-1}\left(1 - {s}/{2} \right) \right) <  \frac{s}{\alpha} \, \right\} \quad \text{ a.s. }%
\label{eq:bars_def}
\eeq

We now prove the asymptotic FDR result in Eq. \eqref{eq:FDRbound} by showing that the following two limits hold almost surely:
\begin{align}
& \lim_{n \to \infty} \,   \frac{ |\hS(\alpha;t)|}{n}   \, =  \, \frac{\bar{s}(\alpha;t)}{\alpha}, 
%
\qquad  \lim_{n \to \infty} \,  \frac{|\hS(\alpha;t)\cap \{i:\, x_{0,i}=0\}|}{n}  \, = \,  (1-\eps) \bar{s}(\alpha;t).
\label{eq:set_sizes}
\end{align}
The continuous mapping theorem then implies that almost surely
\beq
\lim_{n \to \infty} \frac{ |\hS(\alpha;t)| /n}{ 1/n  \, \vee \, |\hS(\alpha;t)\cap \{i:\, x_{0,i}=0\}|/n } = (1-\eps) \alpha.
\eeq
The claim in Eq. \eqref{eq:FDRbound} then follows from dominated convergence. 

To prove the first result in Eq. \eqref{eq:set_sizes}, notice that 
\beq
\frac{ |\hS(\alpha;t)|}{n} = \frac{1}{n} \sum_{i=1}^n \ind( \ox_i^t \in \hC(s_{*}(\alpha; t); \, t) ) = \frac{1}{n} \sum_{i=1}^n \ind\left( |\ox^t_i | \geq \sqrt{\hgamma_t} \,\Phi^{-1}\left(1 - s_{*}(\alpha; t)/2 \right)  \right) ,
\eeq
and 
\begin{align}
 \hpsi_{\epsilon,-}(\ox^t_i; \, s_{*}(\alpha; t)) \, \le \,  \ind\big(\, \ox^t_i \in \hC(s_{*}(\alpha; t); \,t)\big)   \,   \le \hpsi_{\epsilon,+}(\ox^t_i; \, s_{*}(\alpha; t)).
\end{align}
Since $\hgamma_t \to \gamma_t$ and $s_{*}(\alpha; t) \to \bar{s}(\alpha;t)$ almost surely,
by the same argument as in the proof of Lemma  \ref{lemma:HatX}, we have 
\begin{align}
\lim_{n\to\infty}\frac{1}{n}\sum_{i=1}^n\hpsi_{\epsilon,\pm}(\ox^t_i; \, s_{*}(\alpha; t) ) & =\lim_{n\to\infty}\frac{1}{n}\sum_{i=1}^n\psi_{\epsilon,\pm}(\ox^t_i; \, \bar{s}(\alpha; t) )   = \E\big\{\psi_{\epsilon,\pm}(\gamma_tX_0+ \gamma_t^{1/2}\, Z; \, \bar{s}(\alpha; t))\big\}. 
\end{align}
where the second equality follows from Theorem \ref{thm:Special}. Hence
\begin{equation}
\begin{split}
& \lim_{\epsilon\to 0} \E\big\{\psi_{\epsilon,-}(\gamma_t\, X_0+ \gamma_t^{1/2}\, Z; \, \bar{s}(\alpha; t))\big\} \le 
\lim\inf_{n\to\infty}\frac{1}{n}\sum_{i=1}^n \ind\big(\, \ox^t_i \in \hC(s_{*}(\alpha; t); \,t) \big)\\
& \le\lim\sup_{n\to\infty}\frac{1}{n}\sum_{i=1}^n \ind\big(\, \ox^t_i \in \hC(s_{*}(\alpha; t); \,t) \big)  \leq 
\lim_{\epsilon\to 0} \E\big\{\psi_{\epsilon,+}(\gamma_t\, X_0+ \gamma_t^{1/2}\, Z; \, \bar{s}(\alpha; t))\big\}\, .
\end{split}
\label{eq:hC_lims}
\end{equation}
Taking $\epsilon \to 0$, by monotone convergence we find that the limits on the left and the right in \myeqref{eq:hC_lims} are both equal to $\P(\abs{\sqrt{\gamma_t}\, X_0+ \, Z} \geq \Phi^{-1}(1 - \, \bar{s}(\alpha;t)/2)  )$.  Therefore,   we have almost surely
\beq
\lim_{n \to \infty} \frac{ |\hS(\alpha;t)|}{n}  =  \P(\abs{\sqrt{\gamma_t}\, X_0+ \, Z} \geq \Phi^{-1}(1 - \, \bar{s}(\alpha;t)/2)  ) = \frac{\bar{s}(\alpha;t)}{\alpha}.
\eeq
The last equality follows from the definition of $\bar{s}(\alpha;t)$ in \myeqref{eq:bars_def} which implies that $\bar{s}(\alpha;t)$
is the smallest positive solution of 
\[  \P(\abs{\sqrt{\gamma_t}\, X_0+ \, Z} \geq \Phi^{-1}(1 - \, s/2)  )  = \frac{s}{\alpha}.  \]

To prove the second equality in \myeqref{eq:set_sizes}, we use a similar argument, but with slightly different test functions. Let
$u_{\delta}(x_0) = (1-|x_0|/\delta)_+$ and
\begin{align}
\hxi_{\epsilon,\delta,\pm}(x_0,x; \, s) = \hpsi_{\epsilon,\pm} (x; \, s) u_{\delta}(x_0)\, , \qquad
\xi_{\epsilon,\delta,\pm}(x_0,x; \, s) = \psi_{\epsilon,\pm}(x; \, s) u_{\delta}(x_0)\, .
\end{align}
We note that 
\begin{align}
\hxi_{\epsilon,\delta, - }(x_0,  \, \ox^t_i; \, s_{*}(\alpha; t)) \, \le \,  \ind\big( \ox^t_i \in \hC(\, s_{*}(\alpha; t); \,t), \, x_{0,i}=0\big)  \,  \le \, 
\hxi_{\epsilon,\delta, - }(x_0,  \, \ox^t_i; \, s_{*}(\alpha; t)).
\end{align}
Proceeding as above and using arguments similar to Eqs. \eqref{eq:hxi_lim}--\eqref{eq:spar_lim}, we obtain that almost surely
\beq
\lim_{n \to \infty} \,  \frac{|\hS(\alpha;t)\cap \{i:\, x_{0,i}=0\}|}{n} 
= (1-\eps)  \P(\abs{Z} \geq \Phi^{-1}(1 - \, \bar{s}(\alpha;t)/2)  )   = (1-\eps)\bar{s}(\alpha;t),
\eeq
which completes the proof.

\section{Proof of Theorem \ref{thm:Block}}
\label{app:Block}

As discussed already in the main text the matrix $\bA$ of Eq.~(\ref{eq:BlockMatrix}) is of  the form (\ref{eq:SpikedDef}) with $\bv_1,\dots,\bv_k$ an orthonormal basis 
of $\cV_n$, the column space of $\bA_0$, and $\lambda_1=\dots=\lambda_k=\lambda$. Recall that $\bV\in\reals^{n\times k}$ denotes the matrix with columns $\bv_1,\dots,\bv_k$,
and, for economy of notation, we let $\bPhi\in\reals^{n\times k}$ denote the matrix with columns $\bphi_1,\dots,\bphi_k$. Notice that we can construct $\bV$ through the singular value decomposition 
of the matrix $\bx_0\in\reals^{n\times (k+1)}$ (recall that $k=q-1$):
\begin{align}
  \bx_0 = \sqrt{\frac{n}{q}}\, \bV\bS^{\sT}\, ,  \label{eq:bx0_svd}
\end{align}
where $\bS\in\reals^{(k+1)\times k}$ is an orthogonal matrix whose columns span the space $\cW_k= \{\bx\in\reals^{k+1}:\, \<\bx,\bfone\>=0\}$. Note that this implies
immediately that assumption (\ref{ass:SpikesDistr}) of Theorem \ref{thm:Main}, by using
\begin{align}
\tbv_i = \sqrt{q}\,  \bS^{\sT}\bx_{0,i} =\sqrt{q}\,  \bS^{\sT}\projp\be_{\sigma_i}  = \sqrt{q}\, \bfs_{\sigma_i}\, ,
\end{align}
where $\bfs_1,\dots,\bfs_{q}\in\reals^k$ are the rows of $\bS$. 
The other assumptions of  Theorem \ref{thm:Main} are immediate to verify.

As per Theorem \ref{thm:Main}, we set $\bProd_0 = \bProd =  \bPhi^{\sT}\bV\in\reals^{k\times k}$ and $\tbProd_0 = (\id-\bProd_0\bProd_0^{\sT})^{1/2}$. 
Fixing   $\eta_n=n^{-1/2+\eps}$ for $\eps\in(0,1/2)$, the set $\cG_n = \cG_n(\bLambda)$ is given by
\begin{align}
\cG_n= \Big\{\, \bQ\in\reals^{k\times k} :\; \min_{\bR\in\cR(k)}\|\bQ-(1-\lambda^{-2})\bR\|_F\le \eta_n\Big\}\, ,
\end{align}
where, with a slight abuse of notation, we denoted by $\cR(k)$ the orthogonal group in $k$ dimensions.
By Theorem \ref{thm:Main}, we have $\bProd_0 = (1-\lambda^{-2})^{1/2}\bR +O(\eta_n)$ and $\tbProd_0 = \lambda^{-1}\id_k+O(\eta_n)$ eventually almost surely, and
$\bProd_0\stackrel{{\rm d}}{\Rightarrow}  (1-\lambda^{-2})^{1/2} \bQ$ for $\bQ$ Haar distributed in $\cR(k)$.

Further, the state evolution recursion (\ref{eq:Mt_def}),  (\ref{eq:Qt_def}) yields 
\begin{align}
\bM_{t+1} &= \lambda\sqrt{q} \, \E\Big\{f(\sqrt{q}\bM_t\bfs_{\sigma}+\bQ_t^{1/2}\bG)\bfs_{\sigma}^{\sT}\big\} \, ,\label{eq:MBlockProof}\\
\bQ_{t+1} & = \E\Big\{f(\sqrt{q}\bM_t\bfs_{\sigma}+\bQ_t^{1/2}\bG) \, f(\sqrt{q}\bM_t\bfs_{\sigma}+\bQ_t^{1/2}\bG)^{\sT}\Big\}\, ,
\end{align}
where expectation is with respect to $\sigma$ uniform in $\{1,\dots\, q\}$ independent of $\bG\sim\normal(0,\id_{k+1})$.
By Eq.~(\ref{eq:main_SE_result}), and using the fact that $\bM_t$, $\bQ_t$ are continuous in the initial condition $\bM_0$, $\bQ_0$,  under the initialization 
\begin{align}
(\bM_0)_{[k],[k]} &= \bPhi^{\sT}\bV\, ,\;\;\;\;\; (\bM_0)_{k+1,[k]} = \bzero\, ,\\
\bQ_0 &= \lambda^{-1}\diag(1,\dots,1,0)\in \reals^{(k+1)\times (k+1)}\, ,
\end{align}
we get, almost surely, for $\psi:\reals^{2k+1}\to \reals$ pseudo-Lipschitz
\begin{align}
\lim_{n\to\infty}\left|\frac{1}{n}\sum_{i=1}^n\psi(\bx_i^t,\bfs_{\sigma_i}) - \E\big\{\psi(\sqrt{q}\bM_t\bfs_{\sigma}+\bQ_t^{1/2}\bG,\bfs_{\sigma})\big\}\right|= 0\, .  
\label{eq:Proof_SE_Block}
\end{align}

We next define $\tbM_t=q^{-1/2}\bM_t\bS^{\sT}$, and notice that $\bfs_{\sigma} = \bS^{\sT}\projp\be_{\sigma}$ and
$\bS\bS^{\sT} =\projp$. Multiplying Eq.~(\ref{eq:MBlockProof}) on the right by $\bS^{\sT}/\sqrt{q}$, we get
\begin{align}
\tbM_{t+1} &= \lambda\E\Big\{f(q\tbM_t\projp\be_{\sigma}+\bQ_t^{1/2}\bG)\be_{\sigma}^{\sT}\projp\big\} \, ,\\
\bQ_{t+1} & = \E\Big\{f(q\tbM_t\projp\be_{\sigma}+\bQ_t^{1/2}\bG) f(q\tbM_t\bfs_{\sigma}+\bQ_t^{1/2}\bG)^{\sT}\Big\}\, ,
\end{align}
which coincide with Eqs. ~(\ref{eq:SE_Block_1}), (\ref{eq:SE_Block_2}), once we notice that $\tbM_t\projp = \tbM_t$ (and drop the tilde from $\tbM_t$). Also, using \myeqref{eq:bx0_svd}, note that $\tbM_0 = q^{-1/2}\bPhi^{\sT} \bV \bS^{\sT} = (\hbx^0)^{\sT} \bx_0/n$, which is the initialization specified in the statement of Theorem \ref{thm:Block}.

Finally, using $\psi(\bx,\by) = \tilde{\psi}(\bx,\bS\by)$ in Eq.~(\ref{eq:Proof_SE_Block}), and recalling that $\bx_{0,i}=\projp\be_{\sigma_i}$, we get
\begin{align}
\lim_{n\to\infty}\left|\frac{1}{n}\sum_{i=1}^n \tilde{\psi}(\bx_i^t,\bx_{0,i}) - \E\big\{\tilde{\psi}(q\tbM_t\be_{\sigma}+\bQ_t^{1/2}\bG,\projp\be_{\sigma})\big\}\right|= 0\, ,
\end{align}
which is the claim of the theorem (after dropping the tildes).

\section{Estimation of rectangular matrices with rank larger than one}
\label{app:RectangularGeneral}

We recall that the data matrix $\bA \in \reals^{n \times d}$ in the rectangular spiked model is 
\begin{align}
\bA = \sum_{i=1}^k \, \lambda_i \, \bu_i\bv_i^{\sT} + \bW\,  \equiv \, \bU \bLambda \bV^{\sT} + \bW \,.  
\label{eq:SpikedDef-Rect1}
\end{align}
Here $(\bu_i)_{1 \leq i \leq k}$ and $(\bv_i)_{1 \leq i \leq k}$  are each sets of non-random orthonormal vectors, with  $\bu_i \in \reals^n$ and  $\bv_i \in \reals^d$.
The noise matrix $\bW$ has entries $(W_{ij})_{i\le n,j\le d}\sim_{iid}\normal(0,1/n)$. We denote by  $s_1 \geq s_2 \ldots \geq s_{\min\{d,n\}} \geq 0$  the  singular values of $\bA$, with $\bphi_1, \ldots, \bphi_{\min\{d,n\}}$ the corresponding (unit-norm) right singular vectors.

For two sequences of functions  $f_t, \, g_t \, :\reals^{q}\times\reals\to \reals^q$, for $t \geq 0$, we consider the AMP algorithm that produces a sequence of iterates $\bu^t \in \reals^n, \, \bv^t \in \reals^d$ according to the following  recursion:
\begin{align}
\bu^{t} & = \bA f_t(\bv^t, \bz) - g_{t-1}(\bu^{t-1},\by)\, \sB_t^{\sT}\, .  \label{eq:AMPut} \\
\bv^{t+1} &  = \bA^{\sT} g_t(\bu^t,\by) - f_t(\bv^t, \bz)\, \sC_t^{\sT}\, .  \label{eq:AMPxt1}
\end{align}
Here $\by\in \reals^n$ and $\bz \in \reals^d$ are fixed vectors, and it is understood that $f_t, g_t$ are applied row-by-row. For example, denoting by $\bu^t_i\in \reals^q$ the
$i$-th row of $\bu^t$, the $i$-th row of $g_t(\bu^{t};\by)$ is given by $g_t(\bu^t_i,y_i)$. 
The `Onsager coefficients' $\sB_t, \, \sC_t \, \in \reals^{q\times q}$  are matrices given by
\begin{align}
\sB_t = \frac{1}{n}\sum_{i=1}^d \frac{\partial f_t}{\partial \bv}(\bv^t_i, z_i)\, ,
\qquad 
\sC_t = \frac{1}{n}\sum_{i=1}^n\frac{\partial g_t}{\partial \bu}(\bu^t_i,y_i)\,
\label{eq:Onsagerrect}
\end{align}
where $\frac{\partial f_t}{\partial \bv}, \, \frac{\partial g_t}{\partial \bu} \in \reals^{q\times q}$ denote the Jacobian matrices of the functions $f_t(\cdot, z), \, g_t(\cdot, y) \, :\reals^q\to\reals^q$, respectively.  The algorithm is initialized with  $\bv^0 \in \reals^{d \times q}$ and $g_{-1}(\bu^{-1},\by) \in  \reals^{n \times q}$ is taken to be the all-zeros matrix.

We will make the following assumptions:
\begin{enumerate}[font={\bfseries},label={(A\arabic*)}]
%
\item As $n, d \to \infty$, the aspect ratio $d(n)/n \to \alpha \in (0,\infty)$. \label{ass:Aspect}
\item\label{ass:Spikes_rect} The values $\lambda_i(n)$ have finite limits as $n\to \infty$, that we denote by $\lambda_i$. 
Furthermore, there are $k_*$ singular values whose limits are larger than 1. That is,  $\lambda_1\ge \dots \lambda_{k_*}>1 \geq \lambda_{k_*+1} \geq \ldots \geq 
\lambda_{\min\{d,n\}} \geq 0$.  We let $S \equiv (1,\dots, k_*)$, and $\bLambda_S$ denote the diagonal matrix with entries $(\bLambda_S)_{ii} = \lambda_i$, $i\in S$.
\item\label{ass:SpikesInit_rect} Setting $q\ge k_*$, we initialize the AMP iteration in Eqs. \eqref{eq:AMPut}--\eqref{eq:AMPxt1} by setting $\bv^0\in\reals^{n\times q}$ equal to the matrix with first $k_*$ ordered columns
 given by $(\sqrt{n}\bphi_i)_{i\in S}$, and $\bzero$ for the remaining $q-k_*$ columns. 
\item\label{ass:SpikesDistr_rect} The joint empirical distribution of the vectors $(\sqrt{n}\bu_\ell(n))_{\ell\in S}$, and $\by$ has a limit
in Wasserstein-$2$ metric. 
Namely, if we let $\tbu_i = (\sqrt{n}u_{\ell,i})_{\ell\in S}\in\reals^{k_*}$, then there exists a random vector $\bU$ taking values
in $\reals^{k_*}$ and a random variable $Y$, with joint law $\mu_{\bU,Y}$, such that
\begin{align}
\frac{1}{n}\sum_{i=1}^n\delta_{\tbu_i,y_i} \towass \mu_{\bU,Y}\, .
\end{align}
Similarly, the joint empirical distribution of the vectors $(\sqrt{d} \, \bv_\ell(n))_{\ell\in S}$, and $\bz$ has a limit in Wasserstein-$2$ metric.    Letting $\tbv_i = (\sqrt{d} v_{\ell,i})_{\ell\in S}\in\reals^{k_*}$, there exists a random vector $\bV \in \reals^{k_*}$  and a random variable $Z$ with joint law $\mu_{\bV, Z}$ such that 
\begin{align}
\frac{1}{d}\sum_{i=1}^d \delta_{\tbv_i, z_i} \towass \mu_{\bV, Z}\, .
\end{align}
\item\label{ass:Lip_rect}  The functions $f_t(\cdot, \cdot), g_t(\,\cdot\,,\,\cdot\,) \, :\reals^q\times\reals\to\reals^q$ are Lipschitz continuous.
\end{enumerate}

\def\obM{\overline{\bM}}
\def\obQ{\overline{\bQ}}

State evolution operates on the matrices $\bM_t, \obM_t \in\reals^{q\times k_*}$, and $\bQ_t, \obQ_t \in\reals^{q\times q}$, with $\bQ_t, \obQ_t \succeq \bzero$. For $t \geq 0$, these matrices are recursively defined as
\begin{align}
\obM_{t} &= \sqrt{\alpha}  \,  \E\Big\{f_t(\bM_t\bV +\bQ_t^{1/2}\bG, \, Z)\bV^{\sT}\big\} \bLambda_S\, , \label{eq:oMt_def_rect} \\
\obQ_{t} & = \alpha \, \E\Big\{f_t(\bM_t\bV +\bQ_t^{1/2}\bG, \, Z) f_t(\bM_t\bV +\bQ_t^{1/2}\bG, \, Z)^{\sT}\Big\}\, , \label{eq:oQt_def_rect} \\
\bM_{t+1} &= \frac{1}{\sqrt{\alpha}} \, \E\Big\{g_t(\obM_t\bU +\obQ_t^{1/2}\bG, \, Y)\bU^{\sT}\big\} \bLambda_S\, , \label{eq:Mt_def_rect} \\
\bQ_{t+1} & = \E\Big\{g_t(\obM_t\bU +\obQ_t^{1/2}\bG, \, Y) g_t(\obM_t\bU +\obQ_t^{1/2}\bG; \, Y)^{\sT}\Big\}\, , \label{eq:Qt_def_rect}
\end{align}
where expectation is taken with respect to $(\bU,Y)\sim\mu_{\bU,Y}$ and $(\bV, Z)\sim\mu_{\bV,Z}$, all of which are independent of $\bG\sim\normal(0,\id_q)$.
These recursions are initialized with $\bM_0, \bQ_0$, which will be specified in the statement of  Theorem \ref{thm:MainRect}  below.

As in Section \ref{sec:Symmetric}, we define  $\cR(\bLambda) \subseteq \reals^{S\times [k]}$ as  the set of orthogonal matrices $\bR$ (with $\bR\bR^{\sT}= \id_{S}$)
such that $R_{ij}=0$ if $\lambda_i\neq\lambda_j$ or if $j\not \in S$.

\begin{theorem}\label{thm:MainRect}
Let $(\bu^t, \bv^t)_{t \geq 0}$ be the AMP iterates generated by algorithm in 
Eqs. \eqref{eq:AMPut}-\eqref{eq:AMPxt1}, under assumptions \ref{ass:Aspect} to \ref{ass:Lip}, for the spiked matrix model in Eq. \eqref{eq:SpikedDef-Rect1}.  Define  $\bP_S \in \reals^{S \times S}$ as
\beq
\bP_S = \diag\left( \frac{1- \alpha \lambda_1^{-4}}{1 + \alpha \lambda_1^{-2}}, \,  \frac{1- \alpha \lambda_2^{-4}}{1 + \alpha \lambda_2^{-2}}, \, \ldots, \frac{1- \alpha \lambda_{k_*}^{-4}}{1 + \alpha \lambda_{k_*}^{-2}} \right).
\eeq 
For $\eta_n\ge n^{-1/2+\eps}$ such that $\eta_n\to 0$ as  $n\to\infty$, define the set of matrices
\begin{align}
\cG_n(\bLambda)\equiv \Big\{\bQ \in \reals^{S\times [k]}:\, \min_{\bR\in\cR(\bLambda)}\|\bQ- \bP_{S}^{1/2}\bR\|_{F}\le \eta_n\Big\}\, ,
\end{align}
Let $\bProd\equiv \bPhi_{S}^{\sT}\bV\in\reals^{k_*\times k}$ where $\bPhi_{S}\in\reals^{n\times k_*}$ is the matrix with columns $(\bphi_i)_{i\in S}$ and $\bV\in \reals^{n\times k}$ is the matrix 
with columns $(\bv_i)_{i\in [k]}$.  Denote by $\bProd_0\in \reals^{S\times S}$ the submatrix corresponding to the  $k_*$ columns of $\bProd$ with index in $S$, and let $\tbProd_0 = (\id-\bProd_0\bProd_0^{\sT})^{1/2}$.

Then, for any pseudo-Lipschitz function  $\psi:\reals^{q+k_*+1}\to \reals$,  $\psi \in\PL(2)$, the following holds almost surely for $t \geq 0$:
\begin{align}
\lim_{n\to\infty}\, \left|\frac{1}{d(n)}\sum_{i=1}^{d(n)} \,  \psi(\bv_i^t,\tbv_i, z_i) - \E\big\{\psi(\bM_t\bV+\bQ_t^{1/2}\bG,\bV,  Z)\big\}\right|= 0\, , \label{eq:SE_result1_rect} \\
\lim_{n\to\infty}\, \left|\frac{1}{n}\sum_{i=1}^n\psi(\bu_i^t,\tbu_i,y_i) - \E\big\{\psi(\obM_t\bU+\obQ_t^{1/2}\bG,\bU,Y)\big\}\right|= 0\, . \label{eq:SE_result2_rect}
\end{align}
Here $\tbv_i = (\sqrt{d} \, v_{\ell,i})_{\ell\in S}\in\reals^{k_*}$, and similarly, $\tbu_i = (\sqrt{n} \, u_{\ell,i})_{\ell\in S}\in\reals^{k_*}$. The expectations are computed with  $(\bV, Z) \sim \mu_{\bV,Z}$ and $(\bU,Y)\sim \mu_{\bU,Y}$, which are each independent of
$\bG\sim\normal(0,\id_q)$. Finally, $(\bM_t,\bQ_t, \obM_t, \obQ_t)$ is the state evolution sequence specified by Eqs. \eqref{eq:oMt_def_rect}--\eqref{eq:Qt_def_rect} with initialization  $(\bM_0)_{[k_*],[k_*]}=\bProd_0$, $(\bM_0)_{[q]\setminus [k_*],[k_*]}=\bzero$,
$(\bQ_0)_{[k_*],[k_*]} =\tbProd_0^2$, and $(\bQ_{0})_{i,j} = 0$ if $(i,j)\not \in [k_*]\times [k_*]$.

Further, $\prob(\bOmega\in \cG_n(\bLambda)) \ge 1-n^{-A}$ for any $A>0$ provided $n>n_0(A)$,  and 
$\bOmega$ converges in distribution to  $\bP_S^{1/2}\bR$, with $\bR$ 
Haar distributed on $\cR(\bLambda)$.
\end{theorem}

\bibliographystyle{amsalpha}
\newcommand{\etalchar}[1]{$^{#1}$}
\providecommand{\bysame}{\leavevmode\hbox to3em{\hrulefill}\thinspace}
\providecommand{\MR}{\relax\ifhmode\unskip\space\fi MR }
\providecommand{\MRhref}[2]{%
  \href{http://www.ams.org/mathscinet-getitem?mr=#1}{#2}
}
\providecommand{\href}[2]{#2}

\addcontentsline{toc}{section}{References}

\end{document}